\documentclass[12pt,preprint]{amsart}

\usepackage{amsmath, latexsym, amsfonts, amssymb, amsthm}
\usepackage{graphicx, color,hyperref,dsfont,epsfig,caption,wrapfig,subfig}
\usepackage{pstricks,yfonts,mathalfa}
\usepackage{mathtools}
\usepackage{movie15}
\hypersetup{
	linktoc=page,
	linkcolor=red,          
	citecolor=blue,        
	filecolor=blue,      
	urlcolor=cyan,
	colorlinks=true           
}

\numberwithin{equation}{section}

\newtheorem{theorem}{Theorem}[section]
\newtheorem{lemma}[theorem]{Lemma}
\newtheorem{proposition}[theorem]{Proposition}
\newtheorem{corollary}[theorem]{Corollary}

\newcounter{thmc}

\newtheorem{thmcite}[thmc]{Theorem}

\theoremstyle{definition}

\renewcommand{\tilde}{\widetilde}          
\DeclareMathSymbol{\leqslant}{\mathalpha}{AMSa}{"36} 
\DeclareMathSymbol{\geqslant}{\mathalpha}{AMSa}{"3E} 
\DeclareMathSymbol{\eset}{\mathalpha}{AMSb}{"3F}     
\renewcommand{\leq}{\;\leqslant\;}                   
\renewcommand{\geq}{\;\geqslant\;}                   

\newcommand{\C}{\mathbb{C}}
\newcommand{\R}{\mathbb{R}}

\newcommand{\N}{\mathbb{N}}
 
\newcommand{\D}{\mathbb{D}}

\newcommand{\E}{\mathds{E}}
\newcommand{\X}{\bm{\mathrm X}} 
\newcommand{\Y}{\bm{\mathrm Y}} 
\newcommand{\M}{\bm{\mathrm M}} 
\newcommand{\V}{\bm{\mathrm V}} 
\newcommand{\B}{\bm{\mathcal B}}

\renewcommand{\P}{\mathds{P}}

\newcommand{\ps}[1]{\langle #1 \rangle}

\newcommand{\eqlaw}{\overset{\text{(law)}}{=}}

\def\eps{\varepsilon}
\def\S{\mathbb{S}}

\def\bi{\begin{itemize}}
	\def\ei{\end{itemize}}

\def\bnum{\begin{enumerate}}
	\def\enum{\end{enumerate}}

\def\<#1{\langle #1 \rangle}

\def\M{\mathbf{M}}

\definecolor{ao(english)}{rgb}{0.0, 0.5, 0.0}
\definecolor{darkcandyapplered}{rgb}{0.64, 0.0, 0.0}
\definecolor{amber}{rgb}{1.0, 0.49, 0.0}

\newcommand{\norm}[1]{\left\lvert#1\right\rvert}
\newcommand{\expect}[1]{\mathbb{E}\left[#1\right]}
\usepackage{bm}
\usepackage{bbm}

\title[Reflection coefficients in Toda CFTs]{Three-point correlation functions in the $\mathfrak{sl}_3$ Toda theory I: Reflection coefficients}

\author{Baptiste Cercl\'e}
\email{baptiste.cercle@universite-paris-saclay.fr}
\address{Laboratoire de Math\'ematiques d'Orsay, B\^atiment 307. Facult\'e des Sciences d'Orsay, Universit\'e Paris-Saclay. F-91405 Orsay Cedex, France}

\begin{document}

	\maketitle
	\begin{abstract}
		Toda Conformal Field Theories (CFTs) form a family of 2d CFTs indexed by semisimple and complex Lie algebras. They are natural generalizations of the Liouville CFT in that they enjoy an enhanced level of symmetry encoded by W-algebras.
		These theories can be rigorously defined using a probabilistic framework that involves the consideration of correlated Gaussian Multiplicative Chaos measures.
		
		This document provides a first step towards the computation of a class of three-point correlation functions, that generalize the celebrated DOZZ formula and whose expressions were predicted in the physics literature by Fateev-Litvinov, within the probabilistic framework associated to the $\mathfrak{sl}_3$ Toda CFT. 
		
		Namely this first article of a two-parts series is dedicated to the probabilistic derivation of the reflection coefficients of general Toda CFTs, which are essential building blocks in the understanding of Toda correlation functions. Along the computations of these reflection coefficients a new path decomposition for diffusion processes in Euclidean spaces, based on a suitable notion of minimum and that generalizes the celebrated one-dimensional result of Williams, will be unveiled. As a byproduct we describe the joint tail expansion of correlated Gaussian Multiplicative Chaos measures together with an asymptotic expansion of class one Whittaker functions. 
	\end{abstract}
	
	
	\section{Introduction}
	
	\subsection{Toda Conformal Field Theories}
	Liouville Conformal Field Theory (CFT hereafter) has drawn considerable attention in the mathematics community over the past few years. First introduced in the physics literature ---as a model for 2d quantum gravity and string theory--- by Polyakov in his 1981 pioneering work~\cite{Pol81}, the mathematical understanding of this theory is much more recent, and providing a rigorous meaning to the framework proposed by Polyakov has now proved to be key in the study of two-dimensional random geometry. This study led to numerous achievements, for instance thanks to the rich interplays between Liouville CFT and (the scaling limit of) Random Planar Maps~\cite{LeG13, Mie13, MS15a, MS16a, MS16b, HoSu19, DDDF,DFGPS, GM20}. Connections between Liouville CFT and Conformal Loop Ensembles have also proved to be particularly thriving~\cite{SW, DMS14}, as exemplified by the recent derivation of the imaginary DOZZ formula in~\cite{AS_DOZZ} which was shown there to describe certain Conformal Loop Ensemble observables.
	
	Closer to the language of Polyakov and the methods developed in the physics literature to study Liouville CFT is the program initiated by David-Kupiainen-Rhodes-Vargas in 2014~\cite{DKRV}. One of their achievements is a rigorous derivation of the DOZZ formula~\cite{KRV_DOZZ}, an essential building block in the understanding at Liouville CFT which they found to agree with predictions made by physicists~\cite{DO94,ZZ96} in the 90's. Later on and together with Guillarmou was unveiled a recursive procedure~\cite{GKRV}, dubbed \emph{conformal bootstrap} and first proposed by Belavin-Polyakov-Zamolodchikov in a groundbreaking work~\cite{BPZ}, that allows to compute the so-called \emph{correlation functions of Vertex Operators} using the DOZZ formula as a fundamental input. Computing such correlation functions ---that is providing explicit formulas for quantities of the form $\ps{\prod_{k=1}^NV_{\alpha_k}(z_k)}$ with the $V_\alpha$ being Vertex Operators--- is a fundamental issue in Liouville CFT (should it be from the physics or mathematics perspective) and the general derivation of these based on Segal's axioms~\cite{Seg04} a few months ago in~\cite{GKRV_Segal} can be seen as the culmination of this program.
	
	Models with an enhanced level of symmetry appeared shortly after the 1984 article of BPZ, with the introduction by Zamoldchikov in 1985~\cite{Za85} of a generalization of the framework from~\cite{BPZ} designed to adapt this machinery to models that enjoy, in addition to conformal invariance, \textit{higher-spin}- (or \textit{W}-) symmetry. This higher level of symmetry is no longer encoded by the Virasoro algebra but rather by $W$-algebras, which are Vertex Operator algebras that contain the Virasoro algebra as a subalgebra. 
	
	Liouville CFT admits a generalization within this setting through Toda CFTs, that form a family of 2d CFTs indexed by semisimple and complex Lie algebras $\mathfrak{g}$ (Liouville CFT then corresponds to $\mathfrak g=\mathfrak{sl}_2$). These Toda CFTs have been extensively studied in the physics literature but remain far from being completely understood, and Toda correlation functions are only known in a few cases~\cite{FaLi1} (apart for the case of Liouville CFT). From a mathematical viewpoint, a rigorous definition has been recently proposed by Rhodes, Vargas and the author in~\cite{Toda_construction} and was shown to recover certain assumptions from the physics literature in~\cite{Toda_OPEWV}, where Huang and the author provided a manifestation of the $W$-symmetry for the $\mathfrak{g}=\mathfrak{sl}_3$ Toda CFT via the existence of so-called \emph{Ward identities}. Proving such identities is key in the computation of certain $\mathfrak{sl}_3$ Toda three-point correlation functions, that represent the analogs of the DOZZ formula within this setting. Using the framework introduced in~\cite{Toda_construction} these three-point correlation functions depend on three \emph{weights} $\alpha_1,\alpha_2,\alpha_3\in\R^2$ and are denoted by
	\[
	C_\gamma(\alpha_1,\alpha_2,\alpha_3)\coloneqq \langle V_{\alpha_1}(0)V_{\alpha_2}(1)V_{\alpha_3}(\infty)\rangle
	\]
	where the right-hand side admits a probabilistic representation involving the consideration of two correlated Gaussian Multiplicative Chaos (GMC in the sequel) measures. Computing such correlation functions can therefore be understood as an integrability statement for the GMC.
	
	The present document represents a major step towards the achievement of this program in that it provides a probabilistic derivation of key objects that naturally arise when considering Toda three-point correlation functions, the so-called \emph{reflection coefficients}. Indeed it is assumed in the physics literature that such coefficients allow to relate Toda Vertex Operators with same quantum weights, by which is meant that there exists a family of transformations $\hat s:\R^2\to\R^2$ indexed by elements $s$ of the Weyl group associated to $\mathfrak{g}$ (see Subsection~\ref{subsection_background} and Section~\ref{sec:whittaker} for more background) such that for $\alpha\in\R^2$
	\begin{equation}\label{eq:refl_rel}
		V_\alpha=R_s(\alpha)V_{\hat s\alpha}
	\end{equation}
	where recall that the $V_\alpha$ are Toda Vertex Operators while $R_s(\alpha)\in\R$ is a reflection coefficient, whose expression has been predicted in the physics literature~\cite{reflection_simplylaced, refl_non_simply_laced, Fat_refl}. In particular this would imply that the probabilistically defined three-point correlation functions $C_\gamma(\alpha_1,\alpha_2,\alpha_3)$ should satisfy a relation of the form
	\begin{equation}
		C_\gamma(\alpha_1,\alpha_2,\alpha_3)=R_s(\alpha_1)C_\gamma(\hat s\alpha_1,\alpha_2,\alpha_3)
	\end{equation}
	for any such $s$. However one subtle issue in the above expression is that for $s\neq Id$, $C_\gamma(\alpha_1,\alpha_2,\alpha_3)$ and $C_\gamma(\hat s\alpha_1,\alpha_2,\alpha_3)$ cannot both make sense probabilistically speaking. In this document we provide a probabilistic representation of certain Toda reflection coefficients and motivate how a meaning can be given to the reflection relation~\eqref{eq:refl_rel} even within this probabilistic framework, in the same fashion as in the probabilistic approach towards Liouville CFT (see~\cite[Lemma 10.5]{KRV_DOZZ}). 
	
	In a follow-up article and building on the results of the present manuscript we will show that some $\mathfrak{sl}_3$ Toda three-point correlation functions, defined thanks to our probabilistic framework, can be explicitly computed and that their values agree with predictions from the physics literature~\cite{FaLi0, FaLi1}. Namely we will prove in~\cite{Toda_correl2} that if $C^{\text{FL}}_\gamma(\alpha_1,\alpha_2,\alpha_3)$ denotes the expression proposed by Fateev-Litvinov in~\cite[Equation (21)]{FaLi1}, then the following holds true:
	\begin{thmcite}[Adapted from Theorem 1.2 in~\cite{Toda_correl2}]\label{thm:main_result}
		For $\gamma\in[1,\sqrt2)$ and $\mathfrak{g}=\mathfrak{sl}_3$, assume that $\alpha_1$ is as prescribed in~\cite[Equation (21)]{FaLi1}.
		Then as soon as $\alpha_0,\alpha_1$ and $\alpha_\infty$ satisfy the Seiberg bounds~\cite[Item (1) of Theorem 1.1]{Toda_construction}:
		\begin{equation}
			C_\gamma(\alpha_1,\alpha_2,\alpha_3) = C^{\text{FL}}_\gamma(\alpha_1,\alpha_2,\alpha_3).
		\end{equation}
	\end{thmcite} 
	In the probabilistic approach towards Toda CFTs the hypothesis that $0<\gamma<\sqrt 2$ corresponds to the optimal range of values for which the probabilistic representation of the correlation functions make sense. It differs from the usual assumption in Liouville CFT that $\gamma<2$ due to the the fact that the longest roots have squared norm $2$. This is purely a matter of conventions, and rescaling $\gamma$ by a multiplicative factor $\sqrt 2$ allows to recover the usual hypothesis that $\gamma\in(0,\sqrt 2)$. The additional assumption in the above statement that $\gamma$ should satisfy $\gamma>1$ comes from another constraint specific to the probabilistic definition of Toda CFTs, that is the fact that certain four-point correlation functions must be well-defined.
	
	\subsection{On a reflection principle: Williams path decomposition and class one Whittaker functions}
	\subsubsection{Williams path decomposition}
	From a completely different perspective, in a celebrated article~\cite{Williams}, Williams in 1974 described a remarkable path decomposition for Brownian paths and more generally one-dimensional diffusions. In its simplest form where the underlying process is a Brownian motion with positive drift $\nu$ (which we denote $B^\nu$), this decomposition can be formulated by saying that, conditionally on the value of the global minimum of the process $\M\coloneqq \inf\limits_{t\geq 0} B^\nu_t$, the law of $B^\nu$ (knowing $\M$) is no longer a Markov process but it can be realized by joining together two Markov processes. Namely, the first process has the law of $B^{-\nu}$ until reaching $\M$, and the second one has the law of the diffusion process $\mathcal B^{\nu}$ whose law is that of $B^{\nu}$ conditioned on staying above $\M$. For general diffusion processes, the decomposition corresponds to realizing the law of a diffusion process conditionally on the value of its maximum by welding together two Markov processes, the latter processes being defined using a Doob $h$-transform.
	
	Following its discovery by Williams, this path decomposition has been thoroughly investigated in the probability community and has inspired numerous fundamental statements such as Pitman's celebrated theorem~\cite{Pit75}. Extending this decomposition for different classes of processes has been a very active field of research~\cite{Ber1, Cha, KM, BY} since Williams' 1974 article. See the account (in French) by Le Gall~\cite{LeGall_Williams} on this topic. 
	However and to the best of our knowledge, providing a general formulation of Williams path decomposition for diffusions in any dimensions is still lacking and remains an open question up to date. 
	
	\subsubsection{Class one Whittaker functions}
	Since their introduction (in their general form) by Jacquet~\cite{J67}, Whittaker functions have proved to be particularly significant in the study of semisimple Lie groups from the point of view of number and representation theory, as well as in the understanding of the quantum Toda lattice~\cite{K77}. Such objects are also relevant from the perspective of probability theory, since they naturally arise in a wide range of probability-related topics. One striking instance of such a manifestation is the connection between Whittaker functions and the geometric Robinson-Schensted-Knuth correspondence~\cite{COcSZ, OcSZ_RSK_Toda}. This connection in turn allows to relate Whittaker functions to models of random polymers in that ($q$-)Whittaker processes, defined by means of Whittaker functions, can be shown via this connection to describe partition functions of certain directed random polymers~\cite{Oc_polymer,BC_McDo}. These models of random polymers themselves turn out to be relevant in the understanding of the stationary Kardar-Parisi-Zhang equation~\cite{BCFV}. Further details on the interplays between Whittaker functions and probability theory can be found \textit{e.g.} in~\cite{Oc_review}.
	
	\subsubsection{Reminders on reflection groups}\label{subsection_background} 
	Reflection groups are naturally associated to class-one Whittaker functions and, as we will see below, provide the natural setting to extend the path decomposition by Williams to a higher-dimensional context. Therefore and before going any further let us first recall some basic notions on reflection groups. More details can be found \emph{e.g.} in the textbook by Humphreys~\cite{humphreys_reflection}.
	
	We consider an Euclidean space $\V$ equipped with a scalar product $\ps{\cdot,\cdot}$ and associated norm $\norm{x}^2\coloneqq \ps{x,x}$. On this Euclidean space we consider a \textbf{finite}\footnote{Reflection groups will be implicitly taken finite in what follows.} reflection group $W$ (\emph{i.e.} the realisation of a Coxeter group), that is a finite subgroup of the general linear group of $\V$ that is generated by finitely many reflections 
	\begin{equation}
		s_\alpha:x\mapsto x-2\frac{\ps{x,\alpha}}{\ps{\alpha,\alpha}}\alpha
	\end{equation}
	for $\alpha\in\V^*\simeq\V$, which are reflections across the hyperplanes $\left\{x\in\V, \ps{x,\alpha}=0\right\}$. Associated to it is \emph{a root system} of the reflection group $(W,\V)$, defined as a finite subset $\Phi$ of $\V\setminus\{0\}$ such that 
	\begin{equation}
		\begin{split}
			&\text{(i) the elements of }\Phi\text{ span }\V,\\
			&\text{(ii) }\Phi\cap\R\alpha=\{\alpha,-\alpha\}\quad\text{for all }\alpha\in\Phi,\\
			&\text{(iii) }s_\alpha\Phi=\Phi\quad\text{for all }\alpha\in\Phi.
		\end{split} 
	\end{equation} Its elements are called the \emph{roots} and are such that $W$ is generated by the $(s_\alpha)_{\alpha\in\Phi}$. For future convenience we denote by $\alpha^\vee\coloneqq\frac{2\alpha}{\ps{\alpha,\alpha}}$ the coroot of a root $\alpha$. Of particular significance are the \emph{simple roots} of $W$, which are linearly independent roots with the additional property that any element of the root system can be written as a linear combination, with coefficients of the same sign, of simple roots. These simple roots will be denoted by $e_1,\cdots,e_r$\footnote{Different choices of simple roots are of course possible, but will be related by conjugation under $W$: for any two simple systems there is a $s\in W$ such that $s e_i=e'_i$ (up to reordering the roots). Such a simple system always exists~\cite[Section 1.3]{humphreys_reflection} and $r=\dim \V$.}. To these simple roots are associated \emph{fundamental weights} $\omega_i$, $1\leq i\leq r$, which are defined so that $\ps{\omega_i,e_j^\vee}=\delta_{ij}$ for all $1\leq i,j\leq r$. Likewise we consider special elements $(\omega_i)_{1\leq i\leq r}$ of $\V$ defined via $\ps{\omega_i^\vee,e_j}=\delta_{ij}$.
	We further introduce the subset $\Phi^+$ of $\Phi$ made of \emph{positive roots}, that is roots of the form \[\alpha^+=\sum_{i=1}^r \lambda_i e_i\quad\text{with }\lambda_i\geq 0.\]
	The \emph{Weyl vector} $\rho$ is then defined as the half-sum of the positive roots:
	\begin{equation}
		\rho\coloneqq\frac12\sum_{\alpha\in\Phi^+}\alpha.    
	\end{equation}
	It can be written $\rho=\sum_{i=1}^r\omega_i$; in particular $\ps{\rho,e_i^\vee}=1$ for all $1\leq i\leq r$.
	
	The hyperplanes orthogonal to the roots divide the space into finitely many connected components, called \emph{Weyl chambers}, on which $W$ acts freely and transitively. We denote by 
	\begin{equation}
		\mathcal{C}\coloneqq \left\{x\in\V, \ps{x,e_i}>0\quad\text{for all }1\leq i\leq r\right\}
	\end{equation}
	the \emph{fundamental Weyl chamber} and set $\mathcal{C}_-\coloneqq -\mathcal{C}$ (which is also a Weyl chamber).  The boundary $\partial\mathcal{C}$ of this chamber is made of $r$ components $\left(\partial\mathcal{C}_i\right)_{1\leq i\leq r}$ which are defined by
	\begin{equation}
		\partial\mathcal{C}_i\coloneqq \left\{x\in\V, \ps{x,e_i}=0\right\}\cap \partial\mathcal{C}
	\end{equation}
	and which we refer to as \emph{walls}. 
	
	A key example of such reflection groups and root systems is the $A_2$ root system. A special attention will be dedicated to it in this document since in that case the underlying vector space $\V$ is two-dimensional and therefore can be thought of as the simplest model that is not one-dimensional. To represent the root system $A_2$ we thus take $\V=\R^2$ equipped with its canonical scalar product and basis.  The simple roots can then be taken to be $e_1=\left(-\frac{\sqrt3}{\sqrt2},\frac{1}{\sqrt2}\right)$ and $e_2=\left(0,-\sqrt2\right)$, while the root system is given by $\{e_1,e_2,\rho,-e_1,-e_2,-\rho\}$ where $\rho=e_1+e_2$. As for the Weyl group $W$, this is the group made of six elements and generated by the reflections $s_1$ and $s_2$ associated to the simple roots $e_1$ and $e_2$. This is illustrated in Figure~\ref{fig:weyl} below. 
	
	\begin{center}
		\includegraphics[height=8cm]{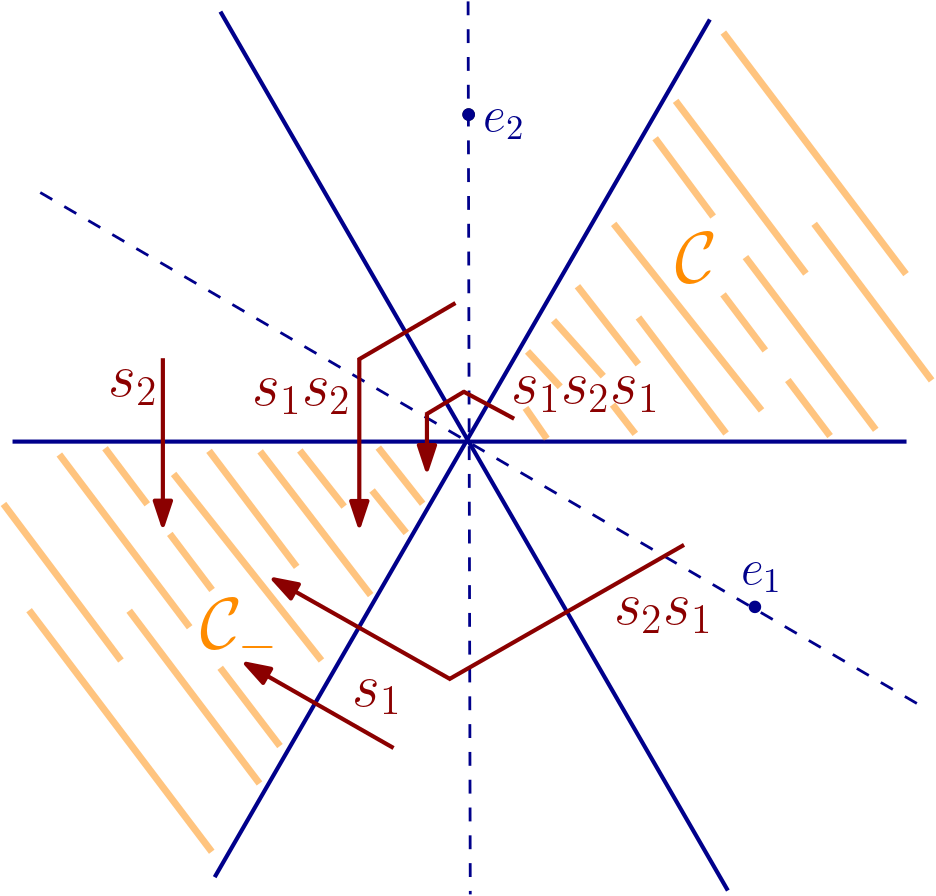}
		\captionof{figure}{The $A_2$ root system}
		\label{fig:weyl}
	\end{center}
	
	Any element $s$ of the reflection group can be written as a composition of reflections orthogonal to the simple roots: $s=s_{i_1}\cdots s_{i_k}$. It admits a reduced expression, that is a product of the above form with a minimal number of simple reflections. This number is the \emph{length} of $s$, $l(s)$, such that $\det(s)=(-1)^{l(s)}$.
	In the sequel we will denote by $\epsilon(s)$ the above quantity ---by analogy with the signature of a group of permutations. This notion of signature is key to derive the general form of the reflection principle~\eqref{refl_brown0} for a certain class of diffusion processes in Euclidean spaces on which a reflection group acts~\cite[Theorem 5]{Gr99} (see Proposition~\ref{prop:reflection_killed} below).
	
	
	\subsection{Presentation of the main results}
	The results presented in this document are twofold. In addition to providing integrability results for Toda CFTs via the computation of certain reflection coefficients, we also state purely probabilistic results which are independent of the setting of Toda CFTs.
	
	\subsubsection{A path decomposition for diffusions in Euclidean spaces}
	Our first result is a generalization of the result of Williams on one-dimensional diffusions~\cite{Williams}. Namely having fixed a reflection group $W$ acting on an Euclidean space $\V$ and introduced a natural (\emph{i.e.} adapted to $W$) analog of the minimum of the process $\M$, we describe the law of a diffusion process conditionally on the value of $\M$. This is done by joining $r+1$ independent Markov processes, in the same spirit as for one-dimensional diffusions. This is the content of Theorem~\ref{thm:williams} whose general scope of application will be presented along Section~\ref{sec:Williams}. We illustrate this statement here by providing such a decomposition for a planar, drifted Brownian motion in the case where the Weyl group is associated to the root system $A_2$ (recall Figure~\ref{fig:weyl}):
	\begin{theorem}\label{thm:brownien_drift}
		Take $\V=\R^2$ and assume that $(W,\V)$ is the reflection group associated to the root system $A_2$ as defined above. For $\nu\in\mathcal{C}$, set 
		\begin{equation*}
			h(x)\coloneqq\sum_{s\in W}\epsilon(s)e^{\ps{s\nu-\nu,x}}\quad\text{and}\quad\partial_i\coloneqq\partial_{\ps{x,e_i}},
		\end{equation*}
		and define a diffusion process $\B^\nu$ on $\V$ started from $x\in\mathcal{C}$ by joining:
		\begin{itemize}
			\item a diffusion process with infinitesimal generator 
			\[
			\frac12\Delta+\nabla \log \partial_{12}h\cdot\nabla,
			\] run until hitting $\partial\mathcal{C}$, say at $z_1\in \partial\mathcal{C}_1$. 
			\item an independent diffusion process started from $z_1$ and with infinitesimal generator 
			\[
			\frac12\Delta+\nabla \log \partial_{2}h\cdot\nabla,
			\] and run until it hits $\partial\mathcal{C}_2$ at a random point $z_2$.
			\item an independent process started from $z_2$, whose law is that of $B^\nu$ conditioned on staying inside $\mathcal{C}$, that is to say the diffusion with generator
			\[
			\frac12\Delta+\nabla \log h\cdot\nabla.
			\]
		\end{itemize}
		Then a planar Brownian motion $B^\nu$ with drift $\nu\in\mathcal{C}$ and started from the origin can be realized by sampling:
		\begin{itemize}
			\item a random variable $\M$ defined by 
			\[
			\P\left(\ps{\M,e_i}\geq \bm m_i\quad\forall1\leq i\leq r\right)=h(-\bm m)\mathds1_{\bm m\in\mathcal{C}_-},
			\]
			and which is such that $\ps{\M,e_i}=\inf_{t\geq 0} \ps{B^\nu_t,e_i}$ for $i=1,2$;
			\item the process $\M+\B^\nu$ described above, where $\B^\nu$ is started from $-\M$.
		\end{itemize}
	\end{theorem}
	See Figures~\ref{fig:planar_BM} and~\ref{fig:planar_BM_decomposed} below that illustrate this decomposition.
	The general formulation of this path decomposition is described in Theorem~\ref{thm:williams}.
	
	We would like to stress that the study of stochastic processes based on root systems has proved to be particularly thriving, for instance via Dunkl processes~\cite{RV_Dunkl} and for the relationships between stochastic processes, Pitman transformations and Littelmann paths~\cite{BBOc}, as well as models for continuous crystals~\cite{BBOc2}. There is also an extensive literature on the topic of Brownian motion in cones and Weyl chambers~\cite{LeGall_cone, Biane_cone, Gr99}.
	
	
	\subsubsection{Vertex Operators: asymptotic expansions and reflection coefficients}
	As stressed in the introduction, Vertex Operators are fundamental in the study of Liouville and Toda CFTs and computing their correlation functions is one of the main issues, for which an extremely powerful method dubbed \emph{conformal bootstrap} has been introduced in the physics literature in~\cite{BPZ}. From a mathematical perspective and in the case of Liouville CFT, the implementation of this machinery ~\cite{GKRV} strongly relies on the fact that these Vertex Operators are (generalized) eigenfunctions of the Liouville Hamiltonian, and (to a lesser extent) on their asymptotic expansion. The asymptotic expansion of the Vertex operators allowed in a subsequent work with Baverez~\cite{BGKRV} to describe the scattering matrix associated to the Liouville Hamiltonian.
	
	Coming back to Toda CFTs, the Toda Vertex Operators can be thought of as functionals over $L^2(\S^1\to\V)$ indexed by weights $\alpha\in\V$. They admit the probabilistic representation 
	\begin{equation}\label{equ:psi_alpha0}
		\psi_{\alpha}(\bm c,\varphi)\coloneqq e^{\ps{\alpha-Q,\bm c}}\E_{\varphi}\left[\exp\left(-\sum_{i=1}^r\mu_ie^{\gamma \ps{\bm c,e_i}}\int_{\mathbb D}\norm{x}^{-\gamma\ps{\alpha,e_i}}M^{\gamma e_i}(d^2x)\right)\right]
	\end{equation}
	where $Q$ is a special vector in $\V$, defined in Equation~\eqref{eq:background_charge} and called the \textit{background charge while} $\left(M^{\gamma e_i}(d^2x)=:e^{\gamma\ps{\X,e_i}}:d^2x\right)_{1\leq i\leq r}$ is a family of GMC measures defined from a vectorial GFF $\X$ with values in $\V$ and $\E_{\varphi}$ is the conditional expectation with respect to $\varphi=\X_{\vert\partial\mathbb{D}}$. Such notions are presented in Section~\ref{sec:whittaker} to which we refer for more details. The process $(B_t)_{t\geq 0}$ defined as the average of the GFF $\X$ on circles centered at the origin and of radii $e^{-t}$ is a Brownian motion over $\V$, so that the decomposition provided by Theorem~\ref{thm:williams} can be applied to (a drift of) this process and is actually the starting point of the derivation of a probabilistic expression for Toda reflection coefficients, thus highlighting a hidden correspondence between the reflection principle in probability and the concept of reflection in (Toda) CFT. 
	Indeed, we prove that Toda reflection coefficients naturally arise in the asymptotic expansion in the $\bm c$ variable of the Toda Vertex Operators, expansion governed by that of the generalized maximum $\M$- of the underlying drifted Brownian motion. We will prove that in the asymptotic where $\bm c\to\infty$ inside $\mathcal{C}_-$ (by which we mean that $\ps{\bm c,e_i}\to-\infty$ for all $1\leq i\leq r$) a remarkable feature of this expansion is that it can be written as a sum over a subset of the Weyl group $W$ in the sense that:
	\begin{theorem}
		For any $\gamma\in(0,\sqrt 2)$, assume that $\alpha\in Q+\mathcal{C}_-$ is sufficiently close to $Q$. Then viewed as functions $\psi_\alpha:\V\times L^2(\S^1\to\V)$, the Toda Vertex Operators admit the expansion
		\begin{equation}\label{equ:expansion_toda}
			\psi_\alpha(\bm c;\varphi)=e^{\ps{\alpha-Q,\bm c}} + \sum_{i=1}^rR_{s_i}(\alpha)e^{\ps{s_i(\alpha-Q),\bm c}}+R_s(\alpha)e^{\ps{s(\alpha-Q),\bm c}}+o\left(e^{\ps{s(\alpha-Q),\bm c}}\right)
		\end{equation}
		as $\bm c\to\infty$ inside $\mathcal{C}_-$ following a certain asymptotic, where $s$ is some element of the Weyl group $W$ with length $2$ while $Q$ is the background charge. The remainder term is such that for any fixed $\varphi$, $\lim\limits_{\bm c\to\infty}e^{-\ps{s(\alpha-Q),\bm c}}o\left(e^{\ps{s(\alpha-Q),\bm c}}\right)\to0$.
	
		The numbers $R_s(\alpha)$ are Toda reflection coefficients and are found to be equal to
		\begin{equation}\label{equ:expression_toda_refl}
			\begin{split}
				R_s(\alpha)&=\epsilon(s)\frac{A\left(s(\alpha-Q)\right)}{A(\alpha-Q)},\quad\text{where}\\
				A(\alpha)&=\prod_{i=1}^r\left(\mu_i\pi l\left(\frac{\gamma^2\ps{e_i,e_i}}{4}\right)\right)^{\frac{\ps{\alpha,\omega_i^\vee}}\gamma}\prod_{e\in\Phi^+}\Gamma\left(1-\frac{\gamma}2\ps{\alpha,e}\right)\Gamma\left(1-\frac{1}\gamma\ps{\alpha,e^\vee}\right).
			\end{split}
		\end{equation}
	\end{theorem}  
	The precise setting in which this statement applies is described in Theorem~\ref{thm:primary_reflection} where in particular we explicit the assumptions made on the weight $\alpha$ as well as on the way $\bm c\to\infty$ inside $\mathcal{C}_-$. Based on similar arguments we are also able to provide an alternative probabilistic description of a more general family of Toda reflection coefficients in Theorem~\ref{thm:expression_refl} (to which we refer for more details), and to show that the expression probabilistically derived is still defined by Equation~\eqref{equ:expression_toda_refl} and agrees with predictions from the physics literature~\cite{reflection_simplylaced, refl_non_simply_laced, Fat_refl}. Note that such quantities do not depend on $\varphi$.
	
	This expansion is much more intriguing than that of Liouville Vertex Operators. Indeed it is far from obvious at first glance why such a specific expansion would hold, and why there are no terms in this asymptotic that would correspond to other directions in $\R^r$ (\emph{e.g.} terms of the form $e^{\ps{\Omega(\alpha-Q),c}}$ for $\Omega$ a rotation), and why they are independent of the variable $\varphi$. 
	However such an expansion should not come as a surprise and is actually reminiscent of the one of spherical functions for the Laplace operator in Weyl chambers, that arise from scattering theory in symmetric spaces (see for instance ~\cite[Theorem 6.2]{MV_sym}) and initiated by the works by Harish-Chandra~\cite{HCa, HCb}. This expansion suggests that the Vertex Operators can be constructed using the Liouville Hamiltonian considered in~\cite{GKRV}, which would be a first step towards implementing the \emph{conformal bootstrap} procedure for Toda CFTs.
	
	Providing such an expansion for Toda Vertex Operators is a key input in the analytic continuation of Toda correlation functions, itself a crucial step in the derivation of Theorem~\ref{thm:main_result}.
	The probabilistic derivation of Toda reflection coefficients will be at the core of the proof of Theorem~\ref{thm:main_result}, in that it will allow to provide an analytic extension of Toda correlation functions beyond the bounds prescribed by~\cite[Theorem 3.1]{Toda_construction}, which is one of the main obstructions at the time being preventing to proving Theorem~\ref{thm:main_result}. For instance, we will see in~\cite{Toda_correl2} that an immediate corollary of the asymptotic expansion of the Vertex Operators is the following extension of the main statement from~\cite{KRV_DOZZ} on the DOZZ formula for the probabilistic Liouville three-point correlation functions $C_{L,\gamma}(\alpha_1,\alpha_2,\alpha_3)$:
	\begin{thmcite}[Corollary 1.3 in~\cite{Toda_correl2}]
		In the setting of Liouville CFT, take $\gamma\in(0,2)$ and assume that $\alpha_1,\alpha_2,\alpha_3\in\R$ are such that $\alpha_k<Q\coloneqq\frac\gamma2+\frac2\gamma$ for $1\leq k\leq 3$. Further assume that $s>-\gamma$ where $s\coloneqq\alpha_1+\alpha_2+\alpha_3-2Q$, and set
		\[
		C_{L,\gamma}(\alpha_1,\alpha_2,\alpha_3)\coloneqq \int_\R e^{sc}\expect{e^{-\mu e^{\gamma c}\rho(\alpha_1,\alpha_2,\alpha_3)}-\mathcal{R}_{\bm\alpha}(c)}dc
		\]
		where $\rho(\alpha_1,\alpha_2,\alpha_3)$ has been introduced in~\cite[Equation (2.17)]{KRV_DOZZ} and is defined using the GMC measure $M^\gamma(d^2x)=:e^{\gamma \X(x)}:d^2x$ by
		\begin{align*}
			\rho(\alpha_1,\alpha_2,\alpha_3)&=\int_\C \frac{\norm{x}_+^{\gamma(\alpha_1+\alpha_2+\alpha_3)}}{\norm{x}^{\gamma\alpha_1}\norm{x-1}^{\gamma\alpha_2}}M^\gamma(d^2x),\quad\text{while}\\
			\mathcal{R}_{\bm\alpha}(c)&\coloneqq \sum_{\mathcal{U}\subset\{1,2,3\}}\mathds{1}_{s<\sum_{k\in\mathcal{U}}2(\alpha_k-Q)}\prod_{k\in\mathcal{U}}R(\alpha_k)e^{2(Q-\alpha_k)c}\quad\text{with}\\
			R(\alpha)&=-\left(\pi\mu l\left(\frac{\gamma^2}{4}\right)\right)^{\frac{2(Q-\alpha)}\gamma}\frac{\Gamma\left(\frac{2(\alpha-Q)}{\gamma}\right)\Gamma\Big(\frac{\gamma}2(\alpha-Q)\Big)}{\Gamma\left(\frac{2(Q-\alpha)}{\gamma}\right)\Gamma\Big(\frac\gamma2(Q-\alpha)\Big)}
		\end{align*}
		the Liouville reflection coefficient. Then $C_{L,\gamma}(\alpha_1,\alpha_2,\alpha_3)$ is analytic in a complex neighbourhood of $\{(\alpha_1,\alpha_2,\alpha_3)\in(-\infty,Q),\quad s>-\gamma\}$. In particular
		\begin{equation}
			C_{L,\gamma}(\alpha_1,\alpha_2,\alpha_3)=C_\gamma^{DOZZ}(\alpha_1,\alpha_2,\alpha_3)
		\end{equation}
		where $C_\gamma^{DOZZ}(\alpha_1,\alpha_2,\alpha_3)$ is given by the DOZZ formula~\cite{DO94,ZZ96}.
	\end{thmcite}
	Such an extension of the range of values for which the probabilistic representation make sense is valid for Toda CFTs too as explained in~\cite{Toda_correl2}, but is less easy to formulate due to the presence of different reflection coefficients. In the above the assumption that $\gamma\in(0,2)$ corresponds to the usual range of definition for probabilistic correlation functions in Liouville CFT. This is a consequence of the fact that the GMC measure considered here is defined from the field $\gamma \X$ while in Toda CFTs they are defined from $\ps{\gamma e_i,\X}$ where $e_i$ has norm $\sqrt 2$.
	
	Note that if $s>-\gamma\vee\max\limits_{1\leq k\leq3}2(\alpha_k-Q)$, the only non-zero term in $\mathcal{R}_{\bm\alpha}(c)$ corresponds to $\mathcal{U}=\emptyset$, so that $\mathcal{R}_{\bm\alpha}(c)=\mathds 1_{s<0}$. The expression proposed above is then given by 
	\[
	C_{L,\gamma}(\alpha_1,\alpha_2,\alpha_3)= \int_\R e^{sc}\expect{e^{-\mu e^{\gamma c}\rho(\alpha_1,\alpha_2,\alpha_3)}-\mathds1_{s<0}}dc
	\]
	and therefore coincides with the probabilistic one from~\cite{KRV_DOZZ} in that case. The main novelty in this statement is the extension of the probabilistic representation to the case where $\max\limits_{1\leq k\leq3}2(\alpha_k-Q)>s>-\gamma$.

	Before moving on, we would like to highlight a remarkable feature in the expression of the reflection coefficients is that it is invariant under a substitution in the parameters which unveils a \emph{duality symmetry} of the model, which manifests itself \emph{e.g.} through the invariance of the reflection coefficients $R_s(\alpha)$ under the transformations
	\begin{equation}
		\gamma \leftrightarrow \frac2\gamma;\quad e_i\leftrightarrow e_i^{\vee}\text{ and } \mu_i\leftrightarrow \mu^\vee_i\coloneqq \frac{\left(\mu_i\pi l\left(\frac{\gamma^2\ps{e_i,e_i}}{4}\right)\right)^{\frac{4}{\gamma^2\ps{e_i,e_i}}}}{\pi l\left(\frac{4}{\gamma^2\ps{e_i,e_i}}\right)}\quad\text{for }1\leq i\leq r.
	\end{equation}
	Put differently and up to an appropriate rescaling of the cosmological constant $\bm \mu=(\mu_1,\cdots,\mu_r)$, the Toda CFT based on the simple Lie algebra $\mathfrak g$ and that based on its Langlands dual ${}^L\mathfrak g$ are related one to the other via the transformation of the coupling constant $\gamma\leftrightarrow\frac2\gamma$. 
	
	\subsubsection{Class-one Whittaker functions and Gaussian Multiplicative Chaos}
	This asymptotic behaviour of the Vertex Operators is a consequence of the tail expansion of correlated GMC measures, which are the building blocks of the probabilistic framework for Toda CFT. We explain in Theorem~\ref{thm:tail_expansion_refl} (to which we refer for more details on the Gaussian Free Fields and GMC measures considered) how Toda reflection coefficients naturally arise in the tail expansion of correlated GMC measures: 
	\begin{theorem}
		Given $\X$ a GFF on $\R^r$ and $(e_i)_{1\leq i\leq r}$ a root system on $\R^r$, define for any $\gamma\in(0,\sqrt 2)$ a family of correlated GMC measures by setting $M^{\gamma e_i}(d^2x)=:e^{\gamma\ps{\X(x),e_i}}:d^2x$ for $1\leq i\leq r$. Then for $\alpha\in Q+\mathcal{C}_-$ close enough to $Q$, as $\bm c\to \infty$ inside $\mathcal{C}$ following a certain asymptotic
		\begin{equation}
			\P\left(\int_{\mathbb D} \norm{x}^{-\gamma\ps{\alpha,e_{i}}}M^{\gamma e_i}(d^2x)>e^{\gamma \ps{\bm c,e_{i}}},\quad i=1,\cdots,r\right) \sim \overline{R_s}(\alpha)e^{\ps{\alpha-\hat s\alpha,\bm c}}
		\end{equation}
		for some $s\in W$ of the form $s_{\sigma 1}\cdots s_{\sigma r}$ ($\sigma$ is a permutation of $\{1,\cdots,r\}$) and where
		\begin{equation}
			\epsilon(s)\prod_{i=1}^r\Gamma\left(1-\frac1\gamma\ps{\hat s\alpha-\alpha,\omega_i^\vee}\right)\overline{R_s}(\alpha)=R_s(\alpha).
		\end{equation} 
	\end{theorem}
	
	The same reasoning gives rise to an asymptotic expansion for class one Whittaker functions. Namely the statement of Theorem~\ref{thm:whittaker_expansion} (additional precisions on the assumptions to be made are detailed there) shows that:
	\begin{theorem}
		For $\mu\in\mathcal{C}$ close enough to zero and  as $\bm x\to\infty$ inside $\mathcal{C}$ following a certain asymptotic, class-one Whittaker functions \[
		\Psi_\mu(\bm x)\coloneqq b(\mu)e^{\ps{\mu, \bm x}}\expect{\exp\left(-\sum_{i=1}^r\frac{\ps{e_i,e_i}}2e^{-\ps{\bm x,e_i}}\int_0^{\infty}e^{-\ps{B^\mu_t, e_i}}dt\right)} \quad \quad\text{satisfy}
		\]
		\begin{equation}
			\Psi_\mu(\bm x)=b(\mu)e^{\ps{\mu,\bm x}}+\sum_{i=1}^r b(s_i\mu)e^{\ps{s_i\mu,\bm x}} + b(s\mu)e^{\ps{s\mu,\bm x}} + o\left(e^{\ps{s\mu,\bm x}}\right)
		\end{equation}
		where $s$ is an element of the Weyl group $W$ while the coefficients $b(\mu)$ are given by
		\begin{equation}
			b(\mu)=\prod_{e\in\Phi^+}\Gamma\left(\ps{\mu,e^\vee}\right).
		\end{equation}
	\end{theorem}
	This expansion is reminiscent of the invariance under the Weyl group of the Whittaker functions, that is to say the fact that $\psi_{s\mu}=\psi_\mu$ for all $s\in W$. It is also consistent with the expansion of class-one Whittaker functions as a sum of fundamental Whittaker functions~\cite[Proposition 4.2]{BOc}.

	\subsection{Overview of the article and perspectives}
	\subsubsection{Organization of the paper}
	In the next section, Section~\ref{sec:Williams}, we will describe the general framework in which our path decomposition applies. To so we provide some reminders on diffusion processes and Doob's conditioning but also highlight its intrinsic connection with the reflection principle in probability by explaining how Theorem~\ref{thm:williams}, our main result in this perspective, can be proved based on this notion.
	
	In the following section a special attention will be dedicated to the properties of such a path decomposition in the special case where the underlying process is a drifted Brownian motion $B^\nu$. Namely we will focus in Section~\ref{section:brown_inf} on the shape of the process defined by $B^\nu$ conditioned on the value of its generalized minimum $\M$ in the asymptotic where $\M\to\infty$ inside $\mathcal{C}_-$.
	
	Based on the description of such a process we will prove in Section~\ref{sec:whittaker} the asymptotic expansions of class one Whittaker functions, Gaussian Multiplicative Chaos measures and Toda Vertex Operators described above. We will see that computing Toda reflection coefficients is made possible thanks to the special behaviour of the process described in Section~\ref{section:brown_inf} and therefore the generalized path decomposition of Theorem~\ref{thm:williams}, thus establishing a connection between the reflection principle in probability and the reflection phenomenon for Toda CFTs.
	
	\subsubsection{Some future directions of work}
	The present article is at the heart of the computation of a family of probabilistic three-point correlation functions associated to the $\mathfrak{sl}_3$ Toda CFT. Indeed we show in~\cite{Toda_correl2} that proving Theorem~\ref{thm:main_result} strongly relies on the probabilistic derivation of Toda reflection coefficients.

    On a similar perspective it would be interesting to investigate whether it is possible to compute reflection coefficients ---as well as structure constants--- associated to Toda CFTs on a boundary. Indeed in the case of Liouville CFT the boundary reflection coefficients are computed in~\cite{RZ} based on a similar reasoning involving the path decomposition by Williams, so that it seems reasonable to expect that the method developed in the present document can be implemented to address the issue of computing reflection coefficients for Toda CFTs on a boundary. To the best of our knowledge there is no proposed expression for such quantities in the physics literature.
	
	
	\subsection*{Acknowledgements} The author acknowledges that this project has received funding from the European Research Council (ERC) under the European Union’s Horizon 2020 research and innovation programme, grant agreement No 725967.
	
	\section{Path decomposition and a reflection principle}\label{sec:Williams}
	This section is dedicated to the description of our generalized path decomposition for diffusion in Euclidean spaces, which we state in Theorem~\ref{thm:williams} below after having introduced the general setting in which it applies ---framework involving in particular the notion of Doob's conditioning. We will then turn to the proof of this statement and to do so we will shed light on the fact that showing this result relies on an ubiquitous concept in probability theory: the reflection principle.
	
	The concept of \emph{reflection principle} in probability as we now it nowadays is often attributed to André in~\cite{Andre}, who allegedly introduced it in order to address the \emph{ballot problem}. However it seems that its geometrical formulation ---which corresponds to the reflection principle as we know it today--- can be traced back to Mirimanoff~\cite{Mir} following a reasoning developed by Aebly~\cite{Aebly}. The interested reader may consult for instance the article~\cite{Renault} for an account on the historical appearance of the reflection principle in probability theory.
	
	There are indeed many ways to formulate this principle. As originally introduced it might be understood as a one-to-one correspondence between paths starting from the origin and whose endpoint is above $0$ and those with endpoints below $0$. This implies that there are twice as many paths starting from $x>0$ and that reach the origin as paths starting from $x$ and whose endpoints are below $0$.  Additional details on the reflection method in probability are exposed for instance in~\cite[Chapter III]{Feller}.
	
	In the context of the Wiener process, the reflection principle states that for a Brownian motion $B$ started from $x>0$, the probability to reach the origin before time $t$ is given by
	\begin{equation}
		\P_x\left(B_s=0\quad \text{for some }0\leq s\leq t\right)=2\P_x\left(B_t\leq0\right).
	\end{equation} 
	This follows from the observation that for a path that reaches $0$, the process $-B$ reflected after having hit the origin has same law as $B$, and either the reflected process or the original one will be above the threshold $0$. A reformulation of this argument leads to the description of the transition probabilities of the Brownian motion killed upon hitting the origin. Namely if we set $T_0\coloneqq\inf\{t\geq0, B_t=0\}$ to be the hitting time of $0$, then for all positive $x$ and $y$
	\begin{equation}\label{refl_brown0}
		\P_x\left(B_t\in dy, t<T_0\right)=\P_x\left(B_t\in dy\right)-\P_{-x}\left(B_t\in dy\right).
	\end{equation}
	This principle admits a natural generalization for diffusion processes on Euclidean spaces~\cite{Gr99}, based on the notion of \emph{reflection group} presented above, and recalled in Proposition~\ref{prop:reflection_killed} below. 

	
	\subsection{A path decomposition}
	This reflection principle is at the core of the path decomposition we propose. Indeed our first application of this reflection principle is a generalization of Williams' celebrated path decomposition for one-dimensional diffusions~\cite{Williams}, which in particular allows to provide a meaning to a process conditioned on its minimal value, which is done by welding two independent diffusion processes before and after having hit the prescribed minimal value. The argument used by Williams in the proof of his main statement is based on the reflection principle~\eqref{refl_brown0} for the one-dimensional Brownian motion. The main input of our method is that this reflection principle can be generalized for Brownian motions in any dimensions by means of the action of a reflection group.
	The explicit form of the process killed upon hitting the boundary of the Weyl chamber is the starting point for a decomposition of Brownian paths in this higher-dimensional context, which is done in terms of the minimums of the process in the directions of the simple roots. 
	
	We formulate our path decomposition as follows: for a certain class of processes $\X$ with values in $\V$, let us introduce the notation (provided it makes sense)
	\begin{equation}
		\M\coloneqq \sum_{i=1}^r \M_i\omega_i \quad \text{with}\quad\M_i\coloneqq\inf_{t\geq0}\ps{\X_t,e_i},
	\end{equation} 
	that generalizes the minimum of a one-dimensional process. Then the process whose law is that of $\X$ conditionally on the value of $\M$ can be realized by joining $r+1$ diffusion processes, and corresponds to a decomposition of the path of $\X$ between times when the successive minimums are attained. The diffusion processes involved are defined using a Doob $h$-transform.
	This statement is illustrated by Figures~\ref{fig:planar_BM} and~\ref{fig:planar_BM_decomposed} below, where we have represented on the left the path of a drifted planar Brownian motion and on the right the decomposition corresponding to the $A_2$ root system viewed in $\V=\R^2$.
	
	\begin{minipage}[c]{0.45\linewidth}
		\includegraphics[height=8cm]{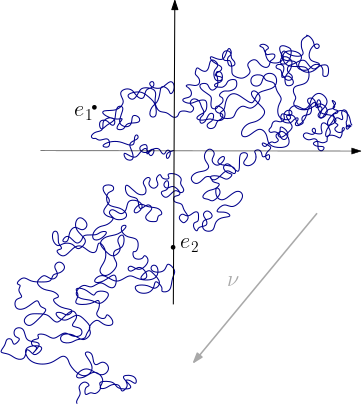}
		\captionof{figure}{Planar Brownian motion with drift $\nu$}
		\label{fig:planar_BM}
	\end{minipage}
	\begin{minipage}[c]{0.45\linewidth}
		\includegraphics[height=8cm]{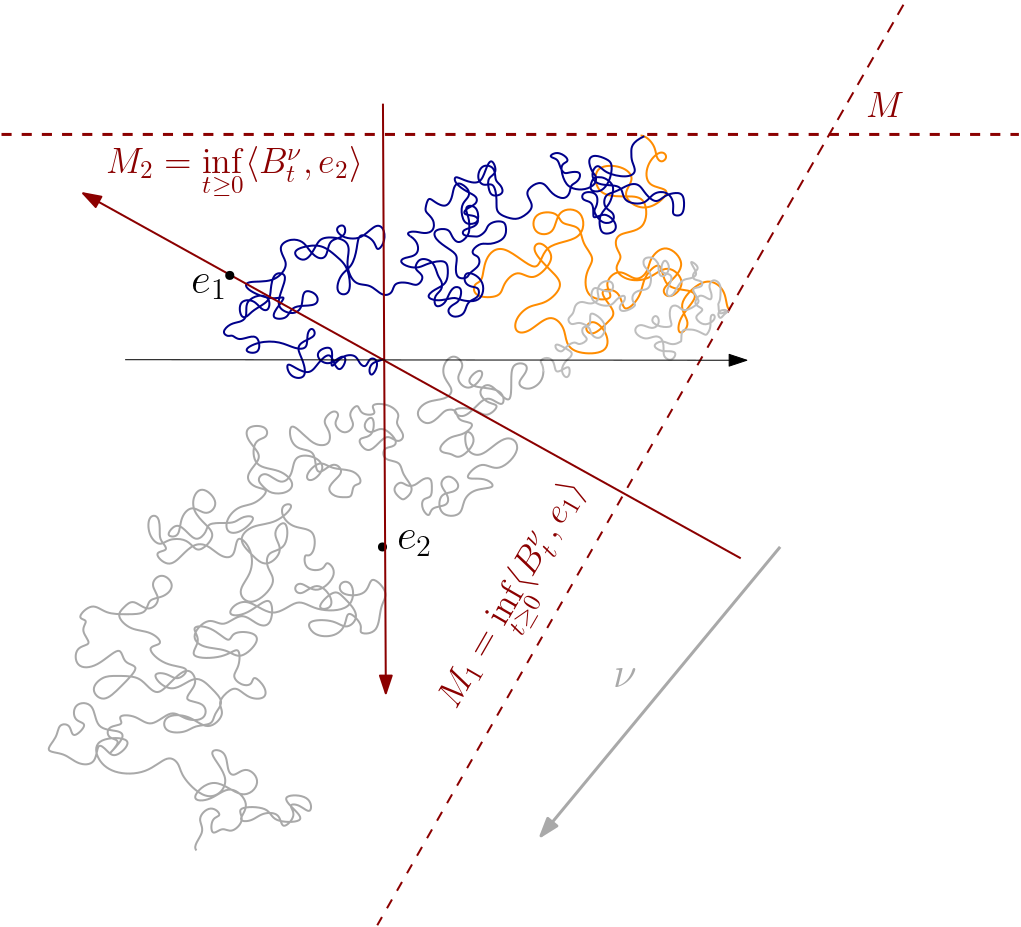}
		\captionof{figure}{Decomposition of the path associated with $A_2$}
		\label{fig:planar_BM_decomposed}
	\end{minipage}

	\subsection{Doob's conditioning and diffusion processes}
	We provide here the necessary notions on conditioning of diffusion processes needed for our purpose, and highlight some requirements needed to be ensured in order to give a meaningful statement. 
	
	\subsubsection{Diffusion processes in Euclidean spaces and Doob's conditioning} 
	The general theory of diffusions is described in great details \emph{e.g.} in the textbooks~\cite{Williams_diff1,Williams_diff2} and~\cite{SV} to which we refer for more details on the definitions and properties of the objects involved. Throughout this document we will consider a (continuous) diffusion process $\X$ with state space $\V$ and transition semigroup $\bm p_t$. We denote its infinitesimal generator by $\mathcal{A}$, which is a second-order differential operator that we assume to be of the form
	\begin{equation}
		\mathcal{A}f(x)=\sum_{i,j=1}^{\text{dim}\V}a_{i,j}(t,x)\partial_{x_i}\partial_{x_j}f(x)+\sum_{i=1}^{\text{dim}\V}b_i(t,x)\partial_{x_i}f(x),
	\end{equation}
	with $a$, $b$ bounded in $(t,x)$ and Lipschitz in their spatial variable $x$. Put differently the killing measure of the process is zero.
	
	Next we recall basic notions about $h$-transforms for our process $\X$ needed for our purpose; we refer \emph{e.g.} to~\cite[Chapter 11]{CW05} for additional details and justifications. Let $h$ be a $C^2$ function that is $\mathcal{A}$-harmonic on $\V$.\footnote{To define the process this hypothesis may be relaxed by taking $h$ excessive, by which is meant that for any $x$ in $\V$ and all time $t$, $\E_x\left[h(\X_t)\right]\leq h(x)$ with $\lim\limits_{t\rightarrow0}\E_x\left[h(\X_t)\right]= h(x)$.} Let us denote $\V_h\coloneqq\left\{x\in\V,\quad 0<h(x)<\infty\right\}$ and $\bm p_t^h(x,y)$ the transition probabilities of the process $\X$ killed upon exiting $\V_h$.
	Then the Doob $h$-transform of $\X$ is the continuous Markov process $\Y$ with transition probabilities given by
	\begin{equation}
		\begin{split}
			\mathfrak{p}_t(x,dy)\coloneqq &\frac{h(y)}{h(x)}\bm p_t^h(x,dy)\quad\text{for $x$ inside $\V_h$}\\
			=&0\quad\text{otherwise}.
		\end{split}
	\end{equation}
	It is a diffusion process with generator 
	\begin{equation}
		\mathcal{A}^h=\mathcal{A}+\sum_{i,j=1}^{\dim\V}\left(\frac{a_{ij}\partial_{x_i}h}h\right)\partial_{x_j}.
	\end{equation}
	It follows from~\cite[Remark 11.4]{CW05} that such a process $\Y$ started inside $\V_h$ will, almost surely, never hit the boundary of $\V_h$. In the special case where $h$ is given by the probability that the process $\X$, started from $x$, never exits a domain of the form $\M+\mathcal{C}$ for some $\M\in\V$, then the process $\Y$ has the law of $\X$ conditioned to stay inside $\M+\mathcal{C}$ at all time.
	
	More generally, one can consider $h_t$ positive and such that 
	\begin{equation}
		\left(\partial_t +\mathcal{A}\right)h_t=0\quad\text{on }\V_h.
	\end{equation}
	The latter implies that the process $h_t(\X_t)$ is a (local) positive martingale. Moreover the $h$-transform of the process $\X$ can still be defined as above. This is done by considering the diffusion process $\Y$ with infinitesimal generator 
	\begin{equation}
		\mathcal{A}^h\coloneqq\mathcal{A}+\sum_{i,j=1}^{\dim\V}\left(\frac{a_{ij}\partial_{x_i}h_t}{h_t}\right)\partial_{x_j}.
	\end{equation}
	This process is such that 
	\[
	d\mathbb P_{\Y}\vert_{\mathcal{F}_t}=\frac{h_t(\X_t)}{h_0(\X_0)}d\mathbb P_{\X}\vert_{\mathcal{F}_t}
	\]
	where $\left(\mathcal{F}_t\right)_{t\geq0}$ is the standard filtration of the processes, and therefore $\Y$ has transition semigroup
	\begin{equation}
		\begin{split}
			\mathfrak{p}_t(x,dy)\coloneqq &\frac{h_t(y)}{h_0(x)}\bm p_t^h(x,dy).
		\end{split}
	\end{equation}
	
	\subsubsection{Exit time from a Weyl chamber and reflection principle}
	Assume that a diffusion process $\X$ as defined above has almost surely a minimum in the direction of the simple roots, by which we mean that almost surely, for any $1\leq i\leq r$ the quantities 
	\begin{equation}\label{equ:finite_M}
		\M_i\coloneqq \inf\limits_{t\geq 0} \ps{\X_t,e_i}
	\end{equation} are finite. 
	The law of this minimum can be computed by noticing that having $\M_i(\X)\geq m_i$ for all $1\leq i\leq r$ means that $\X$ never hits $\bm m + \partial\mathcal{C}$ with $\bm m=\sum_{i=1}^rm_i\omega_i^\vee$. As a consequence 
	\begin{equation}\label{equ:def_h}
		\P_x\left(\forall\text{ }1\leq i\leq r, \text{ }\M_i(\X)\geq m_i\right)=h^{\bm{m}}(x),
	\end{equation}
	where $h^{\bm{m}}$ is a solution to the Dirichlet problem
	\begin{equation}
		\left\{
		\begin{split}
			\mathcal{A}h^{\bm{m}}&=0\quad\text{in }\bm m+\mathcal{C}\\
			h^{\bm{m}}&=0\quad\text{on }\bm m+\partial\mathcal{C}
		\end{split}
		\right.
	\end{equation}
	with the asymptotic $h^{\bm{m}}(x)\rightarrow1$ as $x\rightarrow\infty$ along a ray inside the Weyl chamber. 
	
	We assume that the process $\X$ is such that for any $\bm m\in\mathcal C_-$, such maps $h^{\bm m}$ are of $C^r$ regularity over $\bm m+\mathcal{C}$. In particular for any $1\leq k\leq r$ and $i_1,\cdots,i_k$ distinct in $\left\{1,\cdots,r\right\}$, the quantities defined by 
	\begin{equation}\label{equ:assumption1}
		\partial_{i_1,\cdots,i_k}h^{\bm m}\coloneqq (-1)^k\partial_{\bm m_{i_1}}\cdots\partial_{\bm m_{i_k}}h^{\bm m}
	\end{equation}
	are well-defined in $\bm m+\mathcal{C}$. These maps satisfy 
	\begin{equation}
		\left\{
		\begin{split}
			\mathcal{A}\partial_{i_1,\cdots,i_k}h^{\bm m}&=0\quad\text{in }\bm m+\mathcal{C}\\
			\partial_{i_1,\cdots,i_k}h^{\bm m}&=0\quad\text{on }\bm m+\partial\mathcal{C}\setminus\left(\partial\mathcal{C}_{i_1}\cup\cdots\cup\partial\mathcal{C}_{i_k}\right).
		\end{split}
		\right.
	\end{equation}
	It is also non-negative inside $\bm m+\mathcal{C}$: indeed note that $\partial_{1,\cdots,k}h^{\bm m}$ is given by
	\begin{equation}\label{equ:derivatives_h}
		\begin{split}
			\lim\limits_{\eps_1,\cdots,\eps_r\rightarrow0^+}\frac{1}{\eps_1\cdots\eps_k}\P_x\big(\forall\text{ }1\leq i\leq k, \text{ }m_i\leq \M_i(\Y)&\leq m_i+\eps e_i;\\
			\M_i(\Y)&\geq m_i\text{ for }k+1\leq i\leq r\big).
		\end{split}
	\end{equation} We also stress that $\partial_{1,\cdots,r}h^\M$ represents the probability density function of the random variable $\M$.
	
	In this context it is very natural to introduce the process whose law is that of $\X$ killed when exiting the Weyl chamber $\M+\mathcal{C}$. For this purpose we introduce the reflection group $W^\M$ generated by the reflections centered at $\M$, that is \begin{equation}
		s_i(x)=x-\ps{x-\M,e_i^\vee}e_i,
	\end{equation} and define accordingly
	\begin{equation}\label{equ:stop_proc}
		\bm p^\M_t(x,y)\coloneqq\sum_{s\in W^\M} \epsilon(s)\bm p_t(sx,y)
	\end{equation}
	which vanishes on the boundary of $\M+\mathcal{C}$. 
	For diffusion processes, the reflection principle takes the following form, which generalizes the well-known property of the one-dimensional Brownian motion.
	\begin{proposition}\label{prop:reflection_killed}
		Let $\M\in\V$ and assume that the law of a diffusion process $\X$ is invariant under the action of the reflection group centered at $\M$:
		\begin{equation}
			\forall\text{ }s\in W^\M, \quad s\X\eqlaw\X.
		\end{equation}
		Then the transition probabilities of the process $\X$ killed when exiting the Weyl chamber $\M+\mathcal{C}$ are given by $\bm p^\M_t(x,y)$.
	\end{proposition} 
	\begin{proof}
		For the sake of completeness, we reproduce the argument of~\cite[Theorem 1]{GZ92} and~\cite[Theorem 5]{Gr99}. Let us consider all paths from $sx$ to $y$ where $s$ ranges over $W^\M$. Then to any path which hits the boundary of a Weyl chamber, we can associate another path obtained by reflecting, across the boundary component being hit, the initial path before having reached the boundary of the Weyl chamber. The assumption made on the process shows that both have same probability to occur, so these probabilities cancel out in the alternating sum~\eqref{equ:stop_proc}. The only remaining term is given by the probability that a path never crosses the boundary of a Weyl chamber, or put differently that the process goes from $x$ to $y$ without exiting the Weyl chamber.
	\end{proof}  
	This formula can also be found \emph{e.g.} in~\cite[Equation (5.1)]{BBOc} in the case where the underlying process is a Brownian motion. The assumption made on the process implies that it is reflectable in the sense of~\cite{Gr99}.

	\subsubsection{The objects considered}\label{sub:assumptions}
	Let us consider a diffusion process $\X_0$ such that 
    \begin{equation}\label{eq:hyp0}
        \text{For any $\M\in\V$ and $s\in W^\M$, we have the identity in law $s\X_0\eqlaw\X_0$}
    \end{equation} so that the assumptions of Proposition~\ref{prop:reflection_killed} are satisfied. We denote its semigroup $\bm p_0(t;x,dy)$ and its generator $\mathcal{A}_0$.  The path decomposition applies to a diffusion process $\X$ with generator of the form
	\begin{equation}\label{eq:hyp1}
        \begin{split}
		  &\mathcal{A}=\mathcal{A}_0+\sum_{i,j=1}^{\dim\V}\left(\frac{a_{ij}\partial_{x_i}f_t}{f_t}\right)\partial_{x_j},\quad\text{where $f_t(x)$ is positive and satisfies}\\
        &\left(\partial_t +\mathcal{A}_0\right)f_t(x)=0\quad\text{on }\V.
        \end{split}
	\end{equation}
    We require this process to satisfy the additional assumption that 
    \begin{equation}\label{eq:hyp2}
        \begin{split}
            &\text{(i) For all }1\leq i\leq r,\quad \M_i\coloneqq\inf\limits_{t\geq 0}\ps{\X_t,e_i}\quad\text{is finite almost surely;}\\
            &\text{(ii) For any $\bm m\in\V$ the map $h^{\bm m}$ introduced in Equation~\eqref{equ:def_h} is $C^r$ over $\bm m+\mathcal{C}$.} 
        \end{split}
    \end{equation}
    We denote by $\mathcal{M}$ the law of the random variable 
	\[
    	\M=\sum_{i=1}^r\M_i\omega_i^\vee
	\] corresponding to the minimums of the process, and which can be described in terms of the map $h^{\bm m}$.
	By doing so the transition semigroup of $\X$ is given by $\bm p_t(x,dy)=\bm p_0(t;x,y)\frac{f_t(y)}{f_0(x)}$, while the transition probabilities of the process $\X$ killed upon hitting the boundary of $\M+\mathcal{C}$ are
	\begin{equation}
		\bm p^\M_t(x,y)\coloneqq\frac{f_t(y)}{f_0(x)}\sum_{s\in W^\M} \epsilon(s)\bm p_0(t;sx,y).
	\end{equation}
	
	Having chosen such a process $\X$, we can define, for any $\M$ in $\V$, the transition probabilities $\mathfrak{p}$\footnote{We omit the dependence in $\M$ for such a process in order to keep the notations concise.} of the process $\X$ conditioned to stay in the Weyl chamber $\M+\mathcal{C}$, defined for any $x,y$ inside it by
	\begin{equation}\label{def:pdf_conditioned}
		\mathfrak{p}_t(x,y)\coloneqq \frac{h^\M(y)}{h^\M(x)}\bm p^\M_t(x,y).
	\end{equation}
	Here $h^\M$ is defined by Equation~\eqref{equ:def_h}.
	We further introduce, for any $1\leq k\leq r$ and $i_1,\cdots,i_k$ distinct in $\left\{1,\cdots,r\right\}$, 
	\begin{equation}\label{def:pdf_conditioned_bis}
		\mathfrak{p}_t^{i_1,\cdots,i_k}(x,y)\coloneqq \frac{\partial_{i_1,\cdots,i_k}h^\M(y)}{\partial_{i_1,\cdots,i_k}h^\M(x)}\bm p^\M_t(x,y)
	\end{equation}
	for $x$ such that $\partial_{i_1,\cdots,i_k}h^\M(x)>0$ and set it to zero otherwise. These represent the transition probabilities of a diffusion process $\X^{i_1,\cdots,i_k}$ with generator 
	\begin{equation}
		\mathcal{A}^{i_1,\cdots,i_k}\coloneqq\mathcal{A}^{\partial{i_1,\cdots,i_k}h}=\mathcal{A}+\sum_{i,j=1}^{\dim\V}\left(a_{ij}\frac{\partial_{x_i}\partial{i_1,\cdots,i_k}h}{\partial{i_1,\cdots,i_k}h}\right)\partial_{x_j},
	\end{equation}
	killed upon hitting the boundary of $\M+\mathcal{C}$. 
	We have already highlighted above that for any $x\in\M+\mathcal{C}$, the process $\X^{i_1,\cdots,i_k}$ started from $x$ will almost surely never reach $\M+\partial\mathcal{C}$ on $\partial\mathcal{C}\setminus\left(\partial\mathcal{C}_{i_1}\cup\cdots\cup\partial\mathcal{C}_{i_k}\right)$. A consequence of our main result below is that such a process will almost surely reach $\M+\partial\mathcal{C}$ at $\partial\mathcal{C}_{i_1}\cup\cdots\cup\partial\mathcal{C}_{i_k}$ (and will be killed upon hitting).
	
	\subsection{The generalized path decomposition}
	We are now in position to provide a precise statement for the path decomposition of the process $\X$. In our next statement we assume that this process meets the requirements of Subsection~\ref{sub:assumptions} and adopt the notations introduced there. In a similar fashion as in the one-dimensional case~\cite{Williams}, we see that the path can be divided between different parts corresponding to times where the process reaches its minimums. 
	\begin{theorem}\label{thm:williams}
        Let $\X$ be a diffusion process satisfying the assumptions prescribed by Equations~\eqref{eq:hyp0},~\eqref{eq:hyp} and~\eqref{eq:hyp2}.
		Pick $\M$ according to its marginal law $\mathcal{M}$ and define a process $\Y$ to be the joining of the following processes:
		\begin{itemize}
			\item Start by sampling a diffusion process $\Y^1$ started from the origin, with generator $\mathcal A^{1,\cdots,r}$, and run it until hitting $\M+\partial\mathcal{C}$, say at $z_1\in \M+\partial\mathcal{C}_1$. 
			\item Then run an independent process $\Y^2$ started from $z_1$ and with generator $\mathcal A^{2,\cdots,r}$, upon hitting $\M+\partial\mathcal{C}$.
			\item Thus define a family of processes $(\Y^1,\cdots,\Y^r)$. When $\Y^r$ reaches the boundary of $\M+\partial\mathcal{C}$, sample $\Y^{r+1}$ with generator $\mathcal{A}^{\emptyset}$.
		\end{itemize}
		Then $\Y$ has the law of $\X$, the diffusion process started from the origin and with generator $\mathcal{A}$.
	\end{theorem}
	In other words, the statement above shows that the law of the process $\X$, conditionally on the value of its minimum $\M$, is the joining of $r+1$ processes. It should be noted that the process $\Y^{r+1}$ has the law of $\X$ conditioned to stay in the Weyl chamber $\M+\mathcal{C}$. One easily checks that this statement is indeed consistent with the one-dimensional result by Williams~\cite{Williams}.
	
	In the planar case (that is when $\V=\R^2$), Theorem~\ref{thm:williams} allows to provide a path decomposition with respect to cones with angles $\frac{\pi}{n}$, $n\geq 1$, that correspond to the Weyl chambers associated to the dihedral groups $\mathcal{D}_n$, $n\geq 1$. This is illustrated by Figure~\ref{fig:planar_BM_decomposed_dihedral} below, where we represent the decomposition, associated to $\mathcal{D}_n$ for different values of $n$, of the path of a Brownian motion with drift $\nu$ .
	
	\includegraphics[scale=0.45]{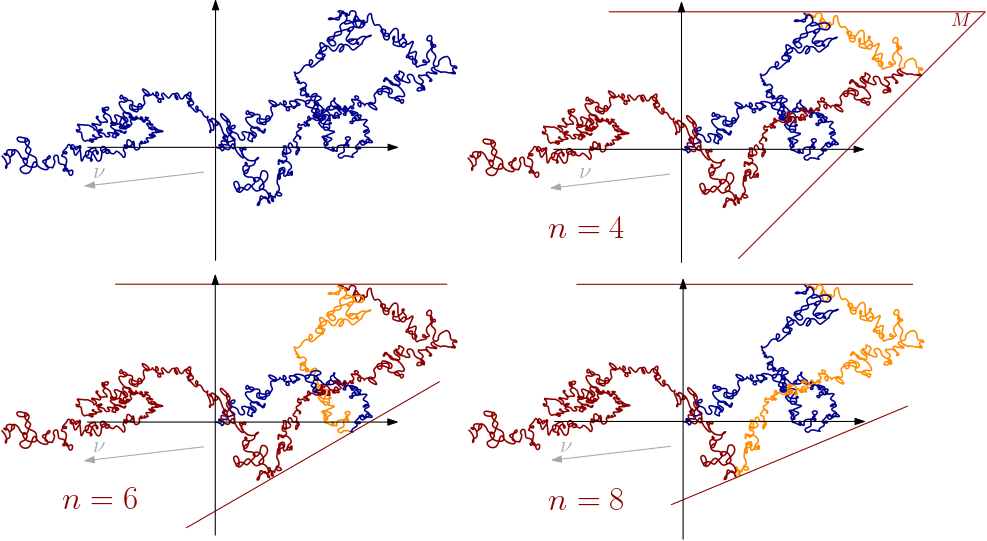}
	\captionof{figure}{Planar Brownian motion with drift $\nu$ and its path decompositions associated to $\mathcal{D}_n$, $n=4,6,8$.}
	\label{fig:planar_BM_decomposed_dihedral}

	As a simple illustration of Theorem~\ref{thm:williams}, we derive the following representation for planar Brownian motions with drift $\nu\in\mathcal{C}$, $B^\nu$, which we have depicted in Figures~\ref{fig:planar_BM} and~\ref{fig:planar_BM_decomposed}. As explained in the introduction, the $A_2$ root system can be realised in $\V=\R^2$ equipped with its standard scalar product and by considering the simple roots given, in the canonical basis, by taking $e_1=\left(-\frac{\sqrt3}{\sqrt2},\frac{1}{\sqrt2}\right)$ and $e_2=\left(0,-\sqrt2\right)$. The Weyl group, generated by the reflections $s_1$ and $s_2$, is then made of six elements. The $h$-function that enters the decomposition of the process is given by the map $h:\mathcal{C}\to\R^+$ defined by
	\begin{equation}\label{equ:max_drift}
		h(x)=\sum_{s\in W}\epsilon(s)e^{\ps{s\nu-\nu,x}}
	\end{equation}
	and which corresponds to the probability that $B^\nu$ started from $x$ never exits the Weyl chamber:
	\begin{proposition}\label{prop:max_drift}
		The minimums of the drifted Brownian motion $B^\nu$ in the direction of the simple roots are described by
		\begin{equation}
			\P\left(\M_i(B^\nu)\geq \bm m_i\quad\forall1\leq i\leq r\right)=h(-\bm m)\mathds1_{-\bm m\in\mathcal{C}}
		\end{equation}
	\end{proposition}
	This is proved for instance in~\cite[Section 5]{BBOc} and holds for any finite reflection group. The transition probabilities of the drifted Brownian motion are given by
	\begin{equation}
		p_\nu(t;x,y)\coloneqq p_t(x,y)\exp\left(\ps{\nu,y-x}-\frac{\norm{\nu}^2}{2}t\right)
	\end{equation}
	with $p_t$ the transition probabilities of a standard Brownian motion.
	\begin{corollary}\label{cor:brownien_drift}
		Let $\M$ be distributed according to 
		\[
		\P\left(\M_i\geq \bm m_i\quad\forall1\leq i\leq r\right)=h(-\bm m)\mathds1_{-\bm m\in\mathcal{C}},
		\] and sample three processes by: 
		\begin{itemize}
			\item Starting from the origin with a process $\X^1$ with transition probabilities $\frac{\partial_{12}h(y-\M)}{\partial_{12}h(x-\M)}p^\M_\nu(t;x,y)$ run until hitting $\M+\partial\mathcal{C}$, say at $z_1\in \M+\partial\mathcal{C}_1$. 
			\item Running an independent process $\X^2$ started from $z_1$, with transition probabilities $\frac{\partial_{2}h(y-\M)}{\partial_{2}h(x-\M)}p^\M_\nu(t;x,y)$, upon hitting $\M+\partial\mathcal{C}_2$ at $z_2$.
			\item Running an independent process $\X^3$ started from $z_2$, with transition probabilities $\frac{h(y-\M)}{h(x-\M)}p^\M_\nu(t;x,y)$.
		\end{itemize}
		Then the joining of these three paths has the law of $B^\nu$ started from the origin.
	\end{corollary}
	By means of the function $h$ defined by Equation~\eqref{equ:max_drift} and using Theorem~\ref{thm:williams}, this statement has a natural extension to drifted Brownian motion evolving on an Euclidean space on which acts a reflection group. 
	In Section~\ref{section:brown_inf}, we study in more details the process thus decomposed. In particular we prove that in the asymptotic where $\M\rightarrow\infty$ along a ray inside $\mathcal{C}$, the process can essentially be realised by joining drifted Brownian motions that will reflect on the different walls of the Weyl chamber. For instance in the planar case, we show that asymptotically the process looks like a Brownian motion with drift $s_1s_2\nu$ (resp. $s_2s_1\nu$) until it reaches the boundary of $\mathcal{C}$ on $\partial\mathcal{C}_1$ (resp. $\partial\mathcal{C}_2$), where it will be reflected and behaves like a Brownian motion with drift $s_2\nu$ (resp. $s_1\nu$), run  until it reaches the boundary of $\mathcal{C}$ on $\partial\mathcal{C}_2$ (resp. $\partial\mathcal{C}_1$). This process will again be reflected and look like a Brownian motion with drift $\nu$ conditioned to stay inside $\mathcal{C}$. The initial drift depends on the way $\M\rightarrow\infty$ inside $\mathcal{C}$. This asymptotic is depicted in Figure~\ref{fig:regimes_brown}.
	
	There is also a counterpart result for the analog of \lq\lq Bessel processes\rq\rq in the context of reflection groups. Thanks to its interplays with Littelmann paths~\cite{BBOc} and models for continuous crystals~\cite{BBOc2}, this process has been extensively studied in the literature. This process is defined using the $h$-transform of a Brownian motion made via the function
	\begin{equation}
		h(x)=\prod_{\alpha\in\Phi^+}\ps{x,\alpha},
	\end{equation}
	which is, up to a multiplicative factor, the only harmonic map inside $\mathcal{C}$ that vanishes on $\partial\mathcal{C}$~\cite[Section 5]{BBOc}. By doing so the process with transition probabilities (with $\M=\bm 0$)
	\[
	\frac{h(y)}{h(x)}p^{\bm 0}_t(x,y)
	\]
	has the law of a Brownian motion conditioned to stay inside the Weyl chamber, and the decomposition of Theorem~\ref{thm:williams} applies with such a function $h$ and $p_t$ the transition semigroup of a Brownian motion on $\V$.
	
	\subsection{Proof of Theorem~\ref{thm:williams}}
	The purpose of this subsection is to establish the validity of Theorem~\ref{thm:williams}. A simple observation that we explain below allows to reduce its proof to the following lemma:
	\begin{lemma}\label{lemma:williams_crucial}
		For any $\M=\sum_{i=1}^r\M_i\omega_i\in\V$ with $m_i\geq0$ for all $1\leq i\leq r$, define a process $\X$ as in Theorem~\ref{thm:williams}. Then for any $0\leq i\leq r$ and $x,y$ inside $\M+\mathcal{C}$, 
		\begin{equation}
			\P_x\left(\X_t\in dy, \X\text{ has hit $\partial\mathcal{C}$ at }\partial \mathcal C_1,\cdots,\partial \mathcal C_i\right)=\frac{\partial_{i+1,\cdots,r}h^\M(y)}{\partial_{1,\cdots,r}h^\M(x)}\partial_{1,i}\bm p_t^\M(x,dy)
		\end{equation}
	\end{lemma}
	Of course a similar result also holds when the process has hit different parts of the boundary.
	
	Before proving this claim, let us explain to what extent Theorem~\ref{thm:williams} boils down to this statement. Pick any $x,y\in\M+\mathcal{C}$ and some positive time $t$. Then, provided that Lemma~\ref{lemma:williams_crucial} holds, we can write that
	\begin{align*}
		\partial_{1,\cdots,r}\left(h^\M(y)\bm p_t^\M(x,dy)\right)&=\sum_{k=0}^r \sum_{\substack{i_1,\cdots,i_k\\\text{distinct}}}\partial_{1,\cdots,r\setminus i_1,\cdots, i_k}h^\M(y)\partial_{i_1,\cdots, i_k}\bm p_t^\M(x,dy)\\
		&=\partial_{1,\cdots,r}h^\M(x)\sum_{k=0}^r \sum_{\substack{i_1,\cdots,i_k\\\text{distinct}}}\P_x\left(\X_t\in dy, \X\text{ has hit }\partial \mathcal C_{i_1},\cdots,\partial \mathcal C_{i_k}\right)\\
		&=\partial_{1,\cdots,r}h^\M(x)\P_x\left(\X_t\in dy\right).
	\end{align*}
	On the other hand, one has the property that 
	\begin{equation}\label{equ:Doob_toprove}
		\partial_{1,\cdots,r}\Big[h^{\bm M}(y)\bm p_t^\M(x,dy)\Big]=\partial_{1,\cdots,r}h^\M(x)\P_x\left(\Y_t\in dy\vert \M\right),
	\end{equation}
	where $\Y$ is the diffusion process with generator $\mathcal{A}$.
	Indeed, the event that for all $1\leq i\leq r$, $\M_i(\Y)\geq m_i$ corresponds to the process $\Y$ never hitting $\bm m + \partial\mathcal{C}$ with $\bm m=\sum_{i=1}^rm_i\omega_i$. Therefore by Doob's conditioning
	\[
	\P_x\left(\Y_t\in dy\vert \forall1\leq i\leq r, \M_i\geq m_i\right)=\frac{h^{\bm m}(y)}{h^{\bm m}(x)}\bm p_t^{\bm m}(x,dy),
	\] 
	so Equation~\eqref{equ:Doob_toprove} reduces to 
	\begin{align*}
		&\partial_{1,\cdots,r}\Big[h^{\bm m}(x)\P_x\left(\Y_t\in dy\vert \forall1\leq i\leq r, \M_i\geq m_i\right)\Big]\\
		&=\partial_{1,\cdots,r}h^{\bm m}(x)\P_x\left(\Y_t\in dy\vert \forall1\leq i\leq r, \M_i= m_i\right).
	\end{align*}
	This follows from the fact that $h^{\bm m}(x)=\P_x\left(\forall1\leq i\leq r, \M_i(\Y)\geq m_i\right)$, while $\bm m\mapsto \partial_{1,\cdots,r}h^{\bm m}(x)$ represents the density of $\M$.
	Eventually this shows that for any $x,y\in\M+\mathcal{C}$ and $t>0$:
	\begin{equation}
		\P_x\left(\Y_t\in dy\vert \M\right)=\P_x\left(\X_t\in dy\right).
	\end{equation}
	More generally the same reasoning shows that for any times $t_1,\cdots,t_k$ and $y_1,\cdots,y_k\in \V$ 
	\begin{equation}
		\P_x\left(\Y_{t_l}\in dy_l; 1\leq l\leq k\vert \M\right)=\P_x\left(\X_{t_l}\in dy_l; 1\leq l\leq k\right)
	\end{equation}
	for $k$ a positive integer. This corresponds to the statement of Theorem~\ref{thm:williams}.
	
	Therefore to prove our main statement on path decomposition it is enough to prove that Lemma~\ref{lemma:williams_crucial} does indeed hold. The remaining part of this subsection is dedicated to proving this statement. 
	Throughout the proof we consider the reflection group centered at $M$, $W^\M$, introduced before. Furthermore in the statement of Lemma~\ref{lemma:williams_crucial} we can always assume that $f_t\equiv 1$, which we will do in what follows.
	
	\subsubsection{The case $i=1$} We start by treating the case where the process $\X$ has hit only one part of the boundary of the Weyl chamber. This particular case contains all the ideas used for the general proof.
	
	First of all, note that since $\X^1$ is defined via a Doob's transform from the process $\Y$, we can write that for any $z\in\partial\mathcal{C}$, 
	\[
	\P_x\left(T_{\M+\partial \mathcal C}(\X)\in du, \X_u\in dz\right)=\frac{\partial_{1,\cdots,r}h^\M(z)}{\partial_{1,\cdots,r}h^\M(x)}\P_x\left(T_{\M+\partial \mathcal C}(\Y)\in du, \Y_u\in dz\right).
	\]
	Whence by independence of the processes appearing in the decomposition of $\X$
	\begin{align*}
		\P_x\left(\X_t\in dy, \X\text{ has only hit }\M+\partial \mathcal C_1\right)=\int_0^t\int_{\M+\partial \mathcal{C}_1}\P_x\left(\X_t\in dy, T_{\M+\partial \mathcal C}\in du, \X_u\in dz\right)&\\
		=\int_0^t\int_{\M+\partial \mathcal{C}_1}\frac{\partial_{1,\cdots,r}h^\M(z)}{\partial_{1,\cdots,r}h^\M(x)}\P_x\left(T_{\M+\partial \mathcal C}(\Y)\in du, \Y_u\in dz\right)\P_z\left(\X^2_{t-u}\in dy\right)&.
	\end{align*}
	In the above expression in order to make sense of $\P_z\left(\X^2_{t-u}\in dy\right)$ and because the conditioning is made with respect to the process killed when hitting the boundary, we may rely on the following fact:
	\begin{lemma}\label{lemma:1-hit}
		Let $(z_n)_{n\in\N}$ be a sequence inside the Weyl chamber that converges to $z\in\M+\partial\mathcal{C}_1$. Then
		\begin{equation}\label{equ:1-hit}
			\P_z\left(\X^2_{t}\in dy\right)\coloneqq\lim\limits_{n\rightarrow\infty}\P_{z_n}\left(\X^2_{t}\in dy\right)=\frac{\partial_{2,\cdots,r}h^\M(y)}{\partial_{1,\cdots,r}h^\M(z)}\partial_{1}\bm p_t^\M(z,dy).
		\end{equation}
		A similar statement holds with $z\in\M+\partial\mathcal{C}_i$ for any $1\leq i\leq r$.
	\end{lemma} 
	\begin{proof}
		For any such sequence, we know that 
		\[
		\P_{z_n}\left(\X^2_{t}\in dy\right)=\frac{\partial_{2,\cdots,r}h^\M(y)}{\partial_{2,\cdots,r}h^\M(z_n)}\bm p_t^\M(z_n,dy).
		\]
		As $z_n$ approaches $\M+\partial\mathcal{C}_1$ both $\partial_{2,\cdots,r}h^\M(z_n)$ and $\bm p_t^\M(z_n,y)$ get close to zero. However we claim that their Taylor expansions near $z$ are given by:
		\begin{align*}
			\partial_{2,\cdots,r}h^\M(z_n)&=\ps{z-z_n,e_1}\partial_{1,\cdots,r}h^\M(z)+o(\norm{z_n-z})\\
			\bm p_t^\M(z_n,y)&=\ps{z-z_n,e_1}\partial_{1}\bm p_t^\M(z,y)+o(\norm{z_n-z}),
		\end{align*}
		where $o(\norm{z_n-z})$ denotes a quantity which is negligible compared to $\norm{z_n-z}$ as $z_n\to z$.
		Of course these expansions imply Equation~\eqref{equ:1-hit}.
		
		To see why such expansions are valid, let us start by considering the first one and write $z_n=m_1^n\omega_1+\sum_{i=2}^r\ps{z,e_i}\omega_i$ with $m_1^n=m_1+\eps_n$ converging to $m_1=\ps{z,e_1}$. Then recalling Equation~\eqref{equ:derivatives_h} we deduce that 
		\begin{align*}
			\P_{z_n}\left(\M_i(\Y)\geq m_i, \text{for } 1\leq i\leq r\right)&=\P_{z_n}\left(m_1^n\leq \M_i(\Y)\leq m_1^n+\eps_n,\quad \M_i(\Y)\geq m_i, \text{for } 2\leq i\leq r\right)\\
			&=\eps_n\partial_1h^{\bm m}(z)+o(\eps_n),
		\end{align*}
		and more generally that
		\begin{align*}
			\partial_{2,\cdots,r}h^\M(z_n)&=\ps{z-z_n,e_1}\partial_{\ps{z,e_1}}\partial_{2,\cdots,r} h^\M(z)+o(\norm{z_n-z})\\
			&=\ps{z-z_n,e_1}\left(-\partial_{\ps{M,e_1}}\right)\partial_{2,\cdots,r} h^\M(z)+o(\norm{z_n-z}).
		\end{align*}
		Alternatively we could have argued using the maximum principle for $\partial_{2,\cdots,r}h^\M$, which thus admits an extremum on the hyperplane $\ps{x,e_1}=\M_1$ with its gradient normal to this hyperplane.
		
		As for $\bm p_t^\M$, by the reflection principle of Proposition~\ref{prop:reflection_killed} 
		\begin{align*}
			\bm p_t^\M(z_n,y)&=\sum_{s\in W^\M}\epsilon(s)\bm p_t(sz_n,y)\\
			&=\frac12\sum_{s\in W^\M}\epsilon(s)\Big(\bm p_t(sz_n,y)-\bm p_t(ss_1z_n,y)\Big)\\
			&=\frac12\sum_{s\in W^\M}\epsilon(s)\Big(\bm p_t(sz_n,y)-\bm p_t\left(s(z_n+\ps{z_n-z,e_1^\vee}e_1),y\right)\Big)\\
			&=\ps{z-z_n,e_1^\vee}\times \frac12\sum_{s\in W^\M}\epsilon(s)\ps{\nabla_x\bm p_t,se_1}(sz,y)+o(\norm{z_n-z}).
		\end{align*}
		Now for $z\in\M+\partial\mathcal{C}_1$ and $\M'$ close to $\M$ we can write, with $s'$ similar to $s$ but centered at $\M'$:  
		\begin{align*}
			\bm p_{t}^{\M'}(z,y)&=\sum_{s'\in W^{\M'}}\epsilon(s')\bm p_{t}(s'z,y)\\
			&=\frac12\sum_{s'\in W^{\M'}}\epsilon(s')\Big[\bm p_{t}(s'z,y)-\bm p_{t}(s's'_1z,y)\Big]\\
			&=\frac12\sum_{s'\in W^{\M'}}\epsilon(s')\Big[\bm p_{t}(s'z,y)-\bm p_{t}\left(s'(z+\ps{\M'-\M,e_1^\vee}e_1),y\right)\Big], \quad\text{hence}
		\end{align*}
		\begin{equation}\label{equ:ptM_boundary}
			\bm p_{t}^{\M'}(z,y)=\frac12\left(\bm p_{t}^{\M'}(z,y)-\bm p_{t}^{\M'}(z+\ps{\M'-\M,e_1^\vee}e_1,y)\right).
		\end{equation}
		Therefore taking derivatives with respect to $\ps{\M,e_1}$ of $\bm p_{t}^{\M'}(z,y)$ yields that for $z\in\partial\mathcal{C}_1$, 
		\[
		\frac12\sum_{s\in W^{\M}}\epsilon(s)\ps{\nabla_x\bm p_t,se_1}(sz,y)=\frac{\ps{e_i,e_i}}{2}\partial_{1}\bm p_{u}^\M(z,y)
		\]
		where the prefactor stems from the fact that $\ps{\M'-\M,e_1^\vee}=\frac{2}{\ps{e_i,e_i}}\ps{\M'-\M,e_1}$.
		This shows that $\bm p_t^\M(z_n,y)=\ps{z-z_n,e_1}\partial_{1}\bm p_t^\M(z,y)+o(\norm{z_n-z})$, concluding the proof of Lemma~\ref{lemma:1-hit}.
	\end{proof}
	We can now make use of Lemma~\ref{lemma:1-hit} in our computations to see that the case $i=1$ in Lemma~\ref{lemma:williams_crucial} follows from the equality, that no longer depends on the conditioning:
	\begin{lemma}\label{lemma:after1-hit}
		For any $x,y\in\M+\mathcal{C}$ and $t>0$,
		\begin{equation}
			\int_0^t\int_{\M+\partial \mathcal{C}_1}\P_x\left(T_{\M+\partial \mathcal C}(\Y)\in du, \Y_u\in dz\right)\partial_{1}\bm p_{t-u}^\M(z,y)=\partial_{1}\bm p_{t}^\M(x,y).
		\end{equation}
	\end{lemma} 
	\begin{proof}
		Note that for $s'\neq Id$ and $\M'$ sufficiently close to $\M$, 
		\begin{equation}\label{equ:crossing_boundary}
			\P_x\left(\Y_t\in d(s'y)\right)=
			\int_0^t\int_{\M+\partial \mathcal{C}}\P_x\left(T_{\M+\partial \mathcal C}(\Y)\in du, \Y_u\in dz\right)\P_z\left(\Y_{t-u}\in d(s'y)\right)
		\end{equation}
		since to reach $s'y$ the path has to cross $\M+\partial\mathcal{C}$ (this is true only for $\M'$ and $\M$ sufficiently close). Moreover the quantity that corresponds to $s'=Id$ in the expression of 
		\[
		\bm p_t^{\M'}(x,y)=\sum_{s'\in W^{\M'}}\epsilon(s')p_t(s'x,y)
		\]
		given by the reflection principle from Proposition~\ref{prop:reflection_killed} does not depend on $\M'$. As a consequence we can write that
		\[
		\partial_{1}\bm p_{t}^\M(x,y)=\int_0^t\int_{\M+\partial \mathcal{C}}\P_x\left(T_{\M+\partial \mathcal C}(\Y)\in du, \Y_u\in dz\right)\partial_{1}\bm p_{t-u}^\M(z,y).
		\]
		To conclude note that $\partial_{1}\bm p_{t}^\M(z,y)$ vanishes when $z\in\M+\partial\mathcal{C}\setminus\partial\mathcal{C}_1$ ---this is a consequence of Equation~\eqref{equ:ptM_boundary} above.
	\end{proof}
	
	\subsubsection{The case $i=2$}
	The approach remains essentially the same for $i=2$; in the same spirit as above, we see that
	\begin{align*}
		&\P_x\left(\X_t\in dy, \X-\M\text{ has hit }\partial \mathcal C_1\text{ before }\partial \mathcal C_2\right)\\
		&=\frac{\partial_{3,\cdots,r}h^\M(y)}{\partial_{1,\cdots,r}h^\M(x)}\int_0^t\int_{t_1}^t\int_{\M+\partial \mathcal{C}_1}\int_{\M+\partial \mathcal{C}_2}\P_x( T_{\M+\partial \mathcal C}(\Y)\in dt_1, \Y_{t_1}\in dz_1)\times\\
		&\hspace{6cm}\lim\limits_{z_n^1\rightarrow z_1}\frac{\P_{z_n^1}\left( T_{\M+\partial \mathcal C}(\Y)\in dt_2, \Y_{t_2}\in dz_2\right)}{\ps{z_1-z_n^1,e_1}}\partial_2\bm p^\M_{t-t_2}(z_2,y)\\
		&=\frac{\partial_{3,\cdots,r}h^\M(y)}{\partial_{1,\cdots,r}h^\M(x)}\int_0^t\int_{\M+\partial \mathcal{C}_1}\P_x( T_{\M+\partial \mathcal C}(\Y)\in dt_1, \Y_{t_1}\in dz_1)\lim\limits_{z_n^1\rightarrow z_1}\frac{\partial_2\bm p^\M_{t-t_1}(z_n^1,y)}{\ps{z_1-z_n^1,e_1}}\\
		&=\frac{\partial_{3,\cdots,r}h^\M(y)}{\partial_{1,\cdots,r}h^\M(x)}\int_0^t\int_{\M+\partial \mathcal{C}_1}\P_x( T_{\M+\partial \mathcal C}(\Y)\in dt_1, \Y_{t_1}\in dz_1)\partial_{1,2}\bm p^\M_{t-t_1}(z_1,y),
	\end{align*}
	where we have used the results from the case $i=1$, and relied on the fact that $\X^1$ will never come back to $\M+\partial \mathcal{C}_1$.
	As a consequence the proof boils down to proving the following analog of Lemma~\ref{lemma:after1-hit}:
	\[
	\int_0^t\int_{\M+\partial \mathcal{C}_1\cup \partial \mathcal{C}_2}\P_x( T_{\partial \mathcal C}(\Y)\in dt_1, \Y_{t_1}\in dz_1)\partial_{1,2}\bm p^\M_{t-t_1}(z_1,y)=\partial_{1,2}\bm p^\M_{t}(x,y).
	\]
	
	\subsubsection{The general case} 
	The proof in the general case relies on the very same computations as above. In the end we see that in order to prove that Theorem~\ref{thm:williams} all one has to do is to check the validity of the lemma below, which we then recursively apply to get the desired statement.
	\begin{lemma}\label{lemma:after-n-hit}
		For any $1\leq i\leq r$, 
		\[
		\int_0^t\int_{\partial \mathcal{C}_1\cup\cdots_\cup\partial \mathcal{C}_i}\P_x( T_{\partial \mathcal C}(\Y)\in du, \Y_{u}\in dz)\partial_{1,\cdots,i}\bm p^\M_{t-u}(z,y)=\partial_{1,\cdots,i}\bm p^\M_{t}(x,y).
		\]
	\end{lemma}
	\begin{proof}
		To start with, recall from Equation~\eqref{equ:ptM_boundary} that if $z\in\M+\partial\mathcal{C}_j$ for $j>i$ then
		\[
		\bm p_{t}^{\M'}(z,y)=\frac12\left(\bm p_{t}^{\M'}(z,y)-\bm p_{t}^{\M'}(z+\ps{\M'-\M,e_j^\vee}e_j,y)\right),
		\]
		and therefore $\partial_{1,\cdots,i}\bm p^\M_{t-t_1}(z,y)=0$. As a consequence we only need to prove that
		\[
		\int_0^t\int_{\partial \mathcal{C}}\P_x( T_{\partial \mathcal C}(\Y)\in du, \Y_{u}\in dz)\partial_{1,\cdots,i}\bm p^\M_{t-u}(z,y)=\partial_{1,\cdots,i}\bm p^\M_{t}(x,y).
		\]
		This follows from Equation~\eqref{equ:crossing_boundary} along the same lines as in the proof of Lemma~\ref{lemma:after1-hit}.
	\end{proof}
	This allows to wrap up the proof of Lemma~\ref{lemma:williams_crucial} and therefore of Theorem~\ref{thm:williams}.
	

	\section{On the process started from infinity.}\label{section:brown_inf}
	In this section we focus on the case where the process being considered in the statement of Theorem~\ref{thm:williams} is a drifted Brownian motion over $\V$. To be more specific we study in this section the behaviour of the drifted Brownian motion $B^\nu$ when conditioned on having $\M\to\infty$ inside $\mathcal{C}_-$. To do so we will provide a detailed analysis of the process $\B^\nu$, $\nu\in\mathcal{C}$, defined in the same fashion as in Corollary~\ref{cor:brownien_drift}, in the asymptotic where its starting point $\bm x$ will diverge inside the Weyl chamber $\mathcal{C}$ . 
		This process is defined by joining:
		\begin{itemize}
			\item A diffusion process $\X^1$ started from $\bm x\in\mathcal{C}$, with generator $\frac12\Delta+\nabla\log\partial_{1,\cdots,r}h$ and run until it hits $\partial\mathcal{C}$, say at $z_1\in \partial\mathcal{C}_1$. Here $h$ is given by Equation~\eqref{equ:def_h}.
			\item Then run an independent process $\X^2$ started from $z_1$ and with generator $\frac12\Delta+\nabla\log\partial_{2,\cdots,r}h$, upon hitting $\partial\mathcal{C}$.
			\item Thus define a family of processes $(\X^1,\cdots,\X^r)$. When $\X^r$ reaches the boundary of $\partial\mathcal{C}$, sample $\X^{r+1}$ with generator $\frac12\Delta+\nabla\log h$.
		\end{itemize}
		Our goal is to describe the behaviour of such a process $\B^\nu$ when $\bm x\to\infty$ inside $\mathcal{C}$. We will see that, in contrast with the uniqueness of the entrance law from $\infty$ of the process studied in~\cite[Section 6]{BOc}, the process will behave in very different ways according to the way $\bm x$ diverges inside $\mathcal{C}$.  In what follows we will say that $\bm x\to\infty$ inside $\mathcal{C}$ when $\ps{\bm x,e_i}\to+\infty$ for all $1\leq i\leq r$.
	
	Let us describe in details this asymptotic in the case of $A_2$. For this purpose, let us introduce for any $0<\eta<1$ and a process $\Y$ starting inside $\mathcal{C}$ the events $\mathcal{E}^\eta_1(\Y)\coloneqq\left\{\forall t<T_{\partial\mathcal{C}},\quad \ps{\Y,e_2}>(1-\eta)\ps{\Y_0,e_2}\right\}$ and analogously $\mathcal{E}^\eta_2$.
	We then define a process $\Y^1$ by joining:
	\begin{enumerate}
		\item A Brownian motion with drift $s_1s_2\nu$, started from $\bm x$, and conditioned on $\mathcal{E}_1^\eta(\Y^1)$. This process is run until hitting $\partial\mathcal{C}$ (over $z_1\in\partial\mathcal{C}_1$).
		\item A Brownian motion started from $z_1$, with drift $s_2\nu$ and conditioned not to hit $\partial\mathcal{C}_1$ again. By this we mean a diffusion with drift
		\[
			d_2(y)\coloneqq \frac{s_2e^{\ps{s_2\nu,y}}-s_1s_2e^{\ps{s_1s_2\nu,y}}}{e^{\ps{s_2\nu,y}}-e^{\ps{s_1s_2\nu,y}}}=\frac{s_2-s_1s_2e^{\ps{\nu,\rho}\ps{y,e_1}}}{1-e^{\ps{\nu,\rho}\ps{y,e_1}}}\cdot
			\]
		This process will be run upon hitting $\partial\mathcal{C}$ at $z_2\in\partial\mathcal{C}_2$.
		\item A Brownian motion with drift $\nu$, started from $z_2$, and conditioned to stay inside $\mathcal{C}$.
	\end{enumerate}
	Similarly by exchanging $e_1$ and $e_2$ we define a process $\Y^2$. As we will see, the process $\B^\nu$ will behave like $\Y^1$ or $\Y^2$ depending on the variable $x^\bot\coloneqq \ps{s_2s_1\nu-s_1s_2\nu,\bm x}$. 
	
	In what follows we consider $(Z^1_t ,Z^2_t)_{t\geq 0}$ a pair of bounded, continuous processes independent from $\B^\nu$, $\Y^1$ and $\Y^2$, as well as stationary (in the sense that the law of $(Z^1_{t+s},Z^2_{t+s})_{t\geq 0}$ is independent of $s\geq0$). We further demand that the processes $(Z^1_t)_{t\leq s}$ and $(Z^2_{t})_{t\geq h}$ become independent as $h-s\to+\infty$ by requiring that for any bounded continuous functions $f_1,f_2$ 
	\begin{align*}
		\lim\limits_{h-s\to+\infty}\expect{\int_0^{s}f_1(t)Z^1_tdt\int_0^{+\infty}f_2(t)Z^2_{t+h}dt}-\expect{\int_0^{s}f_1(t)Z^1_tdt}\expect{\int_0^{+\infty}f_2(t)Z^2_tdt}=0.
	\end{align*}
	With these notations at hand the following statement can be seen as a reformulation of the above heuristic: 
	\begin{proposition}\label{prop:asymptot_A2}
		For $i=1,2$, set for $\Y$ either $\B^\nu$, $\Y^1$ or $\Y^2$
		\begin{equation}
			\mathrm{J}_i(\Y)\coloneqq \int_0^{+\infty} e^{-\ps{\Y_t,e_i}}Z^i_t dt.
		\end{equation} 
		Then for any bounded, continuous function $F:\R^2\to\R$, in the limit where $\bm x\to\infty$ inside $\mathcal{C}$ 
		\begin{equation}
			\E_{\bm x}\left[F\left(\mathrm{J}_1(\B^\nu),\mathrm{J}_2(\B^\nu)\right)\right]\sim\E_{\bm x}\left[F\left(\mathrm{J}_1(\Y^1),\mathrm{J}_2(\Y^1)\right)\right]\quad\text{if }x^\bot\to-\infty.
		\end{equation}
		Conversely if $x^\bot\to+\infty$ then $\E_{\bm x}\left[F\left(\mathrm{J}_1(\B^\nu),\mathrm{J}_2(\B^\nu)\right)\right]\sim\E_{\bm x}\left[F\left(\mathrm{J}_1(\Y^2),\mathrm{J}_2(\Y^2)\right)\right].$
	\end{proposition}
    As a consequence of the proof of Proposition~\ref{prop:asymptot_A2} we will see that if $x^\bot\to -\infty$, then the processes $\B^\nu$ and $\Y^1$ have asymptotically the same behaviour by showing that their drift functions are asymptotically close in the sense that for any positive $\eta$,
    \[
    \lim\limits_{\substack{\bm x\to\infty\\ x^\bot\to -\infty}}\P_{\bm x}\left(\sup_{0\leq t\leq T_{\partial\mathcal{C}}}\norm{\nabla\log\partial_{1,2}h(\B^\nu)-s_1s_2\nu}\leq Ce^{-\eta\ps{\nu,e_1}\ps{\bm x,e_2}}\right)=1
    \]
    and likewise for the two other components of the path.
    On the contrary if $\ps{s_2s_1\nu-s_1s_2\nu,\bm x}\to +\infty$ then the process will behave like $\Y^2$. In the case where $\ps{s_1s_2\nu-s_2s_1\nu,\bm x}$ remains bounded, then the law of the process $\B^\nu$ will be essentially realized by determining which boundary the process hits first and then running one of the two above processes accordingly.
	These different behaviours are illustrated in Figure~\ref{fig:regimes_brown}. We stress that in the limit where $\bm x\to\infty$ the events $\mathcal{E}_1(\Y^1)$ and $\mathcal{E}_2(\Y^2)$ have probability asymptotically $1$ so the conditioning becomes unnecessary.
	
	\begin{minipage}[c]{0.5\linewidth}
		\includegraphics[scale=0.2]{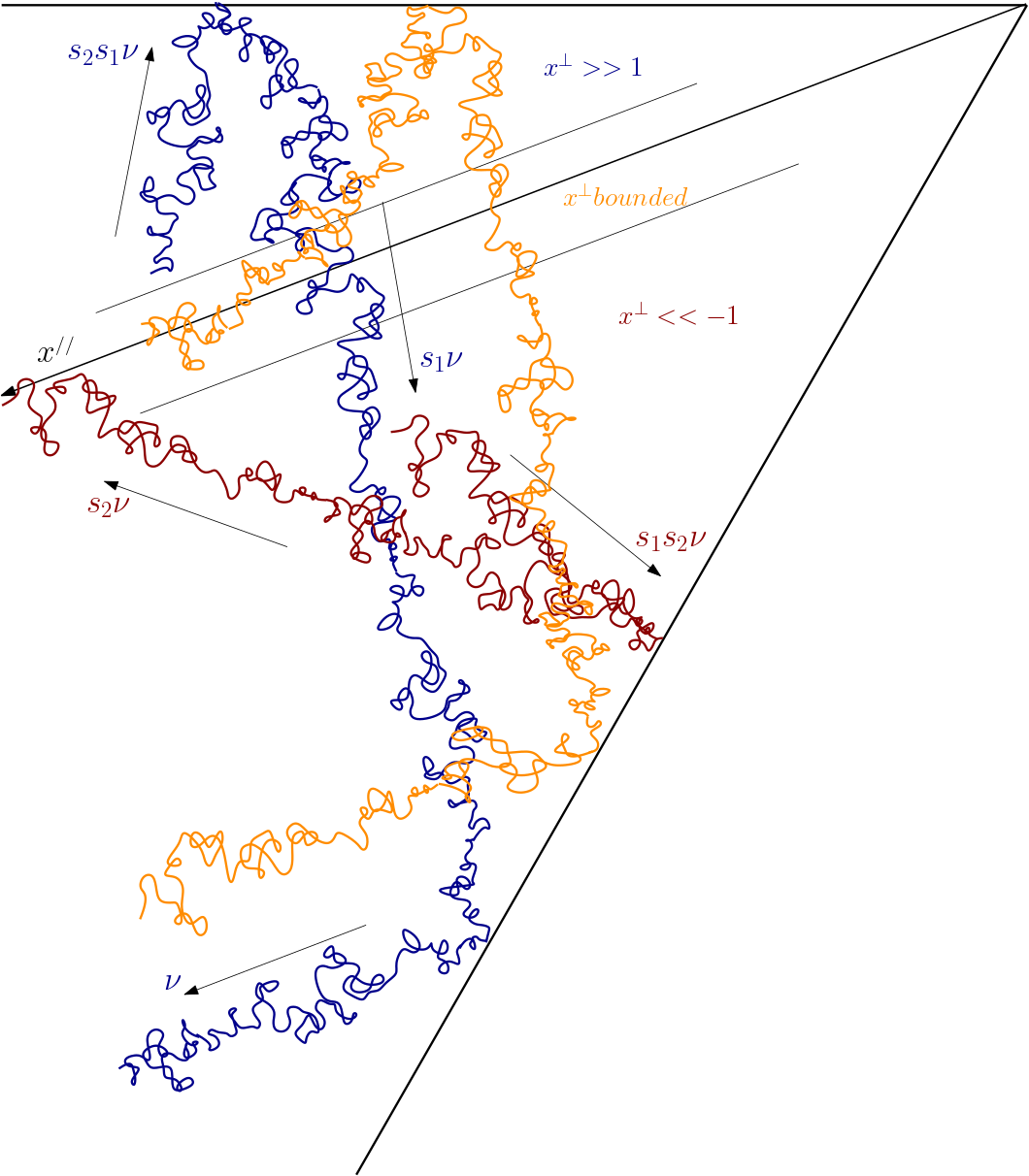}
		\captionof{figure}{The three different behaviours of the process $\B^\nu$ when $\bm x\rightarrow+\infty$.}
		\label{fig:regimes_brown}
	\end{minipage}
	\begin{minipage}[c]{0.5\linewidth}
		\includegraphics[scale=0.2]{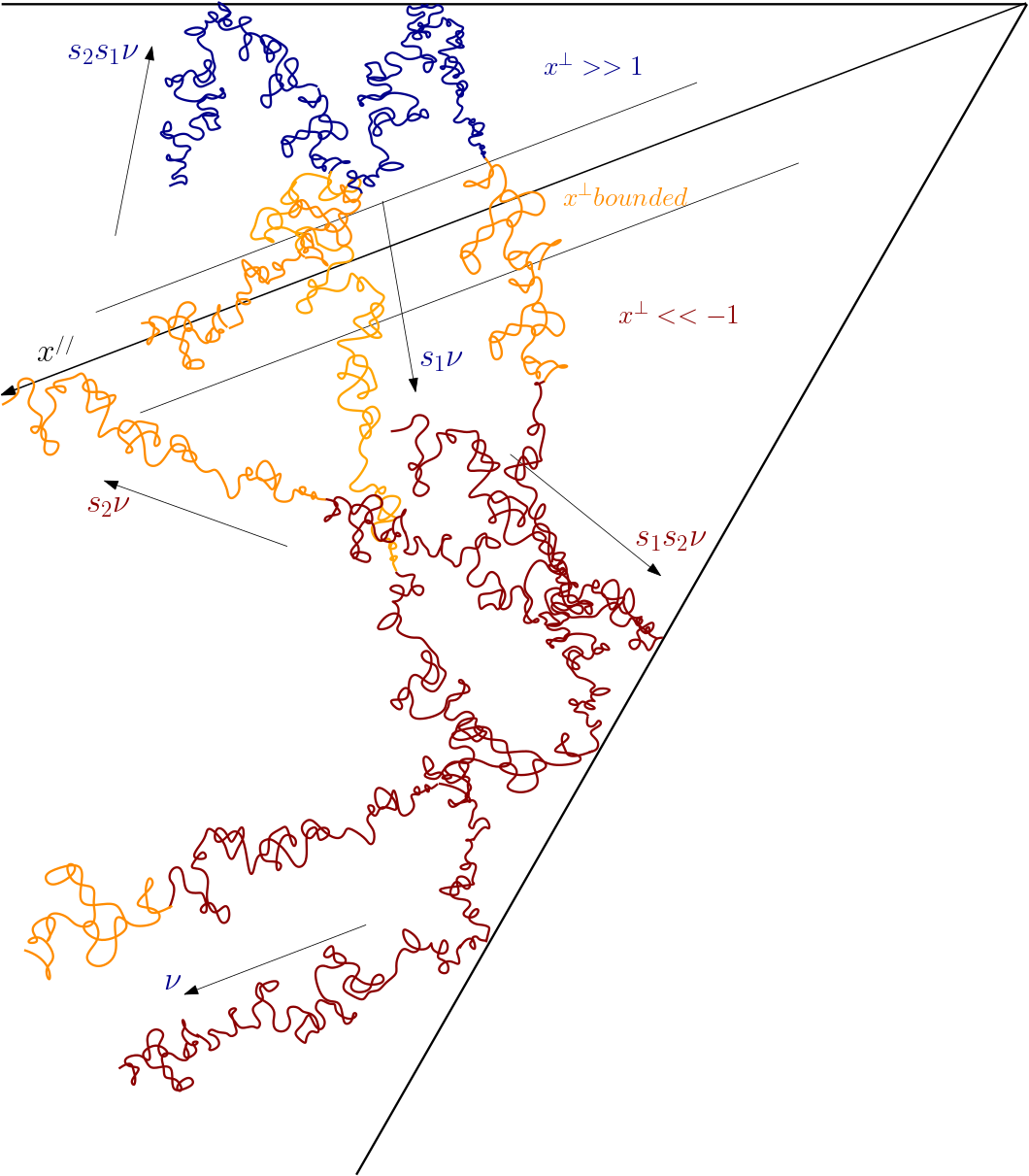}
		\captionof{figure}{Components of the path depending on whether $\ps{\B^\nu,e_i}$ is close to zero or large.}
		\label{fig:time_decomp}
	\end{minipage}
	
	In the general case a similar asymptotic behaviour for the process does hold, but takes more effort to be properly stated. For this purpose, let us consider $(Z^1,\cdots, Z^r)$ as above with the additional property that $Z^i$ and $Z^j$ are independent as soon as $\ps{e_i,e_j}=0$. We also introduce the subset $W_{1,\cdots,r}$ of $W$ defined as the set of the $s\in W$ that admit a reduced expression containing all the reflections $ s_{1},\cdots, s_{r}$ and set for $i=1,\cdots,r$ the notation 
	\begin{equation}\label{eq:def_J}
		J_i(\mu)\coloneqq \int_0^{+\infty}e^{-B^\mu_t}Z^i_tdt
	\end{equation}
	for $\mu\in\R$ positive, where $B^\mu$ is the process started from $+\infty$ whose law is realized by joining a Brownian motion with negative drift $-\mu$ and variance $\ps{e_i,e_i}$ upon hitting the origin and a Brownian motion with positive drift $\mu$ and variance $\ps{e_i,e_i}$ conditioned to stay positive after. An alternative way of writing this integral is to use that $J_i(\mu)$ has same law as
	\[
	\int_{-\infty}^{+\infty}e^{-\mathcal B^\mu_t}Z_t^i dt
	\]
	where $\mathcal B^\mu$ is a two-sided Brownian motion with positive drift $\mu$ and conditioned to stay positive. For more details see Subsection~\ref{subsec:refl_J}.
	\begin{proposition}\label{prop:asymptot_J}
		Set $s=s_1\cdots s_r$ and assume that $\bm x\rightarrow \infty$ inside $\mathcal{C}$ in the asymptotic where $\ps{s\nu- s'\nu,\bm x}\to+\infty$ for all $s'\neq s\in W_{1,\cdots,r}$. Then for $F_1,\cdots, F_r$ bounded continuous over $\R^+$
		\begin{equation}
			\lim\limits_{\bm x\to\infty}\E_{\bm x}\left[\prod_{i=1}^r F_i(\mathrm J_i(\B^\nu))\right]=\prod_{i=1}^r\expect{F_i\Big(J_i\left(\ps{s_{i+1}\cdots s_r\nu,e_i}\right)\Big)}.
		\end{equation}
	\end{proposition}
	A similar statement does hold when different asymptotics are considered, that is when $s$ is an element of $W_{1}$ of the form $s_{\sigma 1}\cdots s_{\sigma r}$ for some permutation $\sigma$ of $\{1,\cdots,r\}$.
	Proposition~\ref{prop:asymptot_J} provides an alternative formulation of the result described above for the $A_2$ case. Proving these two statements is key in the derivation of the asymptotics of Toda Vertex Operators and class one Whittaker functions. In order to prove these statements it will be convenient to start with the $A_2$ case to settle the ideas and then proceed to the general proof and explain how the arguments developed can be adapted in the general framework. Namely in Subsections~\ref{subsec:A21} and~\ref{subsec:A22} we provide a detailed study of the process when the reflection group being considered is associated to $A_2$ and prove Proposition~\ref{prop:asymptot_J} under this assumption. The general proof of Proposition~\ref{prop:asymptot_J} will be carried in Subsection~\ref{subsec:proof_prop} below.
	
	For future purpose we introduce for $s\in W_{1,\cdots,r}$ the shorthand
	\[
	\lambda_s\coloneqq\prod_{i=1}^r\ps{s\nu-\nu,\omega_i^\vee}.
	\]
	
	\subsection{Location of the first hitting point of $\partial\mathcal{C}$}
	Before actually describing the different components of the path as explained above, we start by showing that with high probability in the asymptotic considered, we can assume that the process stays inside a certain subdomain of $\mathcal{C}$. This fact is itself a consequence of the following Proposition, that describes the location of the first hitting point of the boundary of $\mathcal{C}$ by the process $\B^\nu$:
	\begin{proposition}\label{prop:exact_proba_hitting}
		Assume that $\bm x\in\mathcal{C}$ and take any $y\in\V$ such that $\ps{y,e_i}<\ps{\bm x,e_i}$ for all $1\leq i\leq r$. Then for all $z\in y+\partial\mathcal{C}_1$
		\begin{equation}\label{eq:hit_loc_time}
			\P_{\bm x}\left(T_{y+\partial\mathcal{C}}\in dt; \B^\nu_t\in dz\right)=\frac{1}{2}\frac{U(z)}{U(\bm x)}\sum_{s\in W}\epsilon(s)\frac{\ps{s(\bm x-y),e_1^*}}{t(2\pi t)^{\frac r2}}e^{-\frac{\norm{s(\bm x-y)+y-z}^2}{2t}-\frac{\norm{\nu}^2t}{2}}dtdz,
		\end{equation}
		where recall that $e_1^*=\frac{e_1}{\sqrt{\ps{e_1,e_1}}} $ and with
		\[
		U(z)\coloneqq\sum_{s\in W_{1,\cdots,r}}\epsilon(s)\lambda_se^{\ps{s\nu,z}}.
		\]
	\end{proposition} 
	\begin{proof}
		Since the process $\B^\nu$ is defined using a Doob transform
		\begin{align*}
			\P_{\bm x}\left(T_{y+\partial\mathcal{C}}\in dt; \B^\nu_t\in dz\right)=\frac{\partial_{1,\cdots,r}h^{\M}(z)}{\partial_{1,\cdots,r}h^{\M}(\bm x)}e^{\ps{\nu,z-\bm x}-\frac{\norm{\nu}^2t}{2}}\P_{\bm x}\left(T_{y+\partial\mathcal{C}}\in dt; B_t\in dz\right)
		\end{align*}
		where $B$ is a Brownian motion over $\V$.
		Direct computations show that 
		\[
		\frac{\partial_{1,\cdots,r}h^{\M}(z)}{\partial_{1,\cdots,r}h^{\M}(\bm x)}e^{\ps{\nu,z-x}}=\frac{U(z)}{U(\bm x)}
		\] while thanks to the reasoning presented in the proof of Proposition~\ref{prop:reflection_killed}, adapted by considering paths from $s \bm x$ to $z$ for $s\in W^y$ with $\ps{s(\bm x-y),e_1}>0$ and that do not cross the hyperplane $\ps{\cdot-y,e_1}=0$,
		\begin{align*}
			\P_{\bm x}\left(T_{y+\partial\mathcal{C}}\in dt; B_t\in dz\right)=\sum_{\underset{\ps{s(\bm x-y),e_1}>0}{s\in W}}\epsilon(s)\P_{s(\bm x-y)+y}\left(T_{\ps{y,e_1}}(\ps{B,e_1})\in dt; B_t\in dz\right).
		\end{align*}
		Now these probabilities are well-known and given by
		\[
		\P_{\bm x}\left(T_0(\ps{B,e_1})\in dt; B_t\in dz\right)=\frac{\ps{\bm x,e_1^*}}{t(2\pi t)^\frac r2}e^{-\frac{\norm{\bm x-z}^2}{2t}}dtdz.
		\]
		Using the fact that for all $s\in W$, $\epsilon(s_1s)\ps{s_1s\bm x,e_1}=\epsilon(s)\ps{s\bm x,e_1}$, the latter sum can be rewritten in a more elegant fashion as
		\[
		\frac{1}{2}\sum_{s\in W}\epsilon(s)\frac{\ps{s(\bm x-y),e_1^*}}{t(2\pi t)^\frac r2}e^{-\frac{\norm{s(\bm x-y)+y-z}^2}{2t}-\frac{\norm{\nu}^2t}{2}}dtdz.
		\]
	\end{proof}
	Thanks to this proposition we can deduce that the Brownian trajectory will stay within a given region, with high probability in the asymptotic where $\bm x\to\infty$ inside $\mathcal{C}$.
	
	
	\subsubsection{The $A_2$ case}
	To start with, given a positive $\eta<1$ one can split $\mathcal{C}$ between the domains $\mathcal{U}_1$, $\mathcal{U}_2$ and $\mathcal{C}\setminus\mathcal{U}$ where for $(i,j)\in\{(1,2),(2,1)\}$,
	\begin{align*}
		\mathcal{U}_i\coloneqq \left\{z\in\mathcal{C}\text{ s.t. } \ps{z,e_j}> (1-\eta)\ps{\bm x,e_j}\right\}\quad\text{and}\quad \mathcal U\coloneqq \mathcal{U}_1\cup\mathcal{U}_2.
	\end{align*}
	The events  $\mathcal E_{i}$, $i=1,2$ are then the events that the process $\B^{\nu}$ stays inside $\mathcal{U}_{i}$ until it hits the boundary of $\mathcal{C}$. 
	\begin{proposition}\label{prop:asym_prob_A2}
		Assume that $\bm x\rightarrow\infty$ inside $\mathcal{C}$. Then 
		\begin{equation}
			U(\bm x)\P_{\bm x}\left(\mathcal E_{1}\right)-\lambda_{s_1s_2}e^{\ps{s_1s_2\nu,\bm x}}=\mathcal{O}(e^{\ps{s_1s_2\nu,\bm x}-\eta\ps{\bm x,e_2}\ps{\nu,e_1}})\cdot
		\end{equation}
	\end{proposition}
	\begin{proof}
		First note that the event $\mathcal{E}_1$ simply means that the process $\B^\nu$ first hits $y+\partial\mathcal{C}$ over its component $y+\partial\mathcal{C}_1$, where $y=(1-\eta)\ps{\bm x,e_2}\omega_2$. Therefore thanks to Proposition~\ref{prop:exact_proba_hitting} and its proof we can write that
		\begin{align*}
			&U(\bm x)\P_{\bm x}\left(\mathcal E_{i}\right)=\sum_{w\in W_{1,2}}\epsilon(w)\lambda_{w}\E_{\bm x}\left[e^{\ps{w\nu,B_{T_{y+\partial\mathcal{C}}}}-\frac{\ps{w\nu}^2T_{y+\partial\mathcal{C}}}{2}}\mathds 1_{T_{y+\partial\mathcal{C}}=T_{y+\partial\mathcal{C}_1}}\right]\\
			&=\lambda_{s_1s_2}e^{\ps{s_1s_2\nu,\bm x}}-\mathfrak{R}_{s_1s_2}+\mathfrak{R}_{s_2s_1}+\mathfrak{R}_{s_1s_2s_1}
		\end{align*}
		where we have used that $W_{1,2}$ is in that case reduced to the elements $s_1s_2,s_2s_1$ and $s_1s_2s_1$, and where we have set
		\begin{align*}
			\mathfrak{R}_{s_1s_2}\coloneqq \lambda_{s_1s_2}\E_{\bm x}\left[e^{\ps{s_1s_2\nu,B_{T_{y+\partial\mathcal{C}}}}-\frac{\ps{s_1s_2\nu}^2T_{y+\partial\mathcal{C}}}{2}}\mathds 1_{T_{y+\partial\mathcal{C}}=T_{y+\partial\mathcal{C}_2}}\right],\\
			\mathfrak{R}_{w}\coloneqq \epsilon(w)\lambda_{w}\E_{\bm x}\left[e^{\ps{w\nu,B_{T_{y+\partial\mathcal{C}}}}-\frac{\ps{w\nu}^2T_{y+\partial\mathcal{C}}}{2}}\mathds 1_{T_{y+\partial\mathcal{C}}=T_{y+\partial\mathcal{C}_1}}\right].
		\end{align*}
		
		Let us start by considering the first remainder term $\mathfrak{R}_{s_1s_2}$. In virtue of Proposition~\ref{prop:exact_proba_hitting}
		\begin{align*}
			&\mathfrak{R}_{s_1s_2}=e^{\ps{\bm x,s_1s_2\nu}}\sum_{s\in W}\epsilon(s)\int_{\R_+\times\partial\mathcal{C}_2}\frac{\ps{s(\bm x-y),e_2^*}}{2\pi t^2}e^{\ps{s(\bm x-y)+y-\bm x,z}}e^{-\frac{\norm{\bm x-y-z+s_1s_2\nu t}^2}{2t}}dtdz.
		\end{align*}
		These integrals are maximal for $s\in\{Id,s_2\}$ so it suffices to bound them in this case. Then the behaviour of the integral is governed by the minimum of the function $F:\partial\mathcal{C}_2\times\R_+\to\R_+$ given by
		\[
		F(z,t)\coloneqq \frac{\norm{\bm x-y-z+s_1s_2\nu t}^2}{2t}\cdot
		\]
		We need to distinguish between two possibilities: first of all if $\ps{\bm x-y,\omega_1}\ps{\nu,e_1}>\ps{\bm x-y,e_2}\ps{\nu,\omega_2}$, then the minimum of $F$ is attained at $t_0=\frac{\ps{\bm x-y,e_2}}{\ps{\nu,e_1}}$ and $z_0=\ps{\bm x-y+s_1s_2\nu t_0,\omega_1}\omega_1$, where $F$ is found to be equal to $\ps{\bm x-y,e_2}\ps{\nu,e_1}$. If $\ps{\bm x-y,\omega_1}\ps{\nu,e_1}>\ps{\bm x-y,e_2}\ps{\nu,\omega_2}$ then this minimum is reached at $t_0=\frac{\norm{x-y}}{\norm{\nu}}$ and $z_0=0$, and $F(z_0,t_0)=\norm{x-y}\norm{\nu}+\ps{x-y,s_1s_2\nu}$, which in the worst-case scenario where $s_1s_2s_1\nu$ and $y-\bm x$ are colinear, can be bounded below by $\ps{\bm x-y,e_2}\ps{\nu,e_1}$.
		In both cases we have the bound
		\[
		F(z_0,t_0)\geq\ps{\bm x-y,e_2}\ps{\nu,e_1}.
		\] 
		By Laplace's method we can therefore estimate the whole integral and check that 
		\[
		\int_{\R_+\times\partial\mathcal{C}_2}\frac{\ps{\bm x-y,e_2^*}}{2\pi t^2}e^{-F(z,t)}dtdz\leq Ce^{-F(z_0,t_0)}
		\]
		where $C$ is uniformly bounded as $\bm x\to\infty$. Therefore the term $\mathfrak{R}_{s_1s_2}$ is indeed a $\mathcal{O}\left(e^{\ps{\bm x,s_1s_2\nu}-\eta\ps{\bm x,e_2}\ps{\nu,e_1}}\right)$, which is as desired.
		
		We can now turn to the other remainder terms, and follow the approach just developed for $\mathfrak{R}_{s_1s_2}$. Namely we need to evaluate the minimum of the map $F:\partial\mathcal{C}_1\times\R_+\to\R_+$ defined by setting
		\[
		F(z,t)\coloneqq \frac{\norm{\bm x-y-z+w\nu t}^2}{2t}
		\]
		for $w=s_2s_1$ and $w=s_1s_2s_1$. Like before we have $F(z_0,t_0)\geq\ps{\bm x-y,e_1}\ps{\nu,e_2}$ for $w=s_2s_1$, but it can occur that $F(z_0,t_0)=0$ depending on the values of $\bm x$ and $\nu$. Nevertheless we can bound the integrals by a $\mathcal{O}\left(e^{-F(z_0,t_0)}\right)$ by Laplace's method, which yields the upper bound on the remainders
		\[
		\mathfrak{R}_{s_2s_1}\leq Ce^{\ps{\bm x,s_1s_2s_1\nu}}\quad\text{and}\quad\mathfrak{R}_{s_1s_2s_1}\leq Ce^{\ps{\bm x,s_1s_2s_1\nu}}
		\]
		where we have used that $\ps{\bm x,s_2s_1\nu}-\ps{\bm x-y,e_1}\ps{\nu,e_2}=\ps{\bm x,s_1s_2s_1\nu}$.
		All together we see that
		\[
		\mathfrak{R}_{s_1s_2}+\mathfrak{R}_{s_2s_1}+\mathfrak{R}_{s_1s_2s_1}=\mathcal{O}(e^{\ps{s_1s_2\nu,\bm x}+\eta\ps{\bm x,e_2}\ps{\nu,e_1}}).
		\]
	\end{proof}
	Thanks to this Proposition we see that, with high probability, the process $\B^\nu$ will stay in the domain $\mathcal{U}_1\cup\mathcal{U}_2$. Put differently, in the asymptotic where $\bm x\to\infty$ inside $\mathcal{C}$, at any time at most one component of the process along the simple roots $(e_1,e_2)$ will not be very positive. 
	
	
	\subsubsection{The general case}
	Up to reordering the simple roots $e_1,\cdots, e_r$ we consider $s=s_1\cdots s_r$ in what follows. Let us denote by $i_0$ the first index such that $s_{i_0}$ does not commute with all the $s_j$ for $j\leq i_0$ (note that $s_a$ and $s_b$ commute means that $\ps{e_a,e_b}=0$, and also that $s_a e_b=e_b$). For a positive $\eta>0$, let us introduce the subset of $\mathcal{C}$ 
	\begin{align*}
		\mathcal{U}_s\coloneqq \left\{z\in\mathcal{C}\text{ s.t. } \forall i_0\leq j\leq r,\quad \ps{z,e_j}> (1-\eta)\ps{\bm x,e_j}\right\},
	\end{align*}
	and the corresponding event $\mathcal E_s$ that the process $\B^{\nu}$ stays inside $\mathcal{U}_{s}$ until it hits the boundary of $\mathcal{C}$. In particular under $\mathcal{E}_s$ the process will hit $\partial\mathcal{C}$ over its boundary component $\partial\mathcal{C}_s\coloneqq \bigcup_{j=1}^{i_0-1}\partial\mathcal{C}_j$. 
	\begin{proposition}\label{prop:asym_prob}
		Assume that $\bm x\rightarrow\infty$ inside $\mathcal{C}$ in the asymptotic where $\ps{s\nu-s'\nu,\bm x}\to+\infty$ for all $s'\neq s\in W_{1,\cdots ,r}$. Then 
		\begin{equation}
			\lim\limits_{\bm x\to\infty}\P_{\bm x}\left(\mathcal{E}_s\right)=1.
		\end{equation}
	\end{proposition}
	\begin{proof}
		Like before thanks to Proposition~\ref{prop:exact_proba_hitting} we can write
		\[
		U(\bm x)\P_{\bm x}\left(\mathcal{E}_s\right)=\epsilon(s)\lambda_s e^{\ps{s\nu,\bm x}}+\sum_{w\in W_{1,\cdots r}}\epsilon(w)\lambda_w\mathfrak{R}_w
		\]
		where, with $y$ such that $\ps{y,e_j}=0$ for $j<i_0$ and $\ps{y,e_j}=(1-\eta)\ps{\bm x,e_j}$ for $i_0\leq j\leq r$,
		\begin{equation}\label{equ:proba_int}
			\begin{split}
				\mathfrak R_s&\coloneqq e^{\ps{\bm x,s\nu}}\sum_{\tau\in W}\epsilon(\tau)\int_{\R_+\times\partial\mathcal{C}\setminus\partial\mathcal{C}_s}\frac{\ps{\tau(\bm x-y),e_n^*}}{2\pi t^2}e^{\ps{\tau(\bm x-y)+y-\bm x,z}}e^{-\frac{\norm{\bm x-y-z+s\nu t}^2}{2t}}dtdz,\\
				\mathfrak R_w&\coloneqq e^{\ps{\bm x,w\nu}}\sum_{\tau\in W}\epsilon(\tau)\int_{\R_+\times\partial\mathcal{C}_s}\frac{\ps{\tau(\bm x-y),e_n^*}}{2\pi t^2}e^{\ps{\tau(\bm x-y)+y-\bm x,z}}e^{-\frac{\norm{\bm x-y-z+w\nu t}^2}{2t}}dtdz\quad w\neq s,
			\end{split}
		\end{equation}
		where $e_n^*$ denotes the (normalized) simple root normal orthogonal to the boundary component where $z$ lies. Therefore the proof of Proposition~\ref{prop:asym_prob} boils down to studying these integrals. 
		
		To start with, we note that for $z\in \partial\mathcal C$ we have $e^{\ps{\tau(\bm x-y)+y-\bm x,z}}\leq 1$ so the asymptotic is governed by the integrals with $\tau=I_d$. Then by Laplace's method the behaviour of the integrals that appear in $\mathfrak R_w$, $w\neq s$, is governed by the minimum of the map $F:\partial\mathcal{C}_s\times\R_+\to\R_+$ given by
		\[
		F(z,t)\coloneqq \frac{\norm{\bm x-y-z+w\nu t}^2}{2t}.
		\]
		This allows to show that the remainders $\mathfrak R_w$ for $w\neq s$ are lower-order terms in the asymptotic considered, since $e^{\ps{w\nu-s\nu,\bm x}}\to 0$. 
		
		Likewise the term $\mathfrak R_s$ is also a lower-order term in the asymptotic considered, and for this by Laplace's method it is enough to prove that the map $F:\partial\mathcal{C}\times\R_+\to\R_+$ with $w=s$ attains its minimum at some $z_0\in\partial\mathcal{C}_s$ and $t_0>0$ with $F(z_0,t_0)=0$. Now because $s\nu\not\in\mathcal{C}$ we know that this amounts to saying that the half-line $t\mapsto \bm x + s\nu t$ crosses $y+\partial\mathcal{C}$ on $y+\partial\mathcal{C}_s$, which follows from the fact that for $j\geq i_0$ we have $\ps{s\nu,e_j}>0$. To see why, note that $\ps{s\nu,e_j}\leq0$ implies that $s_js\not\in W_{1,\cdots,r}$ since otherwise we would have
		\[
		\ps{s\nu-s_js\nu,\bm x}=\ps{s\nu,e_j^\vee}\ps{\bm x,e_j^\vee}\leq 0,
		\]
		which contradicts our assumptions on the asymptotic of $\bm x$. Moreover explicit computations show that $\ps{s_js_1\cdots s_r\nu,\omega_j^\vee}>0$ as soon as $j\geq i_0$. As a consequence having $s_js\not\in W_{1,\cdots,r}$ implies that $j\leq i_0$. Therefore having $\ps{s\nu,e_j}\leq0$ implies that $j\leq i_0$.
	\end{proof}
	
	
	\subsection{On the decomposition of the path: the $A_2$ case}~\label{subsec:A21}
	With Proposition~\ref{prop:asym_prob_A2} at hand we are now in position to describe the path of the process $\B^\nu$ itself in the asymptotic where the starting point of the process $\bm x$ diverges inside the Weyl chamber $\mathcal{C}$, and to start with we will focus on the $A_2$ case. As explained above, it is very natural to consider the subsets $\mathcal{U}_1$ and $\mathcal{U}_2$ of $\mathcal{C}$, and associated events $\mathcal{E}_1$ and $\mathcal{E}_2$. Because thanks to Proposition~\ref{prop:asym_prob_A2} we know that in the asymptotic where $\bm x\to\infty$ we have that $\P_{\bm x}\left(\mathcal{E}_1\right)+\P_{\bm x}\left(\mathcal{E}_2\right)=1-\mathcal{O}\left(e^{-\ps{\lambda,\bm x}}\right)$ for some $\lambda\in\mathcal{C}$, we can look at the process $\B^\nu$ conditioned on one of the events $\mathcal{E}_i$ and retain all the necessary information. Without loss of generality we will consider the process $\B^\nu$ conditioned on the event $\mathcal{E}_1$.
	
	In order to prove Proposition~\ref{prop:asymptot_A2} we will look at each component of the path separately, according to the moments where the process contribute or not to the integrals $\mathrm J_i$ as described in Figure~\ref{fig:time_decomp}. To be more specific we will introduce stopping times $T_1,T_1'$ and $T_2$ such that:
		\begin{itemize}
			\item the portion in blue in Figure~\ref{fig:time_decomp} corresponds to the path of the process in the time interval $(T_1,T_1')$;
			\item the component in red there is associated to the time interval $(T_2,+\infty)$;
			\item the portion in orange corresponds to $(0,T_1)\cup(T_1',T_2)$.
	\end{itemize}
	
	\subsubsection{The process before $T_{\partial\mathcal{C}}$} 
	Recall that before hitting $\partial\mathcal{C}$ for the first time, the process $\B^\nu$ has a drift given by 
	\[
		\nabla\log\partial_{1,2} h(z)= \sum_{w\in W_{1,2}}w\frac{\epsilon(w)\lambda_we^{\ps{w\nu,z}}}{U(z)},\quad U(z)=\sum_{w\in W_{1,2}}\epsilon(w)\lambda_we^{\ps{w\nu,z}}\cdot
	\]
    The latter can be rewritten under the form
	\begin{align*}
		s_1s_2\nu+(s_2s_1\nu-s_1s_2\nu)\frac{\lambda_{s_2s_1}e^{\ps{s_2s_1\nu,z}}}{\sum_{w\in W_{1,2}}\epsilon(w)\lambda_we^{\ps{w\nu,z}}}+(s_1s_2\nu-s_1s_2s_1\nu)\frac{\lambda_{s_1s_2s_1}e^{\ps{s_1s_2s_1\nu,z}}}{\sum_{w\in W_{1,2}}\epsilon(w)\lambda_we^{\ps{w\nu,z}}}\cdot
	\end{align*}
	Under the assumption that the process stays within the domain $\mathcal{U}_1$, this drift is equal to
	\begin{align*}
		s_1s_2\nu+\mathfrak{R}_1(z),\quad\text{with}\quad\norm{\mathfrak{R}_1(z)}\leq Ce^{\ps{\nu,e_2}\ps{z,e_1}-(1-\eta)\ps{\nu,e_1}\ps{\bm x,e_2}}
	\end{align*}
	where $C$ only depends on $\nu$.
	
	Now for $0<\eta<\frac12$ let us consider $T_1$ to be the first time where the $e_1$ component of the process reaches $\eta\ps{\bm x,e_2}$, that is 
		\begin{align*}
			&T_1\coloneqq \inf\limits_{t\geq0} \ps{\B^\nu_t,e_1}<\eta\ps{\bm x,e_2}\quad \text{(which is finite almost surely)}, \\
			\text{and also set}\quad&T_0\coloneqq\inf\limits_{t\geq T_1} \ps{\B^\nu_t,e_1}>(1-2\eta)\frac{\ps{\nu,e_1}}{\ps{\nu,e_2}}\ps{\bm x,e_2}.
		\end{align*}
		The following lemma explains that as soon as the $e_1$ component has reached the level $\eta\ps{\bm x,e_2}$ (that is for $t\geq T_1$), it will likely stay small until $\B^\nu$ hits the boundary of the Weyl chamber $\mathcal C$:
	\begin{lemma}
		For $\eta>0$ small enough, as $\bm x\to\infty$ inside $\mathcal{C}$
		\begin{equation}
			\lim\limits_{\bm x\to\infty}\P_{\bm x}(T_{\partial\mathcal{C}}<T_0)=1.
		\end{equation}
	\end{lemma}
	\begin{proof}
		Between $T_1$ and $T_0$ the $e_1$ component of the drift of $\B^\nu$ is given by $-\ps{\nu,\rho}+\ps{\mathfrak R_1(z),e_1}$ with $\norm{\ps{\mathfrak R_1(z),e_1}}\leq C'e^{-\eta\ps{\bm x,e_2}}$. As a consequence the probability that $T_0<T_{\partial\mathcal{C}}$ is smaller than the probability that a Brownian motion with drift $-\ps{\nu,\rho}+C'e^{-\eta\ps{\bm x,e_2}}$ (and variance $\ps{e_1,e_1}$) reaches  $\left((1-2\eta)\frac{\ps{\nu,e_1}}{\ps{\nu,e_2}}-\eta\right)\ps{\bm x,e_2}$ before reaching $-\eta\ps{\bm x,e_2}$, which is readily seen to converge to $0$ as $\ps{\bm x,e_2}\to+\infty$ provided that $\eta$ is chosen small enough.
	\end{proof}
	Therefore we see that between $T_1$ and $T_{\partial\mathcal{C}}$, with high probability (under the event $\mathcal{E}_1$) the drift of the process $\B^\nu$ will be given by that of $\Y^1$ plus a remainder term $\mathfrak R_1$ uniformly bounded by $Ce^{-\eta\ps{\bm x,e_2}}$.
	
	\subsubsection{The process between $T_{\partial\mathcal{C}}$ and $T_{\partial\mathcal{C}_2}$}
	After having hit the boundary of $\mathcal{C}$, the process $\B^\nu$ will have a drift given by
	\[
	\nabla \log \partial_{2}h^{\M}(z)=\sum_{s\in W_{2}}\frac{\epsilon(s) \lambda_s e^{\ps{s\nu,z}}}{U_2(z)}s\nu,\quad U_2(z)=\sum_{s\in W_{2}}\epsilon(s)\lambda_s e^{\ps{s\nu,z}}.
	\]
	Like before this drift admits the expression
	\begin{equation}\label{equ:drift_approx1'}
		\nabla \log \partial_{2}h^{\M}(z)=\frac{s_2\nu-s_1s_2\nu e^{\ps{s_1s_2\nu-s_2\nu,z}}}{1-e^{\ps{s_1s_2\nu-s_2\nu,z}}}+\mathfrak R_{2}(z)
	\end{equation}
	where the remainder term can be bounded as $\norm{\mathfrak{R}_2(z)}\leq Ce^{-\ps{\nu,e_1}\ps{z,\rho}}$. With high probability this remainder is negligible between $T_{\partial\mathcal{C}}$ and $T_{\partial\mathcal{C}_2}$. Indeed, we already know that at time $T_{\partial\mathcal{C}}$ under the event $\mathcal{E}^\eta_1$ we have $\ps{\B^\nu,e_2}>(1-\eta)\ps{\bm x,e_2}$, and therefore that $\norm{\mathfrak R_2(\B^\nu)}\leq Ce^{-\eta\ps{\nu,e_1}\ps{\bm x,e_2}}$ between $T_{\partial\mathcal{C}}$ and $T_2$, where \begin{align*}
			T_2\coloneqq\inf_{t\geq 0}\ps{\B_t^\nu,e_2}<\eta\ps{\bm x,e_2}\text{. Likewise if }T_1'\coloneqq\sup\limits_{t\leq T_{\partial\mathcal{C}_2}}\ps{\B_t^\nu,e_1}>\eta\ps{\bm x,e_2}
	\end{align*}
	then for $t$ between $T_1'$ and $T_{\partial\mathcal{C}_2}$, $\norm{\mathfrak R_2(\B^\nu)}\leq Ce^{-\eta\ps{\nu,e_1}\ps{\bm x,e_2}}$. Therefore for such times the process will behave like $\Y^1$, and all we need to check is that with probability $1-o(1)$ we have $T_1'<T_2$:
	\begin{lemma}
		For $\eta>0$ small enough, as $\bm x\to\infty$ inside $\mathcal{C}$
		\begin{equation}\label{eq:conv_proba_T12}
			\lim\limits_{\bm x\to\infty}\P_{\bm x}\left(T_1'<T_2\vert \mathcal{E}_1^\eta\right)=1.
		\end{equation}
	\end{lemma}
	\begin{proof}
		First assume that $\mathfrak R_2=0$. In that case the $e_1$ component of $\B^\nu$ has the law of a Brownian motion with positive drift $\ps{\nu,\rho}$ and variance $\ps{e_1,e_1}=2$ conditioned to stay positive, while the $e_2$ component of the path is an independent Brownian motion with negative drift lower bounded by $-\ps{\nu,\rho}$ and variance $\ps{e_2,e_2}=2$. As a consequence by the time-reversal property of drifted Brownian motion~\cite{Williams} $T_1'-T_{\partial\mathcal{C}}$ has same law as
		$\tau_1'\coloneqq\inf\limits_{t\geq 0}\sqrt 2B_t+\ps{\nu,\rho} t>\eta\ps{\bm x,e_2}$ where $B$ has the law of a standard  one-dimensional Brownian motion started from $0$. The law of $\tau_1'$ is well known~\cite[Equation (5.12)]{KS} and given by \[
		\mathbb P\left(\tau_1'\in dt\right)=\frac{\ps{\nu,\rho}}{\sqrt{4\pi t^3}}e^{-\frac{\left(\eta\ps{\bm x,e_2}-\ps{\nu,\rho}t\right)^2}{4t}},
		\]
		which concentrates around $\eta\frac{\ps{\bm x, e_2}}{\ps{\nu,\rho}}$ as the latter diverges. Likewise we see that with probability $1-o(1)$ we have $\tau_2-T_{\partial\mathcal{C}}>(1-3\eta)\frac{\ps{\bm x, e_2}}{\ps{\nu,\rho}}$, so by taking $\eta$ small enough so that $\frac{1-3\eta}{\eta}>2$ we see that Equation~\eqref{eq:conv_proba_T12} does indeed hold for $\mathfrak{R}_2=0$.
		
		Now it is readily seen that for $T_{\partial\mathcal{C}}\leq t\leq T_1'\wedge T_2$ we have $\norm{\mathfrak R_2(\B^\nu)}\leq Ce^{-\eta\ps{\nu,e_1}\ps{\bm x,e_2}}$. Therefore for any positive $\eps$ and with $\eta$ small enough like above, we can assume that $(1-\eps)T_1'<\tau_1'<2\tau_2<(2+\eps)T_2$ with asymptotic probability $1$, showing Equation~\eqref{eq:conv_proba_T12} in the general case.
	\end{proof}
	Therefore between $T_{\partial\mathcal{C}}$ and $T_{\partial\mathcal{C}_2}$ the process will behave like $\Y^1$, in the sense that $\B^\nu$ will have a drift given by that of $\Y^1$ plus a remainder term $\mathfrak{R}_2$, where $\norm{\mathfrak R_2}\leq C e^{-\ps{\nu,e_1}\ps{\bm x,e_2}}$. 
	
	\subsubsection{The process after $T_{\partial\mathcal{C}_2}$}
	The third component of the process, \emph{i.e.} after having hit both components of $\partial\mathcal{C}$, has the law of $B^\nu$ conditioned to stay inside $\partial\mathcal{C}$. This means that its drift is equal to
	\[
	\nabla \log h^{\M}(z)=\sum_{s\in W_{2}}\frac{\epsilon(s) \lambda_s e^{\ps{s\nu,z}}}{U_3(z)}s\nu,\quad U_3(z)=\sum_{s\in W_{2}}\epsilon(s)\lambda_s e^{\ps{s\nu,z}},
	\]
	which coincides with the drift of the process $\Y^1$.
	
	For future reference and in order to prove Proposition~\ref{prop:asymptot_J} we stress that
	\begin{equation}\label{equ:drift_approx2'}
		\nabla \log h^{\M}(z)=\frac{\nu-s_2\nu e^{\ps{s_2\nu-\nu,z}}}{1-e^{\ps{s_2\nu-\nu,z}}}+\mathfrak R_{3}(z),
	\end{equation}
	where $\norm{\mathfrak{R}_3(z)}\leq Ce^{-\ps{\nu,e_1}\ps{z,e_1}}$. Like before we can use the explicit expression of the drift to show that with probability asymptotically $1$ the $e_1$ component of the process will never come back to $\eta\ps{\bm x,e_2}$ after $T_{\partial\mathcal{C}}$, so that the $e_2$ component of $\B^\nu$ after $T_{\partial\mathcal{C}_2}$ has the law of a Brownian motion with positive drift $\ps{\nu,e_2}$ and variance $2$ conditioned to stay positive. 
	
	\subsection{Proof of Propositions~\ref{prop:asymptot_A2} and~\ref{prop:asymptot_J} for $A_2$}~\label{subsec:A22}
	We are now in position to prove Propositions~\ref{prop:asymptot_A2} and~\ref{prop:asymptot_J} under the assumption that the root system being considered is associated to $A_2$. Let us decompose the integrals $\mathrm{J}_i$, $i=1,2$, as 
	\begin{align*}
		\mathrm J_i&=\int_0^{T_1}e^{-\ps{\B^\nu_t,e_i}}Z^i_tdt+\int_{T_1}^{+\infty}e^{-\ps{\B^\nu_t,e_i}}Z^i_tdt.
	\end{align*}
	As illustrated in Figure~\ref{fig:time_decomp} the integrals over the subset $[0,T_1]$ will become negligible in the $\bm x\to\infty$ limit. Indeed, in this region the $e_1$ and $e_2$ components of the process are bounded below by $(1-2\eta)\ps{\bm x,e_2}$, so that
	\[
	\int_0^{T_1}e^{-\ps{\B^\nu_t,e_i}}Z^i_tdt\leq e^{(1-3\eta)\ps{\bm x,e_2}}\int_{0}^{T_1}e^{-\left(\ps{\B^\nu_t,e_i}-(1-3\eta)\ps{\bm x,e_2}\right)}Z^i_tdt
	\]
	where the integral on the right-hand side is uniformly bounded in $\bm x$.
	As a consequence for any bounded continuous function $F:\R^2\to\R$
	\begin{align*}
		\E_{\bm x}\left[F(\mathrm  J_1,\mathrm  J_2)\vert\mathcal{E}_1\right]=\E_{\bm x}\left[F\left(\int_{T_1}^{+\infty}e^{-\ps{\B^\nu_t,e_1}}Z^1_tdt,\int_{T_1}^{+\infty}e^{-\ps{\B^\nu_t,e_2}}Z^2_tdt\right)\vert\mathcal{E}_1\right]+o(1).
	\end{align*}
	By the Markov property of $\B^\nu$, we know that the process $(\B^\nu_t)_{t\geq T_1}$ only depends on $(\B^\nu_t)_{t<T_1}$ via the location of $\B^\nu_{T_1}$. Therefore the latter can be rewritten as
	\begin{align*}
		\int_{u+\partial\mathcal{C}_2}\P_{\bm x}\left(\B^\nu_{T_1}\in dz\right)\E_z\left[ F\left(\int_{0}^{+\infty}e^{-\ps{\B^\nu_t,e_1}}Z^1_{t+T_1}dt,\int_{0}^{+\infty}e^{-\ps{\B^\nu_t,e_2}}Z^2_{t+T_1}dt\right)\vert\mathcal{E}_1\right]
	\end{align*}
	where $u=u(\bm x)$ is such that $\ps{u,e_2}=(1-\eta)\ps{\bm x,e_2}$ and $\ps{u,e_1}=\eta\ps{\bm x,e_1}$.
	
	Now we have seen that after $T_1$ the drift of the process $\B^\nu$ and that of $\Y^1$ only differ by a remainder term $\mathfrak R(x)$ which becomes negligible in the $\bm x\to\infty$ limit. Using continuity of $F$ together with a comparison result such as~\cite[Chapter 5-Proposition 2.18]{KS} to compare the processes $\B^\nu$ and $\Y^1$ we see that the latter is asymptotically equivalent to
	\begin{align*}
		\int_{u+\partial\mathcal{C}_2}\P_{\bm x}\left(\B^\nu_{T_1}\in dz\right)\E_z\left[ F\left(\int_{0}^{+\infty}e^{-\ps{\Y^1_t,e_1}}Z^1_{t}dt,\int_{0}^{+\infty}e^{-\ps{\Y^1_t,e_2}}Z^2_{t}dt\right)\vert\mathcal{E}_1\right]
	\end{align*}
	where we have used stationarity of the process $(Z^1,Z^2)$.
	To conclude for the proof of Proposition~\ref{prop:asym_prob_A2} it remains to check that the expectation that appears in the integral becomes independent of $z$ in the limit where $\bm x$ diverges inside $\mathcal{C}$. More precisely we prove that
	\begin{lemma}\label{lemma:J_A2}
		As $\bm x\to\infty$ inside $\mathcal{C}$, for any $z\in \partial\mathcal{C}_2$
		\[
		\lim\limits_{\bm x\to\infty}\E_{u(\bm x)+z}\left[ F\left(\int_{0}^{+\infty}e^{-\ps{\Y^1_t,e_1}}Z^1_tdt,\int_{0}^{+\infty}e^{-\ps{\Y^1_t,e_2}}Z^2_tdt\right)\right]=\expect{F(J(\ps{\nu,\rho}),J'(\ps{\nu,e_2})}
		\]
		where $J$ and $J'$ are independent and with law described by Equation~\eqref{eq:def_J}.
	\end{lemma} 
	Assuming for now that such a statement does indeed hold, we see recovering terms that 
	\[
	\E_{\bm x}\left[F(\mathrm  J_1,\mathrm  J_2)\vert\mathcal{E}_1\right]\sim \expect{F(J(\ps{\nu,\rho}),J'(\ps{\nu,e_2})},
	\]
	and $\E_{\bm x}\left[F(\mathrm  J_1,\mathrm  J_2)\vert\mathcal{E}_1\right]\sim\E_{\bm x}\left[F(\mathrm  J_1(\Y^1),\mathrm  J_2(\Y^1))\right]$.
	As a consequence using Proposition~\ref{prop:asym_prob_A2} we can write that
	\[
	\E_{\bm x}\left[F(\mathrm  J_1,\mathrm  J_2)\right]=\P_{\bm x}(\mathcal{E}_1)\expect{F(J(\ps{\nu,\rho}),J'(\ps{\nu,e_2})}+\P_{\bm x}(\mathcal{E}_2)\expect{F(J(\ps{\nu,\rho}),J'(\ps{\nu,e_2})}+o(1),
	\]
	where besides $\P_{\bm x}(\mathcal{E}_1)\to1$ and $\P_{\bm x}(\mathcal{E}_2)\to 0$ in the asymptotic where $\ps{s_1s_2\nu-s_2s_1\nu,\bm x}\to+\infty$. Therefore in this asymptotic we both have that 
	\[
	\E_{\bm x}\left[F(\mathrm  J_1,\mathrm  J_2)\right]\sim \expect{F(J(\ps{\nu,\rho}),J'(\ps{\nu,e_2})}\quad\text{and}\quad \E_{\bm x}\left[F(\mathrm  J_1,\mathrm  J_2)\right]\sim \expect{F(\mathrm J_2(\Y^1),\mathrm J_2(\Y^1)},
	\]
	concluding for the proof of Proposition~\ref{prop:asymptot_A2} and~\ref{prop:asymptot_J} for the $A_2$ case.
	\begin{proof}[Proof of Lemma~\ref{lemma:J_A2}]
		Let us split the integrals involved as
		\begin{align*}
			\int_{0}^{T_1'}e^{-\ps{\Y^1_t,e_1}}Z^1_tdt+\int_{T_1'}^{+\infty}e^{-\ps{\Y^1_t,e_1}}Z^1_tdt,\quad
			\int_{0}^{T_2}e^{-\ps{\Y^1_t,e_2}}Z^2_tdt+\int_{T_2}^{+\infty}e^{-\ps{\Y^1_t,e_2}}Z^2_tdt.
		\end{align*}
		Like before the integrals $\int_{T_1'}^{+\infty}e^{-\ps{\Y^1_t,e_1}}Z^1_tdt$ and $\int_{0}^{T_2}e^{-\ps{\Y^1_t,e_2}}Z^2_tdt$ will become negligible in the $\bm x\to\infty$ limit, whence by the Markov property for $\Y^1$
		\begin{align*}
			&\E_z\left[ F\left(\int_{0}^{+\infty}e^{-\ps{\Y^1_t,e_1}}Z^1_tdt,\int_{0}^{+\infty}e^{-\ps{\Y^1_t,e_2}}Z^2_tdt\right)\right]\\
			&=\int_{v+\partial\mathcal{C}_1}\mathbb{P}_z\left(\Y^1_{T_2}\in dz'\right)\E_{z,z'}\left[ F\left(\int_{0}^{T_1'}e^{-\ps{\Y^1_t,e_1}}Z^1_tdt,\int_{0}^{+\infty}e^{-\ps{\tilde{\Y}^1_t,e_2}}Z^2_{t+T_2}dt\right)\right]+o(1)
		\end{align*}
		where $v=\eta\ps{\bm x,e_2}\omega_2$ and under $\E_{z,z'}$ the processes $\Y^1$, $\tilde{\Y}^1$ are independent and started respectively from $z$ and $z'$. We have seen before that after $T_2$, the $e_1$ component of the process $\Y^1$ was with high probability very large and therefore that the $e_2$ component of the process for $t>T_2$ behaves like a one-dimensional Brownian motion with negative drift $-\ps{\nu,e_2}$ and variance $2$ upon hitting the origin and positive drift $\ps{\nu,e_2}$ conditioned to stay positive after having hit $0$, and only depends of $z'$ through $\ps{z',e_2}=\eta\ps{\bm x,e_2}$. The same applies for the $e_1$ component of the process before time $T_1'$. Moreover we have already seen that $T_2-T_1'\to\infty$ almost surely, whence the processes $(Z^1_t)_{t\leq T_1}$ and $(Z^2_t)_{t\geq T_2}$ decorrelate as $\bm x\to\infty$.
		As a consequence and based on~\cite[Chapter 5-Proposition 2.18]{KS} we see that by stationarity of $Z$ 
		\begin{align*}
			&\E_z\left[ F\left(\int_{0}^{+\infty}e^{-\ps{\Y^1_t,e_1}}Z^1_tdt,\int_{0}^{+\infty}e^{-\ps{\Y^1_t,e_2}}Z^2_tdt\right)\right]\\
			&\sim\int_{v+\partial\mathcal{C}_1}\mathbb{P}_z\left(\Y^1_{T_2}\in dz'\right)\E_{\eta\ps{\bm x,e_2},\eta\ps{\bm x,e_2}}\left[ F\left(\int_{0}^{T_1'}e^{-B_t^{\ps{\nu,\rho}}}Z^1_tdt,\int_{0}^{+\infty}e^{-\tilde{B}^{\ps{\nu,e_2}}_t}Z^2_tdt\right)\right]+o(1)
		\end{align*} 
		which is independent of $z$ and converges to $\expect{F(J(\ps{\nu,\rho}),J'(\ps{\nu,e_2})}$ as desired.
	\end{proof}
	
	
	\subsection{The general case}\label{subsec:proof_prop}
	In the general case the main ideas remain unchanged and the arguments are the same. Therefore in this subsection we mostly shed light on the results that differ from the study of the $A_2$ case. We will consider $s=s_1\cdots s_r$ in the sequel.

	\subsubsection{Decomposition of the path}
	In the same way as above, we see that the we are able to provide another decomposition of the path, based on the analogs of times $T_1,T_1',T_2$. The main difference will be that in the general case we have to split $s$ as 
	\[
	s=(s_1\cdots s_{i_0-1})(s_{i_0}\cdots s_{i_1-1})\cdots (s_{i_{p-1}}\cdots s_r)=:w_0\cdots w_p
	\] 
	where for all $0\leq k\leq p$ and $i_{k-1}\leq j,l\leq i_{k}-1$, the reflections $s_{j}$ and $s_l$ commute (via the convention $i_{-1}=1$ and $i_p=r+1$). In particular we have $s e_j=-w_1\cdots w_p e_j$ for all $j<i_0$.
	
	Namely, we already know from Proposition~\ref{prop:asym_prob} that the event $\mathcal{E}_s$ that the process stays inside the domain $\mathcal{U}_s$ has probability $1-o(1)$. Therefore before reaching $\partial\mathcal{C}$ the process has a drift
	\[
	\nabla\log\partial_{1,\cdots,r}h^\M(z)=\sum_{w\in W_{1,\cdots,r}}\frac{\epsilon(w)\lambda_w e^{\ps{w\nu,z}}}{U(z)}w\nu,\quad U(z)=\sum_{w\in W_{1,\cdots,r}}\epsilon(w)\lambda_w e^{\ps{w\nu,z}},
	\]
	which can be put under the form $s\nu+\mathfrak R_1(z)$ where $\norm{\mathfrak R_1(z)}\leq Ce^{-\eta\ps{\bm x,u}}$ for any $x\in\mathcal{U}_s$ and where $u$ is some vector inside $\mathcal{C}$, up to an event with probability asymptotically $0$. 
	The $e_i$ component of this drift is negative and given by $-\ps{w_1\cdots w_p\nu,e_j}$ for $j<i_0$, but positive equal to $\ps{w_0\cdots w_p\nu,e_j}$ for $j\geq i_0$. We stress that before $T_{\partial\mathcal{C}}$ the $e_j$ components of the path are bounded below for $j\geq i_0$, but this is not the case for $j<i_0$.
	
	After having hit $\partial\mathcal{C}$ (say over $\partial\mathcal{C}_1$) for the first time, the drift of the process will be given by
	$\nabla\log\partial_{2,\cdots,r}h^\M$. Denoting by $\B^\nu_{T_{\partial\mathcal{C}}}$ the location of the process when first hitting $\partial\mathcal{C}$ it is readily seen that it satisfies $\ps{s_2\cdots s_r\nu-s'\nu,\B^\nu_{T_{\partial\mathcal{C}}}}\to+\infty$ for all $s'\in W_{2,\cdots,r}$ as $\bm x\to\infty$ with probability tending to $1$. Indeed let us set $r_1$ to be equal to $s_1$ if $s'\not\in W_{1,\cdots,r}$ and $\mathrm Id$ if $s'\in W_{1,\cdots,r}$, so that $r_is'\in W_{1,\cdots,r}$. Then we can write that
	\[
	\ps{s_{2}\cdots s_r\nu-s'\nu,\B^\nu_{T_{\partial\mathcal{C}}}}=\ps{s_1\cdots s_r\nu-r_1 s'\nu,\B^\nu_{T_{\partial\mathcal{C}}}},
	\]
	since $s_1\B^\nu_{T_{\partial\mathcal{C}}}=\B^\nu_{T_{\partial\mathcal{C}}}$. We can further write 
	\[
	\ps{s_{2}\cdots s_r\nu-s'\nu,\B^\nu_{T_{\partial\mathcal{C}}}}=\ps{s\nu-r_1 s'\nu,\bm x}+\ps{s\nu-r_1 s'\nu,\B^\nu_{T_{\partial\mathcal{C}}}-\bm x},
	\]
	where the first term diverges to $+\infty$ since $r_1s'\in W_{1,\cdots,r}$, while the second one is positive (since $\B^\nu_{T_{\partial\mathcal{C}}}-\bm x$ stays between $s\nu t_0\pm C\sqrt{t_0}^{1+\eps}$ with probability $1-o(1)$).
	
	Therefore we can proceed in the same way after $T_{\partial\mathcal{C}}$: up to an event of probability $1-o(1)$, before hitting $\partial\mathcal{C}$ for the second time the process will have a drift given by \[
	\nabla\log\partial_{2,\cdots,r}h^\M(z)=\frac{s_2\cdots s_r\nu-s_1\cdots s_r\nu e^{-\ps{s_2\cdot s_r\nu,e_1}\ps{z,\nu}}}{1-e^{-\ps{s_2\cdot s_r\nu,e_1}\ps{z,\nu}}}+\mathfrak R_2(z)
	\]
	with $\norm{\mathfrak R_2(z)}\leq Ce^{-\eta\ps{\bm x,u}}$ close to $\B^\nu_{T_{\partial\mathcal{C}}}$. Now for $j< i_0$ the $e_j$ component of the drift is negative and given by $\ps{s_2\cdots s_r\nu,e_j}<0$ but is still positive for $j\geq i_0$ as soon as $i_0>2$.
	Therefore before hitting $\partial\mathcal{C}$ for the second time and under the assumption that $i_0>2$ we see that the event that $\ps{\B^\nu,e_j}>\eta\ps{\bm x,e_j}$ for $j\geq i_0$ has probability asymptotically $1$.
	
	A similar scheme will be carried out until the process has hit all the boundary components $\partial\mathcal{C}_j$ for $j<i_0$. After this event there will be a time $T_1'$ such that after $T_1'$ the $e_j$ components of the path will be bounded below by $\eta\ps{\bm x,e_j}$ for $j<i_0$. In addition the time $T_2=\inf\limits_{t\geq 0}\left\{\ps{\B^\nu_t,e_j}<\eta\ps{\bm x,e_j}\text{ for some }j\geq i_0\right\}$ will be such that $T_2-T_1'\to+\infty$, and $T_2$ the drift of the process is given by $w_1\cdots w_p\nu$. The components of the drift corresponding to $j\geq i_1$ or $j<i_0$ will remain bounded below but that corresponding to $i_0\leq j < i_1$ will not be necessary bounded below. The process will then hit all the components of $\partial\mathcal{C}$ of the form $\partial\mathcal{C}_j$ for $i_0\leq j< i_1$, and after having hit all these parts of $\partial\mathcal{C}$ there will be a time $T_2'$ such that after $T_2'$ the components associated to the roots $e_j$ for $i_0\leq j<i_{1}$ are bounded below, and $T_3-T_2'\to+\infty$ with $T_3=\inf\limits_{t\geq 0}\left\{\ps{\B^\nu_t,e_j}<\eta\ps{\bm x,e_j}\text{ for some }j\geq i_1\right\}$. The same scheme will be repeated until the process has hit all the boundary components; after that the process will have drift $\nu$ conditioned to stay inside $\mathcal{C}$.
	
	In a nutshell, we see that the path can be divided between times $T_1<T_1'<T_2<\cdots<T_p<T_p'=+\infty$ such that
	\begin{itemize}
		\item Between $T_{k}$ and $T_{k}'$, the process will contribute only to the integrals $\mathrm{J}_j$ where $j$ ranges over $\{i_{k-1},\cdots,i_k-1\}$, since the $e_j$ components of the process for $j\not\in\{i_{k-1},\cdots,i_k-1\}$ will be bounded below by some constants of the form $\eta\ps{\bm x,e_j}$. Between $T_{k}'$ and $T_{k+1}$ all the $e_j$ components of the path are bounded below by $\eta\ps{\bm x,e_j}$
		\item For such $j\in\{i_{k-1},\cdots,i_k-1\}$, the $e_j$ component of the process $\B^\nu$ between $T_k$ and $T_{k}'$ will behave like a Brownian motion with negative drift $-\ps{w_{k+1}\cdots w_p\nu,e_i}$ and variance $\ps{e_i,e_i}$ upon hitting the origin, where it will be reflected and will have the law of a one-dimensional Brownian motion with positive drift $\ps{w_{k+1}\cdots w_p\nu,e_i}$ and conditioned to stay positive. The $e_j$ components for $j\not\in\{i_{k-1},\cdots,i_k-1\}$ will have a positive drift.
		\item With probability tending to $1$ as $\bm x\to\infty$, all the time increments $T_{k}'-T_k$ and $T_{k+1}-T_k'$ will diverge to $+\infty$.
	\end{itemize}
	
	\subsubsection{Proof of Proposition~\ref{prop:asymptot_J}}
	We are now ready to address the proof of Proposition~\ref{prop:asymptot_J}. Based on the above decomposition for the path, we can split the integrals involved, for $1\leq j\leq r$, as:
	\[
	\mathrm{J}_j=\int_{T_{k}}^{T_{k}'}e^{-\ps{\B^\nu_t,e_j}}Z_t^j dt + \int_{[0,T_{k})\cup(T_k',+\infty)}e^{-\ps{\B^\nu_t,e_j}}Z_t^j dt
	\]
	where $k$ is such that $i_{k-1}\leq j\leq i_k-1$. 
	Like in the rank two case the second integral will vanish in the limit since for $t\not\in[T_{k-1},T_k']$ the $e_j$ component of the process can be bounded below by $\eta\ps{\bm x,e_j}$.
	
	Now between time $T_k$ and $T_k'$, the $e_j$ components of the process $\B^\nu$ for $i_{k-1}\leq j\leq i_k-1$ can be approximated by a Brownian motion with drift $-\ps{w_{k+1}\cdots w_p,e_j}$ and variance $\ps{e_j,e_j}$ (joined with its reflection on the origin). Moreover since for such $j,l$ we have $\ps{e_j,e_l}=0$ we know that these Brownian motions are independent.
	Likewise, because $T_{k+1}-T_k'$ diverges to $+\infty$ with probability asymptotically $1$, we can use the Markov property in the same way as in the $A_2$ case to see that the integrals that run over distinct intervals decorrelate in the limit. 
	With these two decorrelations at hand we see that in the end for any bounded continuous functions $F_j:\R\to\R$
	\[
	\E_{\bm x}\left[\prod_{j=1}^rF_j\left(\mathrm J_j\right)\right]\sim \prod_{j=1}^r\E_{x_j}\left[F_j\left(\int_{0}^{+\infty}e^{-B^j_t}Z_t^j dt\right)\right],
	\]
	where $x_j$ is some positive number that diverges to $+\infty$ as $\bm x\to\infty$, and where $B^j$ is a Brownian motion with variance $\ps{e_j,e_j}$, and negative drift $-\ps{w_{k+1}\cdots w_p,e_j}$ upon hitting the origin and positive drift  $\ps{w_{k+1}\cdots w_p,e_j}$ conditioned to stay positive after. The latter does indeed converge towards
	\[
	\prod_{j=1}^r\E_{+\infty}\left[F_j\left(\int_{0}^{+\infty}e^{-B^j_t}Z_t^j dt\right)\right],
	\]
	concluding the proof of Proposition~\ref{prop:asymptot_J}.
	


	\section{Asymptotic expansions: class one Whittaker functions, Gaussian Multiplicative Chaos measures and Toda Vertex Operators}\label{sec:whittaker}
	In the two previous sections we have described a path decomposition for a drifted Brownian motion $B^\nu$ over an Euclidean space and investigated the behaviour of this process in the asymptotic where its generalized minimum $\M$ diverges inside the negative Weyl chamber $\mathcal{C}_-$. Based on the properties of the process thus conditioned we infer in this section the asymptotic expansions of some probabilistic objects defined from $B^\nu$: class one Whittaker functions, Gaussian Multiplicative Chaos measures and Toda Vertex Operators.
	
	As explained in the introduction, class one Whittaker functions as well as Toda CFTs can be constructed from semisimple and complex Lie algebras $\mathfrak{g}$, so that the framework presented in Subsection~\ref{subsection_background} naturally arises in this context. Namely to $\mathfrak{g}$ is naturally attached an Euclidean vector space $\V$ equipped with a scalar product $\ps{\cdot,\cdot}$. This space corresponds to $\mathfrak{a}$, the real part of the Cartan subalgebra of $\mathfrak{g}$, while the scalar product is inherited from the Killing form of $\mathfrak{g}$. One way to realize it is to take $\V=\R^r$ equipped with its standard scalar product, where $r$ is the rank of $\mathfrak{g}$, the reflection group $W$ that acts on $\V$ being generated by the reflections with respect to the simple roots $(e_1,\cdots,e_r)$ of $\mathfrak{g}$. We represent them as elements in $\R^r$, normalized by assuming that the longest roots have norm $\sqrt2$ and satisfy
	\begin{equation}
		\ps{e_i^\vee,e_j}=A_{i,j}
	\end{equation} 
	with $A$ the Cartan matrix of $\mathfrak{g}$. For future convenience we set $e_i^*\coloneqq \frac{e_i}{\sqrt{\ps{e_i,e_i}}}\cdot$
	
	\subsection{The probabilistic framework}
	We provide here the definition of class one Whittaker functions, Gaussian Multiplicative Chaos measures and Toda Vertex Operators. To do so we describe the probabilistic environment in which they evolve.
	
	\subsubsection{Class one Whittaker functions}
	Class one Whittaker functions $\Psi_\mu$ enjoy the remarkable feature that they admit a probabilistic representation~\cite[Proposition 5.1]{BOc} when $\mu$ lies in the fundamental Weyl chamber $\mathcal{C}$. This representation takes the form
	\begin{equation}
		\Psi_\mu(\bm x)=b(\mu)e^{\ps{\mu, \bm x}}\expect{\exp\left(-\sum_{i=1}^r\frac{\ps{e_i,e_i}}2e^{-\ps{\bm x,e_i}}\int_0^{\infty}e^{-\ps{B^\mu_t, e_i}}dt\right)}
	\end{equation}
	where $B^\mu$ is a Brownian motion on $\V$ with drift $\mu$, and with the coefficient $b(\mu)$ appearing in this expression equal to
	\begin{equation}
		b(\mu)\coloneqq \prod_{e\in\Phi^+}\Gamma\left(\ps{\mu,e^\vee}\right).
	\end{equation} 
	The above representation of the class one Whittaker functions is very much similar to the probabilistic expression of Toda Vertex Operators. In order to provide a meaningful definition of the latter, we first need to recall some background on Gaussian Free Fields (GFFs) and Gaussian Multiplicative Chaos (GMC hereafter). 
	\subsubsection{Gaussian Free Fields}\label{subsec:GFF}
	GFFs are random distributions which, in the study of the Toda CFTs, are defined over $\C$ and have values in the Euclidean space $\V$. These Gaussian random distributions are characterized by their covariance kernels, which in the present case is given by
	\begin{equation}
		\E[\ps{u,\X(x)}\ps{v,\X(y)}]=\ps{u,v}  G(x,y),\quad x\neq y\in\C
	\end{equation}
	for any $u,v\in\V$, and where $G$ is the Green kernel of the Laplace operator on $\C$ in the metric $\frac{\norm{d^2z}}{\norm{z}_+^4}$:
	\begin{equation}
		G(x,y)=\ln\frac1{\norm{x-y}}+\ln\norm{x}_ + + \ln\norm{y}_+.
	\end{equation}
	Here we denoted $\norm{x}_+\coloneqq\max\left(\norm{x},1\right)$. More details on this object can be found \emph{e.g.} in~\cite{dubedat, She07}.
	
	One remarkable property of the GFF defined above (and which allows to make the connection with Whittaker functions) is its relationship with Brownian motions. Indeed, let us consider the expression of the GFF in polar coordinates and write 
	\begin{equation}\label{eq:rad_ang_dec}
		\X(e^{-t+i\theta})= B_t+Y(t,\theta)
	\end{equation}
	where 
	\[
	B_t\coloneqq \frac{1}{2\pi}\int_0^{2\pi}\X(e^{-t+i\theta})d\theta\quad\text{and}\quad Y(t,\theta)\coloneqq \X(e^{-t+i\theta})- B_t.
	\]
	Then from their covariance kernels we see that the process $(B_t)_{t\geq0}$ is a Brownian motion on $\R^r$ started from $0$  and independent of $Y$. Moreover $Y$ has covariance kernel given by 
	\begin{equation}\label{eq:cov_Y}
		\E[\ps{u,Y(t,\theta)}\ps{v,Y(t',\theta')}]=\ps{u,v}  \ln\frac{e^{-t}\vee e^{-t'}}{\norm{e^{-t+i\theta}-e^{-t'+i\theta'}}}
	\end{equation}
	for any $u,v\in\V$. In the sequel we may consider the restriction of the field to the unit disk $\D$, which corresponds to the domain $t\geq0$.
	
	There is yet another decomposition of the GFF that we may be interested in. It consists in setting
	\begin{equation}\label{equ:decomp_dir}
		\X=\X_0+P\varphi
	\end{equation} 
	where $\varphi$ is the trace of the field $\X$ on $\partial\D$ and $P\varphi$ its harmonic extension inside $\D$. In this decomposition $\X_0$ has the law of a GFF inside $\D$ with Dirichlet boundary conditions and is independent from $\varphi$. Since both $\X$ and $\X_0$ have zero-mean on $\partial\D$, harmonicity of $P\varphi$ inside $\D$ implies that $P\varphi(0)=0$.
	
	\subsubsection{Gaussian Multiplicative Chaos}\label{subsec:GMC}
	The probabilistic interpretation of the path integral defining Toda CFTs requires the introduction of the exponential of the field $\X$. However the latter being highly non-regular a regularization procedure known as Gaussian Multiplicative Chaos~\cite{Kah,RV_GMC} is necessary to give it a rigorous meaning.
	
	To do so, one first introduces a regularization $\X_\eps$ of the GFF by setting
	\begin{equation}\label{regularization}
		\X_\eps\coloneqq \X*\eta_\eps=\int_\C \X(\cdot-z)\eta_\eps(z)d^2z
	\end{equation}
	where $\eta_\eps\coloneqq \frac1{\eps^2}\eta(\frac{\cdot}{\eps})$ is a smooth, compactly supported mollifier. The field $\X_\eps$ thus obtained is now a smooth function, so we can take its exponential and set a random measure over $\D$ via
	\begin{equation}
		M_\eps^{\gamma e_i}(d^2x)\coloneqq \:  e^{  \langle\gamma e_i,  \X_\eps(x) \rangle-\frac12\expect{\langle\gamma e_i, \X_\eps(x) \rangle^2}} d^2x.
	\end{equation}
	The limit as $\eps\rightarrow0$ of this random measure is well-defined and is called a \emph{Gaussian Multiplicative Chaos measure}. More precisely, assume that $\gamma<\sqrt 2$\footnote{The assumption that $\gamma<\sqrt 2$ is the optimal one to ensure that all the GMC measures for $1\leq i\leq r$ are well-defined (at least in the subcritical regime). The fact that we need to assume that $\gamma<\sqrt 2$ instead of the usual bound $\gamma<2$ stems from the fact that the longest roots have squared norm $2$. To recover the usual range of values one would need to rescale $\gamma$ by a multiplicative factor $\sqrt 2$ to take into account the fact that some roots are not normalized.}. Then for $1\leq i\leq r$, within the space of Radon measures equipped with the weak topology, the following convergence holds in probability
	\begin{equation}
		M^{\gamma e_i}(d^2x)\coloneqq \underset{\eps \to 0}{\lim} M_\eps^{\gamma e_i}(d^2x),
	\end{equation} 
	where  $M^{\gamma e_i}$ is a non-trivial random measure~\cite{Ber,Sha}.
	
	The polar decomposition~\eqref{eq:rad_ang_dec} of the field $\X$ provides an alternative way of defining the GMC measure defined above. Namely the measure $M^{\gamma e_i}$ can be put under the form
	\begin{equation}\label{equ:radial_angular}
		M^{\gamma e_i}(d^2x)=e^{\gamma\ps{B_t-Qt,e_i}} M_Y^{\gamma e_i}(dt,d\theta)
	\end{equation}
	with $M_Y$ the GMC measure defined from $Y$. For $1\leq i\leq r$, we further introduce the measure which corresponds to the average of $M_Y^{\gamma e_i}$ over the circle of radius $e^{-t}$,
	\begin{equation}\label{equ:radial_angular_Z}
		Z_t^i\coloneqq \int_{0}^{2\pi}M_Y^{\gamma e_i}(t,\theta)d\theta.
	\end{equation}
	From~\cite{RV_GMC} we recall that the integral of $Z$ is a well-defined random variable with finite moments up to order $p<\frac{2}{\gamma^2}$:
	\begin{lemma}\label{lemma:moments_Z}
		For $i=1,\cdots,r$ and any bounded and non-trivial interval $I$
		\begin{equation}
			\expect{\left(\int_I Z_t^i\right)^p}<\infty
		\end{equation}
		for $-\infty<p<\frac{2}{\gamma^2}$.
	\end{lemma}
	
	\subsubsection{Toda Vertex Operators}
	The parameter $\gamma\in(0,\sqrt 2)$ that enters the definition of the GMC corresponds to the so-called \emph{coupling constant} of Toda CFTs.  Toda CFTs depend as well on the so-called cosmological constants $\mu_1,\cdots,\mu_r>0$. The \emph{background charge} $Q$ is also key in their studies; it is an element of $\mathfrak a$ defined as
	\begin{equation}\label{eq:background_charge}
		Q\coloneqq \gamma\rho+\frac2\gamma\rho^\vee
	\end{equation}
	where recall that $\rho$ is the Weyl vector of $\mathfrak g$, and $\rho^\vee$ is defined as the half-sum of \emph{dual} positive roots. For future convenience we also consider the reflection group generated by the reflections associated to the simple roots but centered at $Q$. More precisely for $s\in W$ we set  
	\begin{equation}\label{eq:hat_s}
		\hat s(\alpha)\coloneqq Q+s(\alpha-Q).
	\end{equation}
	Within the probabilistic framework of Toda CFT, the primary fields admit a probabilistic expression involving the GMC measures introduced above. Namely recall that in agreement with Equation~\eqref{equ:decomp_dir} we can write the GFF $\X$ under the form $\X=\X_0+P\varphi$. Based on this writing we can define Toda primary fields for $\alpha$ belonging to the shifted Weyl chamber $Q+\mathcal{C}_-$ as the functionals
	\begin{equation}\label{equ:psi_alpha}
		\psi_{\alpha}(\bm c,\varphi)\coloneqq e^{\ps{\alpha-Q,\bm c}}\E_{\varphi}\left[\exp\left(-\sum_{i=1}^r\mu_ie^{\gamma \ps{\bm c,e_i}}I_i\right)\right]
	\end{equation}
	where $\bm c$ is any vector in $\V$ while $I_i$ is given by the GMC measure with respect to the GFF $\X=\X_0+P\varphi$
	\begin{equation}
		I_i\coloneqq\int_{\mathbb D}\norm{x}^{-\gamma\ps{\alpha,e_i}}M^{\gamma e_i}(d^2x).
	\end{equation}
	Finally $\E_{\varphi}$ is the conditional expectation with respect to $\varphi$ (thus over $\X_0$) 
	From the polar decomposition~\eqref{equ:radial_angular} the above integral can be put in the form:\begin{equation}\label{equ:int_radial_angular}
		\int_{\mathbb D}\norm{x}^{-\gamma\ps{\alpha,e_i}}M^{\gamma e_i}(d^2x)=\int_0^{\infty}e^{\gamma\ps{B^\nu_t,e_i}}Z_t^idt,
	\end{equation}
	where $B^\nu$ is a Brownian motion on $\R^r$ with drift $\nu=\alpha-Q\in\mathcal{C}_-$.
	For future convenience we set $\frac{1}{\gamma}\ln\bm\mu\coloneqq\frac1\gamma\sum_{i=1}^r\ln(\mu_i)\omega_i^\vee$. 
	
	
	\subsection{Statement of the results}
	Thanks to these probabilistic representations and based on our generalized path decomposition, we are able to provide an asymptotic expansion of class one Whittaker functions (resp. Toda Vertex Operators) in the limit where $\bm x\to\infty$ (resp. $\bm c\to\infty$) inside the Weyl chamber $\mathcal{C}$ (resp. $\mathcal{C}_-$). 
	
	Our first statement is concerned with Whittaker functions. In particular it is reminiscent of the invariance of the Whittaker functions under the action of the Weyl group: $\Psi_{\mu}=\Psi_{s\mu}$ for $s\in W$.
	\begin{theorem}\label{thm:whittaker_expansion}
		Assume that $\mu\in\mathcal{C}$ is such that $\ps{s\mu-\mu,\omega_i^\vee}>-1$ for all $1\leq i\leq r$ and $s\in W$. 
		Then, in the limit where $\bm x\to\infty$ inside $\mathcal{C}$ with $\ps{s_1s_2\mu-s'\mu,\bm x}\to+\infty$ for all $s'\neq s_1s_2\in W$ with length greater than or equal to two, the class one Whittaker functions admit the asymptotic expansion
		\begin{equation}
			\Psi_\mu(\bm x)=b(\mu)e^{\ps{\mu,\bm x}}+\sum_{i=1}^r b(s_i\mu)e^{\ps{s_i\mu,\bm x}} + b(s_1s_2\mu)e^{\ps{s_1s_2\mu,\bm x}} + u 
		\end{equation}
		where $e^{-\ps{s_1s_2\mu,\bm x}}u\to 0$.
	\end{theorem}
	In what follows and in analogy with Toda Vertex Operators we will call the numbers $r_{s}(\alpha)\coloneqq \frac{b(s\mu)}{b(\mu)}$ Whittaker reflection coefficients.
	There is a counterpart statement for Toda Vertex Operators. This expansion allows to make sense of Toda reflection coefficients, which are key in the understanding of these theories as we will explain along the proof of Theorem~\ref{thm:main_result}.
	\begin{theorem}\label{thm:primary_reflection}
		Assume that $\alpha-Q\in\mathcal{C}_-$ satisfies $\ps{\hat s\alpha-\alpha,\omega_i^\vee}<\gamma$ for all $1\leq i\leq r$ and $s\in W$ (where recall the notation $\hat s$ from Equation~\eqref{eq:hat_s}).
		Then there exist non-zero real numbers $R_s(\alpha)$ such that the Vertex Operators admit the asymptotic expansion
		\begin{equation}\label{equ:primary_reflection1}
			\psi_\alpha(\bm c,\varphi)=e^{\ps{\alpha-Q,\bm c}}+\sum_{i=1}^rR_{s_i}e^{\ps{s_i(\alpha-Q),\bm c}}+R_{s_1s_2}(\alpha)e^{\ps{s_1s_2(\alpha-Q),\bm c}}+u
		\end{equation}
		in the asymptotic where $\ps{\hat s_1\hat s_2\alpha-\hat s'\alpha,\bm c}\to +\infty$ for all $s'\neq s_1s_2\in W$ with length greater than or equal to two, and where $u$ is such that $e^{-\ps{s_1s_2(\alpha-Q),\bm c}}u\to 0$.
	\end{theorem}
	In both statements, the assumption that $\ps{\hat s_1\hat s_2\alpha-\hat s'\alpha,\bm c}\to +\infty$ for all $s'\neq s_1s_2\in W$ simply amounts to selecting the term corresponding to $e^{\ps{s_1s_2(\alpha-Q),\bm c}}$ as the last term in the expansion, and one could replace $s_1s_2$ by any $s\in W$ with length two and get a similar statement. The first assumption corresponds to the range of values for which the probabilistic interpretation of the $b$ function or reflection coefficients makes sense.
	
	The proof of Theorem~\ref{thm:primary_reflection} relies on the fact that Toda reflection coefficients can be obtained via a limiting procedure. More precisely, we will see that as $\bm c\to\infty$ inside $\mathcal{C}_-$ with $\ps{\hat s_i\hat s_j\alpha-\hat s_j\hat s_i\alpha,\bm c}\to+\infty$:
	\begin{align*}
		& \E_{\varphi}\left[\exp\left(-e^{\gamma \ps{\bm c,e_i}}I_i\right)-1\right]\sim e^{\ps{\hat s_i\alpha-\alpha,\bm c}}R_{s_i}(\alpha)\\
		&\E_{\varphi}\left[\left(\exp\left(-e^{\gamma \ps{\bm c,e_i}}I_i\right)-1\right)\left(\exp\left(-e^{\gamma \ps{\bm c,e_j}}I_j\right)-1\right)\right]\sim e^{\ps{\hat s_i\hat s_j\alpha-\alpha,\bm c}}R_{s_is_j}(\alpha).
	\end{align*}
	This property extends to other elements of the Weyl group as well, and allows to make sense of Toda reflection coefficients for $s$ belonging to a subset of the Weyl group. Namely, let us introduce the subset $W_{1,\cdots,r}$ of $W$ defined as the set of the $s\in W$ that admit a reduced expression containing all the reflections $ s_{1},\cdots, s_{r}$.
	\begin{theorem}\label{thm:expression_refl}
		Assume that $\alpha-Q\in\mathcal{C}_-$ satisfies $\ps{\hat s\alpha-\alpha,\omega_i^\vee}<\gamma$ for all $1\leq i\leq r$ and $s\in W_{1,\cdots,r}$. Further assume that there exists $s\in W_{1,\cdots, r}$ with length $r$ such that $\bm c\to\infty$ with $\ps{\hat s\alpha-\hat s'\alpha,\bm c}\to+\infty$ for all $s'\neq s\in W_{1,\cdots,r}$. 
		Then there exists a non-zero real number $R_s(\alpha)$ such that
		\begin{equation}
			\E_{\varphi}\left[\prod_{k=1}^p\left(\exp\left(-e^{\gamma \ps{\bm c,e_{i_k}}}I_{i_k}\right)-1\right)\right]\sim e^{\ps{\hat s\alpha-\alpha,\bm c}}R_{s}(\alpha).
		\end{equation}
		$R_s(\alpha)$ is a \emph{Toda reflection coefficient}, and is equal to
		\begin{equation}\label{eq:expression_refl_1}
			\begin{split}
				R_s(\alpha)&=\epsilon(s)\frac{A\left(s(\alpha-Q)\right)}{A(\alpha-Q)},\quad\text{where}\\
				A(\alpha)&=\prod_{i=1}^r\left(\mu_i\pi l\left(\frac{\gamma^2\ps{e_i,e_i}}{4}\right)\right)^{\frac{\ps{\alpha,\omega_i^\vee}}\gamma}\prod_{e\in\Phi^+}\Gamma\left(1-\frac{\gamma}2\ps{\alpha,e}\right)\Gamma\left(1-\frac{1}\gamma\ps{\alpha,e^\vee}\right)\cdot
			\end{split}
		\end{equation}
	\end{theorem}
	This statement allows to make sense of Toda reflection coefficients when $s=s_{\sigma 1}\cdots s_{\sigma r}$ for $\sigma$ a permutation of $\{1,\cdots,r\}$.
	The value of the Toda reflection coefficients thus computed is in agreement with predictions from the physics literature~\cite{reflection_simplylaced, refl_non_simply_laced, Fat_refl}. The feature that Toda reflection coefficients arise in the asymptotic expansion of the Vertex Operators somewhat appears in~\cite{reflection_simplylaced, refl_non_simply_laced}.
	We stress the remarkable property that Liouville reflection coefficients $R_L(\alpha)$ (see Equation~\eqref{equ:refl_lio} below) can be recovered from Toda reflection coefficients via the identification
		\[
		R_{s_i,\gamma}(\alpha)=R_{L,\tilde\gamma}\left(\ps{\alpha,e_i}\sqrt{\frac{\ps{e_i,e_i}}{2}}\right),\quad\tilde\gamma=\sqrt{\frac{2}{\ps{e_i,e_i}}}\gamma
		\]
		where with these notations we have stressed the dependence in the coupling 
		constant.
		Moreover Toda reflection coefficients can be computed recursively thanks to the noticeable identity, valid for any $s$ and $\tau$ in $W$:
		\[
		R_{s\tau}(\alpha)=R_{s}(\tau\alpha)R_{\tau}(\alpha),
		\]
		which in particular allows to reduce the computation of Toda reflection coefficients to that of Liouville. Using this relation one checks that in the cases we consider we have an alternative representation of the Toda reflection coefficients:
	\begin{equation}\label{equ:expression_toda_refl_bis}
		\begin{split}
			R_s(\alpha)=\epsilon(s)\prod_{i=1}^r\left(\mu_i\pi l\left(\frac{\gamma^2\ps{e_i,e_i}}{4}\right)\right)^{\frac{\ps{\hat s\alpha-\alpha,\omega_i^\vee}}\gamma}\frac{\Gamma\left(1-\frac{\gamma}2\ps{\hat s\alpha-\alpha,\omega_i}\right)\Gamma\left(1-\frac{1}\gamma\ps{\hat s\alpha-\alpha,\omega_i^\vee}\right)}{\Gamma\left(1+\frac{\gamma}2\ps{\hat s\alpha-\alpha,\omega_i}\right)\Gamma\left(1+\frac{1}\gamma\ps{\hat s\alpha-\alpha,\omega_i^\vee}\right)}
		\end{split}
	\end{equation}
	where the $\omega_i^\vee$ are defined via $\ps{\omega_i^\vee,e_j}=\delta_{ij}$.
	
	Eventually, we show that these reflection coefficients naturally arise in the tail expansion of the GMC measures. To this end, let us introduce the \emph{unit volume reflection coefficient} $\overline{R_s}(\alpha)$ by setting $\mu_i=1$ for all $1\leq i\leq r$ and
	\begin{equation}\label{equ:unit_vol_refl2}
		\begin{split}
			\overline{R_s}(\alpha)\coloneqq \epsilon(s)\prod_{i=1}^r\Gamma\left(1-\frac1\gamma\ps{\hat s\alpha-\alpha,\omega_i^\vee}\right)^{-1}R_s(\alpha).
		\end{split}
	\end{equation}
	These quantities allow to describe the tail expansion of the GMC measures: 
	\begin{theorem}\label{thm:tail_expansion_refl}
		Under the same assumptions as in Theorem~\ref{thm:expression_refl},
		\begin{equation}
			\P\left(\int_{\mathbb D} \norm{x}^{-\gamma\ps{\alpha,e_{i_k}}}M^{\gamma e_i}(d^2x)>e^{-\gamma \ps{\bm c,e_{i_k}}},\quad k=1,\cdots,p\right) \sim \overline{R_s}(\alpha)e^{\ps{\hat s\alpha-\alpha,\bm c}}.
		\end{equation}
	\end{theorem}
	This extends the analog statement of~\cite[Section 7]{KRV_DOZZ} which corresponds to the case $r=1$ ---see also the works~\cite{RV_tail,MDW_tail} for more general results in the $r=1$ case.
	
	The rest of this Section is dedicated to proving such asymptotic expansions.
	
	
	\subsection{Reflection coefficients and path decompositions}\label{subsec:refl_J}
	To start with, we collect some important facts about these reflection coefficients and more precisely on their probabilistic interpretation.
	
	\subsubsection{Brownian motions conditioned to stay positive}\label{subsec:brown_cond_neg}
	In the definition of Toda and Whittaker reflection coefficients, we will need to consider a certain diffusion process defined thanks to the one-dimensional path decomposition by Williams~\cite{Williams}. Namely for positive $\nu$, we introduce the process $\mathcal{B}^\nu$ started from $x>0$ to be the junction of:
	\begin{itemize}
		\item a one-dimensional Brownian motion with negative drift $-\nu$, upon hitting $0$;
		\item an independent one-dimensional Brownian motion with positive drift $\nu$, conditioned to stay positive. This process is a diffusion with generator
		\[
		\frac12\frac{d^2}{dx^2}+\nu\text{coth}(\nu x)\frac{d}{dx}\cdot
		\]
	\end{itemize}
	Whittaker reflection coefficients will be expressed in terms of the random variable $J(\nu)$, where for positive $\nu$ we have set
	\begin{equation}\label{eq:def_J_nu}
		J(\nu)\coloneqq \int_{0}^{+\infty}e^{-\mathcal B^\nu_t} dt.
	\end{equation}
	In the above equation the process $\mathcal{B}^\nu$ is (formally) started from $+\infty$. A proper way to make sense of the random variable thus defined is via a limiting procedure by setting for any positive, bounded and continuous map $F:\R^+\to\R^+$
	\[
		\expect{F\left(J(\nu)\right)}\coloneqq \lim\limits_{x\to+\infty}\expect{F\left(\int_{0}^{+\infty}e^{-\mathcal B^\nu_t}dt\right)}.
	\]
	Alternatively one can use that $J(\nu)$ has same law as
	\[
		\int_{-\infty}^{+\infty}e^{-\tilde{\mathcal B}^\nu_t} dt
	\]
	where $(\tilde{\mathcal B}^\nu_t)_{t\geq 0}$ and $(\tilde{\mathcal B}^\nu_{-t})_{t\geq 0}$ are two independent Brownian motions with positive drift $\nu$ and conditioned to stay positive (this follows from the time-reversal property of such a process~\cite[Theorem 2.5]{Williams}, see for instance~\cite[Section 3]{KRV_DOZZ} for more details).
	
	Likewise, Toda reflection coefficients will be defined using the random variable $J_\gamma(-\nu)$, where for positive $\nu$
	\begin{equation}\label{eq:def_J_gamma}
		J_\gamma(-\nu)\coloneqq \int_{0}^{+\infty}e^{-\gamma \mathcal B^\nu_t}Z_t dt.
	\end{equation}
	Here $Z_tdt$ is ithe random measure independent of $\mathcal{B}^\nu$ and whose law is that of $\int_0^{2\pi}M^{\gamma}_Y(t,\theta)d\theta$ with $Y$ the Gaussian field from Equation~\eqref{eq:cov_Y} in the case where $r=1$. Like before this random variable is properly defined via
	\[
	\expect{F\left(J_\gamma(-\nu)\right)}\coloneqq \lim\limits_{x\to+\infty}\expect{F\left(\int_{0}^{+\infty}e^{-\gamma \mathcal B^\nu_t}Z_t dt\right)},
	\]
	or equivalently by noting that, by stationarity of the process $Z_t$, $J_\gamma(-\nu)$ has same law as
	\[
	\int_{-\infty}^{+\infty}e^{-\gamma \tilde{\mathcal B}^\nu_t}Z_t dt
	\]
	with $\tilde{\mathcal B}^\nu_t)_{t\in\R}$ as above. We recall the statement of~\cite[Lemma 3.3]{KRV_DOZZ} which provides an analog result of Lemma~\ref{lemma:moments_Z}.
	\begin{lemma}\label{lemma:moments_cond}
		For any $\nu\in(-\frac2\gamma,0)$ and $-\infty<p<\frac{4}{\gamma^2}$,
		\begin{equation}
			\expect{J_\gamma(\nu)^p}<\infty.
		\end{equation}
	\end{lemma}
	In the above since $-\nu$ is assumed to be positive the notation $J_\gamma(\nu)$ simply corresponds to taking $J_\gamma\Big(-(-\nu)\Big)$ in Equation~\eqref{eq:def_J_gamma}.
	
	\subsubsection{Liouville reflection coefficient}
	In $\mathfrak{sl}_2$ Toda CFT, \emph{i.e.} Liouville theory, the reflection coefficient arises in the asymptotic expansion of the Vertex operators. This asymptotic expansion is performed in terms of the constant mode $\bm c$ of the Liouville field. More precisely, it follows from the proof of~\cite[Lemma 7.1]{KRV_DOZZ} (see also~\cite[Proposition 4.1]{BGKRV} where it is proved that such an expansion allows to compute a scattering coefficient associated to the Liouville Hamiltonian introduced in~\cite{GKRV}) that, for $\ps{\alpha-Q,e_1^\vee}<0$ sufficiently small,
	\[
	\psi_\alpha(\bm c;\varphi)=e^{\ps{\alpha-Q,\bm c}}+R_{s_1}(\alpha)e^{\ps{s_1(\alpha-Q),\bm c}}+o\left(e^{\ps{s_1(\alpha-Q),\bm c}}\right)
	\]
	in the asymptotic where $\ps{\bm c,e_1}\rightarrow-\infty$, and with $R_L$ given by the reflection coefficient of Liouville theory. The remainder term satisfies $\lim\limits_{\bm c\to\infty}e^{-\ps{s_1(\alpha-Q),\bm c}}o\left(e^{\ps{s_1(\alpha-Q),\bm c}}\right)=0$ for any fixed $\varphi$.
	The reflection coefficient that arises in such an expansion admits a probabilistic expression involving the random variable $J_\gamma(\nu)$ introduced above. Indeed it was proved in~\cite{KRV_DOZZ} that for $\alpha-Q\in(-\frac2\gamma,0)$, this reflection coefficient can be written under the form
	\begin{equation}\label{eq:refl_expect}
		R_L(\alpha)=mu_L^{\frac{Q-\alpha}\gamma}\frac{Q-\alpha}{\gamma}\Gamma\left(\frac{\alpha-Q}{\gamma}\right)\expect{J_{\sqrt 2\gamma}\left(\frac{\alpha-Q}{\sqrt2}\right)^{\frac{Q-\alpha}{\gamma}}}
	\end{equation}
	where the $\sqrt 2$ terms come from the fact that the simple roots are not normalized, and with $\mu_L$ Liouville cosmological constant.
	One of the achievements of the aforementioned article is the evaluation of this reflection coefficient, which is in agreement with predictions from the physics literature. Namely~\cite[Theorem 3.5]{KRV_DOZZ} shows that
	\begin{equation}\label{equ:refl_lio}
		R_L(\alpha)=-\left(\pi\mu_L l\left(\frac{\gamma^2}{2}\right)\right)^{\frac{Q-\alpha}\gamma}\frac{\Gamma\left(\frac{\alpha-Q}{\gamma}\right)\Gamma\Big(\frac{\gamma}2(\alpha-Q)\Big)}{\Gamma\left(\frac{Q-\alpha}{\gamma}\right)\Gamma\Big(\frac\gamma2(Q-\alpha)\Big)}
	\end{equation}
	for $\alpha-Q\in(-\frac2\gamma,0)$, where $l$ is the special function
	\[
	l(x)\coloneqq\frac{\Gamma(x)}{\Gamma(1-x)}\cdot
	\]
	In particular this allows to evaluate the expectation term as
	\begin{equation}\label{equ:eval_J}
		\expect{J_{\gamma}\left(\nu\right)^{-\frac{2\nu}\gamma}}=\left(\pi l\left(\frac{\gamma^2}{4}\right)\right)^{-\frac{2\nu}\gamma}\frac{\Gamma\Big(1+\frac{\gamma}2\nu\Big)}{\Gamma\left(1-\frac2\gamma\nu\right)\Gamma\Big(1-\frac{\gamma}2\nu\Big)}\cdot
	\end{equation}
	for $\nu\in(-\frac{2}{\gamma},0)$ and $\gamma\in(0,2)$.
	
	\subsubsection{Whittaker reflection coefficients}
	As we will see, a similar picture holds for class one Whittaker functions. Indeed the reflection coefficients in that case will be defined using the following analog of Equation~\eqref{eq:refl_expect}:
	\begin{equation}\label{equ:refl_whitt}
		r_{s_i}(\mu)=-2^{-\ps{\mu,e_i^\vee}}\Gamma(1-\ps{\mu,e_i^\vee})\expect{J\left(\frac12\ps{\mu,e_i^\vee}\right)^{\ps{\mu,e_i^\vee}}}
	\end{equation}
	where recall that $J$ is defined in Equation~\eqref{eq:def_J_nu} by 
		\[
		J(\mu)=\int_0^\infty e^{-\mathcal{B}_t^\mu}dt
		\] from the process $\mathcal{B}^\mu$ started from $-\infty$.
	In the same way as in Equation~\eqref{equ:eval_J}, we are able to evaluate the expectation that appears in the right-hand side:
	\begin{proposition}\label{prop:eval_J}
		For positive $\mu$,
		\begin{equation}\label{equ:identity_refl_whitt0}
			\expect{\left(\int_0^\infty e^{-\mathcal{B}_t^\mu}dt\right)^{2\mu}}=\frac{2^{2\mu}}{\Gamma(1+2\mu)}\cdot
		\end{equation}
	\end{proposition}
	In order to prove this identity we will use the analogy between Toda Vertex Operators and class one Whittaker functions, and more precisely we rely on the strategy developed in~\cite{KRV_DOZZ} to evaluate the Liouville reflection coefficient. Namely we prove that the above expression can be obtained via a limiting procedure:
	\begin{lemma}
		For positive $\eps$, we have that
		\begin{equation}\label{equ:identity_refl_whitt1}
			\lim\limits_{\eps\to 0}\eps\expect{\left(\int_0^\infty e^{-B_t^\mu}dt\right)^{2\mu-\eps}}=\frac{1}{2\mu}\expect{\left(\int_0^\infty e^{-\mathcal{B}_t^\mu}dt\right)^{2\mu}}
		\end{equation}
	where $B^\mu$ is a one-dimensional Brownian motion with drift $\mu$ and started from the origin.
	\end{lemma}
	\begin{proof}
		Using the one-dimensional path decomposition by Williams~\cite{Williams} we see that
		\begin{align*}
			\expect{\left(\int_0^\infty e^{-B_t^\mu}dt\right)^{2\mu-\eps}}&=\frac{1}{2\mu}\int_0^{\infty}e^{-2\mu m}\E_{m}\left[\left(\int_0^\infty e^{m}e^{-\mathcal B_t^\mu}dt\right)^{2\mu-\eps}\right]dm
		\end{align*}
		where under $\E_{m}$ the process $\mathcal{B}^\mu$ is started from $m$. The latter is equal to
		\begin{align*}
			&=\frac{1}{2\mu}\int_0^{\infty}e^{-\eps m}\E_{m}\left[\left(\int_0^\infty e^{-\mathcal B_t^\mu}dt\right)^{2\mu-\eps}\right]dm\\
			&=\frac1\eps\frac{1}{2\mu}\int_0^{\infty}e^{-m}\E_{\frac m\eps}\left[\left(\int_0^\infty e^{-\mathcal B_t^\mu}dt\right)^{2\mu-\eps}\right]dm.
		\end{align*}
		Because $\E_{\frac m\eps}\left[\left(\int_0^\infty e^{-\mathcal B_t^\mu}dt\right)^{2\mu-\eps}\right]\to \expect{\left(\int_0^\infty e^{-\mathcal{B}_t^\mu}dt\right)^{2\mu}}$ for all positive $m$, by dominated convergence
		$\eps\expect{\left(\int_0^\infty e^{-B_t^\mu}dt\right)^{2\mu-\eps}}$ goes to $\frac{1}{2\mu}\expect{\left(\int_0^\infty e^{-\mathcal{B}_t^\mu}dt\right)^{2\mu}}$ in the $\eps\to0$ limit.
	\end{proof}
	Now the quantity that appears on the right-hand side in Equation~\eqref{equ:identity_refl_whitt1} is easily computed since the law of the random variable is completely explicit:
	\begin{lemma}
		For positive $p<2\mu$, 
		\begin{equation}\label{equ:identity_refl_whitt2}
			\expect{\left(\int_0^\infty e^{-B_t^\mu}dt\right)^p}=2^p\frac{\Gamma(2\mu-p)}{\Gamma(2\mu)}\cdot
		\end{equation}
	\end{lemma}
	\begin{proof}
		The law of the whole integral is known to be given by an inverse Gamma law. More precisely it is a classical fact~\cite{D90} that
		\begin{equation}
			\int_0^\infty e^{-B_t^\mu}dt\overset{\text{(law)}}{=}\frac{2}{\gamma_{2\mu}}
		\end{equation}
		where $\gamma_{2\mu}$ follows the Gamma law with index $2\mu$. This implies the result.
	\end{proof}
	Via the identity $z\Gamma(z)=\Gamma(1+z)$, combining Equations~\eqref{equ:identity_refl_whitt1} and~\eqref{equ:identity_refl_whitt2} we see that~\eqref{equ:identity_refl_whitt0} does indeed hold.
	
	
	\subsection{Asymptotic expansion of the Vertex Operators: proof of Theorem~\ref{thm:primary_reflection}}
	Assuming Theorem~\ref{thm:expression_refl} to hold, we provide in this subsection a proof of Theorem~\ref{thm:primary_reflection} and derive a precise asymptotic expansion for reflection coefficients associated with elementary reflections. 
	
	Let us first explain how the path decomposition unveiled in the previous sections can be used to understand the asymptotic behaviour as $\bm c\to\infty$ inside $\mathcal{C}_-$ of the quantity $	\E_{\varphi}\left[\exp\left(-\sum_{i=1}^re^{\gamma \ps{\bm c,e_i}}I_i\right)\right]$.
	We start with a tautology: write that
	\begin{equation*}
		\begin{split}
			\E_{\varphi}\left[\exp\left(-\sum_{i=1}^re^{\gamma \ps{\bm c,e_i}}I_i\right)\right]=\sum_{k=0}^r\sum_{i_1<\cdots<i_k}\E_{\varphi}\left[\prod_{j=1}^k\left(\exp\left(-e^{\gamma \ps{\bm c,e_{i_k}}}I_{i_j}\right)-1\right)\right].
		\end{split}
	\end{equation*}
	In this equation we can rewrite the integrals involved using the radial-angular decomposition from~\eqref{equ:int_radial_angular} as
	\begin{align*}
		I_i=\int_{\mathbb D}\norm{x}^{-\gamma\ps{\alpha,e_i}}M^{\gamma e_i}(d^2x)=\int_0^{\infty}e^{\gamma\ps{B^\nu_t,e_i}}Z_t^idt
	\end{align*}for $i=1,\cdots, r$ and with $\nu=\alpha-Q$. The probability for these integrals to be large is governed by the maximums of the processes $\left(\ps{B^\nu_t,e_i}\right)_{t\geq 0}$ so that the path decomposition of Theorem~\ref{thm:williams} naturally applies within this setting. To be more specific we can rewrite the random variables $I_i$ under the form
	\[
	I_i=e^{\gamma\ps{\M,e_i}}\int_0^{\infty}e^{\gamma\ps{\B^\nu_t,e_i}}Z_t^idt=:e^{\gamma\ps{\M,e_i}}\mathrm J_i
	\]
	where $\M$ has its law described by 
	\[
	\P\left(\M_i\leq \bm m_i\quad\forall1\leq i\leq r\right)=h(\bm m)\mathds1_{\bm m\in\mathcal{C}},
	\]
	while the independent diffusion process $\B^\nu$ that appears there is defined by joining
	\begin{itemize}
		\item A diffusion process $\X^1$ started from $-\M\in\mathcal{C}_-$, with generator $\frac12\Delta+\nabla\log\partial_{1,\cdots,r}h$ and run until it hits $\partial\mathcal{C}_-$, say at $z_1\in\partial\mathcal{C}_1$. 
		\item Then run an independent process $\X^2$ started from $z_1$ and with generator $\frac12\Delta+\nabla\log\partial_{2,\cdots,r}h$, upon hitting $\partial\mathcal{C}_-$.
		\item Thus define a family of processes $(\X^1,\cdots,\X^r)$. When $\X^r$ reaches the boundary of $\partial\mathcal{C}_-$, sample $\X^{r+1}$ with generator $\frac12\Delta+\nabla\log h$.
	\end{itemize}
	In the above $h$ is defined in Equation~\eqref{equ:max_drift}. By doing so we reformulate our problem by writing that 
	\begin{align*}
	\E_{\varphi}\left[\prod_{i=1}^r\left(\exp\left(-e^{\gamma \ps{\bm c,e_{i}}}I_{i}\right)-1\right)\right]=\int_{\mathcal C}d\P(\M)\E_{\varphi,-\M}\left[\prod_{i=1}^r\left(\exp\left(-e^{\gamma \ps{\bm c+\M,e_{i}}}\mathrm J_{i}\right)-1\right)\right]
	\end{align*}
	where $d\P(\M)$ is the density of the random variable $\M$ while with the notation $\E_{\varphi,-\M}$ we indicate that the process that enters the definition of the $(\mathrm{J}_{i})_{1\leq i\leq r}$ is started from $-\M$. Note that Theorem~\ref{thm:williams} and therefore such a decomposition can be adapted when we rather consider expressions of the form $\E_{\varphi}\left[\prod_{j=1}^k\left(\exp\left(-e^{\gamma \ps{\bm c,e_{i_k}}}I_{i_j}\right)-1\right)\right]$ since the $(e_{i_j})_{1\leq j\leq k}$ still form a root system in the Euclidean space $\text{span}(e_{i_j})_{1\leq j\leq k}$. The Weyl chamber would then be a subset of this vector space while the $(e_{i_j})_{1\leq j\leq k}$ components of the drifted Brownian motion would still form a drifted Brownian motion over it.
	
	Assuming for now Theorem~\ref{thm:expression_refl} to be true we see that in the above expression the terms that involve a product of more than two terms will be lower order terms in the asymptotic studied in Theorem~\ref{thm:primary_reflection}. More precisely
	\begin{equation*}
		\begin{split}
			\E_{\varphi}\left[\exp\left(-\sum_{i=1}^re^{\gamma \ps{\bm c,e_i}}I_i\right)\right]=1&+\sum_{i=1}^r\E_{\varphi}\left[\left(\exp\left(-e^{\gamma \ps{\bm c,e_{i}}}I_{i}\right)-1\right)\right]\\
			&+\E_{\varphi}\left[\left(\exp\left(-e^{\gamma \ps{\bm c,e_{1}}}I_{1}\right)-1\right)\left(\exp\left(-e^{\gamma \ps{\bm c,e_{2}}}I_{2}\right)-1\right)\right]+u
		\end{split}
	\end{equation*}
	where $u$ is as in the statement of Theorem~\ref{thm:primary_reflection}. 
	Using Theorem~\ref{thm:expression_refl} to evaluate the last expectation term we see that all we need to do is to provide a precise asymptotic expansion of the first sum. This issue is addressed thanks to the following:
	\begin{proposition}\label{prop:asymptot_1_term}
		Assume that $\ps{\alpha-Q,e_i^\vee}\in(-\gamma,0)$. Then for any positive $\eta<1$,
		\begin{equation}
			\E_{\varphi}\left[\exp\left(-e^{\gamma \ps{\bm c,e_{i}}}I_{i}\right)\right]=1+R_{s_i}\left(\alpha\right)e^{\ps{\hat s_i\alpha-\alpha,\bm c}}+\mathcal{O}\left(e^{(1-\eta)\gamma\ps{\bm c, e_i}}\right).
		\end{equation}
	\end{proposition} 
	\begin{proof}
		As explained above we can rewrite $I_i$ as $e^{\gamma \ps{\M,e_i}}\mathrm J_i$, and since only a one-dimensional Brownian motion is involved in the computation of $\E_{\varphi}\left[\exp\left(-e^{\gamma \ps{\bm c,e_i}}I_i\right)-1\right]$ we can readily use the one-dimensional path decomposition. Using this decomposition yields
		\begin{equation*}
			\begin{split}
				&\E_{\varphi}\left[\exp\left(-e^{\gamma \ps{\bm c,e_i}}I_i\right)-1\right] = \int_{0}^{+\infty}\ps{\nu,e_i^\vee}e^{\ps{\nu,e_i^\vee}M}\E_{\varphi,-M}\left[1-\exp\left(-e^{\gamma \ps{\bm c+M\omega_i^\vee,e_i}}\mathrm J_i\right)\right]dM
			\end{split}
		\end{equation*}
		since the random variable $M=\sup_{t\geq0}\ps{B^\nu,e_i}$ has law given by $d\P(M)=\ps{\nu,e_i^\vee}e^{\ps{\nu,e_i^\vee}M}\mathds 1_{M>0}$.
		By making the change of variable $M\leftrightarrow M+\ps{\bm c,e_i}$ we end up with
		\begin{equation*}\label{equ:starting_point_proof0}
			\begin{split}
				&\E_{\varphi}\left[\exp\left(-e^{\gamma \ps{\bm c,e_i}}I_i\right)-1\right] = e^{-\ps{\nu,e_i^\vee}{\ps{\bm c,e_i}}}R_{s_i}(\alpha;\bm c,\varphi)
			\end{split}
		\end{equation*}
		with 
		\begin{equation}\label{equ:approx_reflL}
			\begin{split}
				&R_{s_i}(\alpha;\bm c,\varphi)\coloneqq\int_{\ps{\bm c,e_i}}^{+\infty}\ps{\nu,e_i^\vee}e^{\ps{\nu,e_i^\vee}M}\E_{\varphi,-M+\ps{\bm c,e_i}}\left[1-\exp\left(-e^{\gamma M}\mathrm J_i\right)\right]dM.
			\end{split}
		\end{equation}
		To evaluate this term, let us split the integral in Equation~\eqref{equ:approx_reflL} as
		\[
		\int_{\ps{\bm c,e_i}}^{\ps{(1-\eta)\bm c,e_i}}+\int_{\ps{(1-\eta)\bm c,e_i}}^{+\infty}\ps{\nu,e_i^\vee}e^{\ps{\nu,e_i^\vee}M}\E_{\varphi,-M+\ps{\bm c,e_i}}\left[1-\exp\left(-e^{\gamma M}\mathrm J_i\right)\right]dM.
		\]
		In the first integral above we can bound the expectation term as
		\[
		\E_{\varphi,-M+\ps{\bm c,e_i}}\left[1-\exp\left(-e^{\gamma M}\mathrm J_i\right)\right]\leq e^{\gamma M}\E_{\varphi,-M+\ps{\bm c,e_i}}\left[\mathrm J_i\right].
		\]
		Note that thanks to Lemma~\ref{lemma:moments_cond} and because $\varphi$ is uniformly bounded over $\D$ the expectation on the right-hand side is uniformly bounded in $\varphi, M$. This implies that
		\begin{align*}
			&\int_{\ps{\bm c,e_i}}^{\ps{(1-\eta)\bm c,e_i}}\ps{\nu,e_i^\vee}e^{\ps{\nu,e_i^\vee}M}\E_{\varphi,-M+\ps{\bm c,e_i}}\left[1-\exp\left(-e^{\gamma M}\mathrm J_i\right)\right]dM\\
			\leq&\int_{\ps{\bm c,e_i}}^{\ps{(1-\eta)\bm c,e_i}}\ps{\nu,e_i^\vee}e^{\left(\gamma+\ps{\nu,e_i^\vee}\right)M}\E_{\varphi,-M+\ps{\bm c,e_i}}\left[\mathrm J_i\right]dM\\
			\leq &Ce^{\left(\gamma+\ps{\nu,e_i^\vee}\right)\ps{(1-\eta)\bm c,e_i}}
		\end{align*}
		for some positive constant $C$. 
		
		Let us now turn to the other integral. It is readily seen that since $-M+\ps{\bm c,e_i}$ will diverge to $-\infty$, the integral $\mathrm{J}_i$ will converge to
		\[
		\int_0^{\infty}e^{\gamma\ps{\mathcal{B}^\nu_t,e_i}}Z_t^idt\overset{\text{(law)}}{=}J_{\gamma_i}\left(\ps{\nu,e_i^*}\right)
		\]
		where $\gamma_i^2=\ps{e_i,e_i}\gamma^2$. This allows us to write $\E_{\varphi,-M+\ps{\bm c,e_i}}\left[1-\exp\left(-e^{\gamma M}\mathrm J_i\right)\right]$ under the form
		\[
		\E_{\varphi,-M+\ps{\bm c,e_i}}\left[1-\exp\left(-e^{\gamma M}\mathrm J_i\right)\right]=\E\left[1-\exp\left(-e^{\gamma M} J_{\gamma_i}\left(\ps{\nu,e_i^*}\right)\right)\right]+\mathfrak{R}(\varphi;M-\ps{\bm c,e_i}).
		\]
		Using the same estimates as above we see that 
		\begin{align*}
			&\int_{\ps{(1-\eta)\bm c,e_i}}^{+\infty}\ps{\nu,e_i}e^{\ps{\nu,e_i}M}\E_{\varphi,-M+\ps{\bm c,e_i}}\left[1-\exp\left(-e^{\gamma M}\mathrm J_i\right)\right]dM\\
			=&\int_{-\infty}^{+\infty}\ps{\nu,e_i^\vee}e^{\ps{\nu,e_i^\vee}M}\E\left[1-\exp\left(-e^{\gamma M} J_{\gamma_i}\left(\ps{\nu,e_i^*}\right)\right)\right]dM\\
			+&\int_{\ps{(1-\eta)\bm c,e_i}}^{+\infty}\ps{\nu,e_i^\vee}e^{\ps{\nu,e_i^\vee}M}\mathfrak{R}(\varphi;M-\ps{\bm c,e_i})dM+\mathcal{O}\left(e^{\left(\gamma+\ps{\nu,e_i^\vee}\right)\ps{(1-\eta)\bm c,e_i}}\right).
		\end{align*}
		The integral that appears in the second line can be evaluated and is found to be equal to 
		\[
		-\frac{\ps{\nu,e_i^\vee}}\gamma\Gamma\left(\frac{\ps{\nu,e_i^\vee}}\gamma\right)\displaystyle\E\left[J_{\gamma_i}\left(\ps{\nu,e_i^*}\right)^{-\frac{\ps{\nu,e_i^\vee}}\gamma}\right].
		\]
		Using Equation~\eqref{equ:eval_J} and noting that $\frac{\ps{\nu,e_i^\vee}}\gamma=2\frac{\ps{\nu,e_i^*}}{\gamma_i}$ we recover the expression~\eqref{eq:expression_refl_1} given for the reflection coefficient, provided that $\ps{\nu,e_i^*}\in(-\frac2{\gamma_i},0)$ (which does hold via our assumptions) and using the identity $z\Gamma(z)=\Gamma(1+z)$
		\[
		R_{s_i}(\alpha)=-\left(\pi\mu l\left(\frac{\gamma^2}{2}\right)\right)^{\frac{\ps{Q-\alpha,e_i^\vee}}\gamma}\frac{\Gamma\left(1+\frac{\ps{\alpha-Q,e_i^\vee}}{\gamma}\right)\Gamma\Big(1+\frac{\gamma}2\ps{\alpha-Q,e_i}\Big)}{\Gamma\left(1+\frac{\ps{Q-\alpha,e_i^\vee}}{\gamma}\right)\Gamma\Big(1+\frac\gamma2\ps{Q-\alpha,e_i}\Big)}\cdot
		\]
		
		As a consequence we need to focus on the integral of the remainder term. We will rely on the fact that since $M\geq \ps{(1-\eta)\bm c,e_i}$, the quantity $x\coloneqq -M+\ps{\bm c,e_i}$ (the starting point of the process $\mathcal B^\nu$) will diverge to $-\infty$. Now for any negative $x'>x$, we can rewrite $J_i$ under the form  
		\[
		\mathrm J_i=\int_{0}^{T_{x'}}+\int_{T_{x'}}^{\infty} e^{\gamma\mathcal B_t^{\ps{\nu,e_i}}}Z_t^idt
		\]
		where $T_{x'}$ is the first time when the process $\mathcal B^{\ps{\nu,e_i}}$ reaches $x'$. Then by the Markov property the process $\left(\mathcal B_t^{\ps{\nu,e_i}}\right)_{t\geq T_{x'}}$ has the law of $\mathcal B^\nu$ started from $x'$ and is independent of $\left(\mathcal B_t^{\ps{\nu,e_i}}\right)_{0\leq t\leq T_{x'}}$. Therefore $E_{\varphi,x'}\left[\exp\left(-e^{\gamma M}J_i\right)\right]-\E_{\varphi,x}\left[\exp\left(-e^{\gamma M}J_i\right)\right]$ is equal to
		\begin{align*}
			&\E_{\varphi,x}\left[\exp\left(-e^{\gamma M}\int_{T_{x'}}^{\infty} e^{\gamma\mathcal B_t^{\ps{\nu,e_i}}}Z_t^idt\right)\left(1-\exp\left(-e^{\gamma M}\int_{0}^{T_{x'}} e^{\gamma\mathcal B_t^{\ps{\nu,e_i}}}Z_t^idt\right)\right)\right]\\
			&\leq e^{\gamma M}\E_{\varphi,x}\left[\int_{0}^{T_{x'}} e^{\gamma\mathcal B_t^{\ps{\nu,e_i}}}Z_t^idt\exp\left(-e^{\gamma M}\int_{T_{x'}}^{\infty} e^{\gamma\mathcal B_t^{\ps{\nu,e_i}}}Z_t^idt\right)\right]\\
			&\leq Ce^{\gamma \left(M+x'\right)}
		\end{align*}
		where $C$ is uniformly bounded in $M,x',\varphi$ by Lemma~\ref{lemma:moments_Z}. Indeed the term that appears in the exponential can be bounded by $1$ while \[
		\E_{\varphi,x}\left[\int_{0}^{T_{x'}} e^{\gamma\mathcal B_t^{\ps{\nu,e_i}}}Z_t^idt\right]=e^{\gamma x'}\E_{\varphi,0}\left[\int_{-L_{x-x'}}^{0} e^{\gamma\mathcal B_{-t}^{\ps{\nu,e_i}}}Z_t^idt\right]
		\]
		where $L_{x-x'}$ is the last hitting-time of $x-x'$ (this follows from the time-reversal property of the process $\mathcal B^{\ps{\nu,e_i}}$, see~\cite[Theorem 2.5]{Williams}). 
		By letting $x\to-\infty$, we end up with
		\[
		\norm{\mathfrak{R}(\varphi; M-\ps{\bm c,e_i})}\leq  C e^{\gamma\ps{\bm c, e_i}}
		\]
		provided that $\lim\limits_{x\to-\infty} \E_{\varphi,x}\left[\exp\left(-e^{\gamma M}\mathrm J_i\right)\right]=\expect{\exp\left(-e^{\gamma M}J_{\gamma_i}\left(\frac{\ps{\alpha-Q,e_i^\vee}}{2\gamma}\right)\right)}$. This follows from the proof of~\cite[Lemma 7.1]{KRV_DOZZ}, by using the fact that $\varphi(0)=0$. All together this implies that the integral of this remainder term can be bounded by
		\[
		C'e^{\left(\gamma+(1-\eta)\ps{\nu,e_i^\vee}\right)\ps{\bm c, e_i}}.
		\]
		Finally we have proved the desired result: for any positive $\eta$,
		\begin{equation*}
			R_{s_i}(\alpha;\bm c,\varphi)=R_L\left(\ps{\alpha,e_i}\right)+\mathcal{O}\left(e^{(1-\eta)\left(\gamma+\ps{\nu,e_i^\vee}\right)\ps{\bm c, e_i}}\right).
		\end{equation*}
	\end{proof}
	We are now in position to wrap up the proof of Theorem~\ref{thm:primary_reflection}. Namely because the assumptions of Theorem~\ref{thm:primary_reflection} require that $\ps{\hat s_1\hat s_2\alpha-\alpha-(1-\eta)\gamma e_i,\bm c}\rightarrow+\infty$ for some positive $\eta$ the statement of Proposition~\ref{prop:asymptot_1_term} shows that in this asymptotic
	\[
	\E_{\varphi}\left[\left(\exp\left(-e^{\gamma \ps{\bm c,e_{i}}}I_{i}\right)-1\right)\right]=R_{s_i}(\alpha)e^{\ps{\hat s_i\alpha-\alpha,\bm c}}+o\left(e^{\ps{\hat s_1\hat s_2\alpha-\alpha,\bm c}}\right).
	\]
	Therefore we can write that 
	\begin{align*}
		\E_{\varphi}\left[\exp\left(-\sum_{i=1}^re^{\gamma \ps{\bm c,e_i}}I_i\right)\right]=1+\sum_{i=1}^rR_{s_i}(\alpha)e^{\ps{\hat s_i\alpha-\alpha,\bm c}}+R_{s_1s_2}(\alpha)e^{\ps{\hat s_1\hat s_2\alpha-\alpha,\bm c}}+o\left(e^{\ps{\hat s_1\hat s_2\alpha-\alpha,\bm c}}\right).
	\end{align*}
	This concludes the proof of Theorem~\ref{thm:primary_reflection}, provided that Theorem~\ref{thm:expression_refl} does indeed hold.
	
	\subsection{Probabilistic Toda reflection coefficients: proof of Theorem~\ref{thm:expression_refl}}\label{subsec:all_refl}
	Building on the reasoning developed above, we now turn to the proof of Theorem~\ref{thm:expression_refl} for the general case where the element of the Weyl group being considered is not necessarily assumed to be an elementary reflection. 
	Without loss of generality we assume that $s=s_1\cdots s_r$ and $\ps{\hat s\alpha-\hat s'\alpha,\bm c}\to+\infty$ for all $s'\neq s\in W_{1\cdots,r}$. Recall that we investigate the asymptotics of
	\begin{align*}
		\E_{\varphi}\left[\prod_{k=1}^p\left(\exp\left(-e^{\gamma \ps{\bm c,e_{k}}}I_{k}\right)-1\right)\right]
	\end{align*}
	in the limit where $\bm c\to\infty$ inside $\mathcal{C}$.
	
	\subsubsection{The case of length $2$}
	As a warm-up we first consider the case where $s$ has length $2$, that is to say only two integrals are involved.  Then we can work in $\V\simeq\R^2$ in which lie the two roots $(e_1,e_2)$ (up to reordering the roots we assume that $(i,j)=(1,2)$ to simplify the notations). Since $\ps{e_1,e_2}=0$ implies that the components $e_1$, $e_2$ of $B^\nu$ are independent so the expectations factorize, this case reduces to the case of length one. Therefore we may assume that it is not the case. In doing so we can write
	\begin{equation*}
		\begin{split}
			&\E_{\varphi}\left[\left(\exp\left(-e^{\gamma \ps{\bm c,e_1}}I_1\right)-1\right)\left(\exp\left(-e^{\gamma \ps{\bm c,e_2}}I_2\right)-1\right)\right] \\
			&= \int_{\mathcal{C}}d\P(\M)\E_{\varphi,-\M}\left[\left(\exp\left(-e^{\gamma \ps{\bm c+\M,e_1}}\mathrm J_1\right)-1\right)\left(\exp\left(-e^{\gamma \ps{\bm c+\M,e_2}}\mathrm J_2\right)-1\right)\right].
		\end{split}
	\end{equation*}
	In the above equation $d\P(\M)$ is the density of the random variable $\M$, given by
	\begin{equation*}
		d\P(\M)=\partial_{1,2}h(-\M)d\M=\sum_{s\in  W}\epsilon(s)\ps{\hat s\alpha-\alpha,\omega_1^\vee}\ps{\hat s\alpha-\alpha,\omega_2^\vee}e^{\ps{\alpha-\hat s\alpha,\M}}d\M.
	\end{equation*}
	Here the coefficients $\ps{\hat s\alpha-\alpha,\omega_1^\vee}\ps{\hat s\alpha-\alpha,\omega_2^\vee}$ vanish when $s\in \{Id,s_1,s_2\}$, so the sum actually ranges over $s\in W_{1,2}$. Hence a change in variable $\M\leftrightarrow\M-\bm c$ yields
	\begin{equation*}
		\E_{\varphi}\left[\left(\exp\left(-e^{\gamma \ps{\bm c,e_1}}I_1\right)-1\right)\left(\exp\left(-e^{\gamma \ps{\bm c,e_1}}I_1\right)-1\right)\right] = \sum_{ s\in W_{1,2}}e^{\ps{\hat s\alpha-\alpha,\bm c}}R_s(\alpha;\bm c,\varphi)
	\end{equation*}
	where we have introduced 
	\begin{equation}\label{equ:approx_refl}
		\begin{split}
			R_s(\alpha;\bm c,\varphi)\coloneqq \epsilon(s)&\int_{\mathcal{C}+\bm c}\ps{\hat s\alpha-\alpha,\omega_1^\vee}\ps{\hat s\alpha-\alpha,\omega_2^\vee}e^{\ps{\alpha-\hat s\alpha,\M}}\\
			&\E_{\varphi,c-\M}\left[\left(\exp\left(-e^{\gamma \ps{\M,e_1}}\mathrm J_1\right)-1\right)\left(\exp\left(-e^{\gamma \ps{\M,e_2}}\mathrm J_2\right)-1\right)\right]d\M.
		\end{split}
	\end{equation}
	However and unlike the case where only one integral was involved, there is one subtle issue that needs to be taken care of here. Indeed in the asymptotic where the starting point of the process $\B^\nu$ from Corollary~\ref{cor:brownien_drift} will diverge, it is far from clear how the quantities $\mathrm J_i$ will behave. In Section~\ref{section:brown_inf} we have provided a refined study of the process $\B^{\nu}$ in this regime, and a consequence of Proposition~\ref{prop:asym_prob} is in particular that as $\bm x\rightarrow\infty$ inside $\mathcal C_-$ with $\ps{s_1s_2(\alpha-Q)-s_2s_1(\alpha-Q),\bm x}\to+\infty$,
	\begin{equation}
		\lim\limits_{\bm x\to\infty} \E_{\bm x}\left[F(\mathrm J_1)G(\mathrm J_2)\right]=\expect{F\Big(J_{\gamma_1}\left(\ps{s_2(\alpha-Q),e_1^*}\right)\Big)}\expect{G\Big(J_{\gamma_2}\left(\ps{\alpha-Q,e_2^*}\right)\Big)}
	\end{equation} for $F,G$ be bounded continuous over $\R^+$ (the reasoning conducted in Section~\ref{section:brown_inf} still works for $Z$ as considered here).
	This allows to provide the desired result:
	\begin{lemma}\label{lemma:conv_refl}
		Under the assumptions of Theorem~\ref{thm:expression_refl} $R_s(\alpha;\bm c,\varphi)$ converges to a well-defined limit. When $s=s_1s_2$ this limit is equal to $R_s(\alpha)$.
	\end{lemma}
	\begin{proof}
		Let us pick some positive $\eta$ and split the integral between the domains $\mathcal{C}+(1-\eta)\bm c$ and $\mathcal{C}_\eta\coloneqq\mathcal{C}+\bm c\setminus\left(\mathcal{C}_\eta+(1-\eta)\bm c\right)$.
		
		The integral over the domain $\mathcal{C}_\eta$ will be negligible. Indeed for any $\M\in\mathcal{C}_\eta$ there is some $i\in\{1,2\}$ such that $\ps{\M,e_i}<(1-\eta)\ps{\bm c,e_i}$ (say $i=1$); as a consequence over this domain the expectation term is  of order at most $e^{\gamma (1-\eta)\ps{\bm c,e_i}}$:
		\begin{align*}
			&\E_{\varphi,c-\M}\left[\left(\exp\left(-e^{\gamma \ps{\M,e_1}}\mathrm J_1\right)-1\right)\left(\exp\left(-e^{\gamma \ps{\M,e_2}}\mathrm J_2\right)-1\right)\right]\\
			&\leq \gamma e^{\gamma (1-\eta)\ps{\bm c,e_1}}\E_{\varphi,c-\M}\left[\mathrm J_1\exp\left(-e^{\gamma (1-\eta)\ps{\bm c,e_1}}\mathrm J_1\right)\left(1-\exp\left(-e^{\gamma \ps{\M,e_2}}\mathrm J_2\right)\right)\right]\\
			&\leq \gamma e^{\gamma (1-\eta)\ps{\bm c,e_1}}\E_{\varphi,c-\M}\left[\mathrm J_1\left(1-\exp\left(-e^{\gamma \ps{\M,e_2}}\mathrm J_2\right)\right)\right].
		\end{align*} 
		This implies that
		\begin{align*}
			&\int_{\ps{\bm c,e_1}<\ps{\M,e_1}<(1-\eta)\ps{\bm c,e_1}}e^{\ps{\alpha-\hat s\alpha,\M}}\E_{\varphi,c-\M}\left[\left(\exp\left(-e^{\gamma \ps{\M,e_1}}\mathrm J_1\right)-1\right)\left(\exp\left(-e^{\gamma \ps{\M,e_2}}\mathrm J_2\right)-1\right)\right]d\M\\
			&\leq \gamma e^{\gamma (1-\eta)\ps{\bm c,e_1}}\int_{\ps{\bm c,e_1}<\ps{\M,e_1}<(1-\eta)\ps{\bm c,e_1}}e^{\ps{\alpha-\hat s\alpha,\M}}\E_{\varphi,c-\M}\left[\mathrm J_1\left(1-\exp\left(-e^{\gamma \ps{\M,e_2}}\mathrm J_2\right)\right)\right]d\M\\
			&\leq C e^{\gamma (1-\eta)\ps{\bm c,e_1}}e^{\ps{\bm c,\alpha-\hat s\alpha}}
		\end{align*}
		for some positive constant $C$. Choosing $\eta$ small enough this term vanishes as $\ps{\bm c,e_1}$ and $\ps{\bm c,e_2}$ go to $-\infty$ under the assumptions of Theorem~\ref{thm:primary_reflection}. Therefore the integral over the whole domain $\mathcal{C}_\eta$ becomes negligible in the limit.
		
		Therefore the integral over the other domain $\mathcal{C}+(1-\eta)\bm c$ will be the contributing one. Indeed this set has been defined so that inside $\mathcal{C}+(1-\eta)\bm c$, the starting point of the process will diverge to $\infty$ inside $\mathcal{C}_-$, so that we are in the setting of Proposition~\ref{prop:asymptot_A2}:
		\begin{align*}
			&\E_{\varphi,c-\M}\left[\left(\exp\left(-e^{\gamma \ps{\M,e_1}}\mathrm J_1\right)-1\right)\left(\exp\left(-e^{\gamma \ps{\M,e_2}}\mathrm J_2\right)-1\right)\right]\text{ will converge pointwise to}\\
			&\expect{\exp\left(-e^{\gamma \ps{\M,e_1}}J_{\gamma_1}\left(\ps{s_2(\alpha-Q),e_1^*}\right)\right)-1}\expect{\exp\left(-e^{\gamma \ps{\M,e_2}}J_{\gamma_2}\left(\ps{\alpha-Q,e_2^*}\right)\right)-1}.
		\end{align*}
		By dominated convergence this implies that 
		\begin{align*}
			&\int_{\mathcal{C}+(1-\eta)\bm c}e^{\ps{\alpha-\hat s\alpha,\M}}\E_{\varphi,c-\M}\left[\left(\exp\left(-e^{\gamma \ps{\M,e_1}}\mathrm J_1\right)-1\right)\left(\exp\left(-e^{\gamma \ps{\M,e_2}}\mathrm J_2\right)-1\right)\right]d\M
		\end{align*}
		converges to  the quantity 
		\begin{align*}
			\int_{\V}&e^{\ps{\alpha-\hat s\alpha,\omega_1^\vee}\ps{\M,e_1}}\E\left[\left(\exp\left(-e^{\gamma \ps{\M,e_1}}J_{\gamma_1}\left(\ps{s_2(\alpha-Q),e_1^*}\right)\right)-1\right)\right]\times\\
			&e^{\ps{\alpha-\hat s\alpha,\omega_2^\vee}\ps{\M,e_2}}\expect{\left(\exp\left(-e^{\gamma \ps{\M,e_2}}J_{\gamma_2}\left(\ps{\alpha-Q,e_2^*}\right)\right)-1\right)}d\M.
		\end{align*}
		The latter is nothing but
		\begin{align*}
			&\frac1\gamma\Gamma\left(\frac1\gamma{\ps{\alpha-\hat s\alpha,\omega_1^\vee}}\right)\expect{J_{\gamma_1}\left(\ps{s_2(\alpha-Q),e_1^*}\right)^{\frac{\ps{\hat s\alpha-\alpha,\omega_1^\vee}}{\gamma}}}\times\\ &\frac1\gamma\Gamma\left(\frac1\gamma{\ps{\alpha-\hat s\alpha,\omega_2^\vee}}\right)\expect{J_{\gamma_2}\left(\ps{\alpha-Q,e_2^*}\right)^{\frac{\ps{\hat s\alpha-\alpha,\omega_2^\vee}}{\gamma}}}.
		\end{align*}
		Now if we choose $s=s_1s_2$, we can write that $\frac{\ps{s_2(\alpha-Q),e_1^*}}{\gamma_1}=\frac{\ps{s_2(\alpha-Q),e_1^\vee}}{2\gamma}=\frac{\ps{\alpha-\hat s\alpha,\omega_1^\vee}}{2\gamma}$ as well as $\frac{\ps{\alpha-Q,e_2^*}}{\gamma_2}=\frac{\ps{\alpha-Q,e_2^\vee}}{2\gamma}=\frac{\ps{\alpha-\hat s\alpha,\omega_2^\vee}}{2\gamma}$.
		As a consequence we can evaluate the above quantity thanks to Equation~\eqref{equ:eval_J}:
		\begin{align*}
			&\left(\pi l\left(\frac{\gamma_1^2}{4}\right)\right)^{\frac{\ps{\hat s\alpha-\alpha,\omega_1^\vee}}\gamma}\frac{1}{\gamma}\frac{\Gamma\left(\frac1\gamma{\ps{\alpha-\hat s\alpha,\omega_1^\vee}}\right)\Gamma\left(1+\frac\gamma2{\ps{\alpha-\hat s\alpha,\omega_1^\vee}}\right)}{\Gamma\left(1+\frac1\gamma{\ps{\hat s\alpha-\alpha,\omega_1^\vee}}\right)\Gamma\left(1+\frac\gamma2{\ps{\hat s\alpha-\alpha,\omega_1^\vee}}\right)}\times\\
			&\left(\pi l\left(\frac{\gamma_2^2}{4}\right)\right)^{\frac{\ps{\hat s\alpha-\alpha,\omega_2^\vee}}\gamma}
			\frac{1}{\gamma}\frac{\Gamma\left(\frac1\gamma{\ps{\alpha-\hat s\alpha,\omega_2^\vee}}\right)\Gamma\left(1+\frac\gamma2{\ps{\alpha-\hat s\alpha,\omega_2^\vee}}\right)}{\Gamma\left(1+\frac1\gamma{\ps{\hat s\alpha-\alpha,\omega_2^\vee}}\right)\Gamma\left(1+\frac\gamma2{\ps{\hat s\alpha-\alpha,\omega_2^\vee}}\right)}\cdot
		\end{align*}
		Therefore, collecting up terms, we see that
		\begin{align*}
			&\int_{\mathcal{C}_\eta}+\int_{\mathcal{C}+(1-\eta)\bm c}e^{\ps{\alpha-\hat s\alpha,\M}}\E_{\varphi,c-\M}\left[\left(\exp\left(-e^{\gamma \ps{\M,e_1}}J_1\right)-1\right)\left(\exp\left(-e^{\gamma \ps{\M,e_2}}J_2\right)-1\right)\right]d\M
		\end{align*}
		does converge, and for $s=s_1s_2$ the limit of $R_s(\alpha;\bm c,\varphi)$ will be given by $R_s(\alpha)$.
	\end{proof}
	This allows to conclude for the proof of Theorem~\ref{thm:expression_refl} in the case where $p=2$.
	\subsubsection{The general case}
	Without loss of generality we assume that $s=s_1\cdots s_r$ and\\ $\ps{\hat s\alpha-\hat s'\alpha,\bm c}\to+\infty$ for all $s'\neq s\in W_{1\cdots,r}$. Then along the same lines as above we get
	\begin{align*}
		&\E_{\varphi}\left[\prod_{k=1}^r\left(\exp\left(-e^{\gamma \ps{\bm c,e_{i_k}}}I_{i_k}\right)-1\right)\right]\\
		&=\sum_{s'}e^{\ps{\hat s'\alpha-\alpha,\bm c}}\epsilon(s')\int_{\mathcal{C}+\bm c}\prod_{i=1}^r\ps{\hat s'\alpha-\alpha,\omega_i^\vee}e^{\ps{\alpha-\hat s'\alpha,\M}}\E_{\varphi,\bm c-\M}\left[\prod_{k=1}^p\left(\exp\left(-e^{\gamma \ps{\bm c,e_{i_k}}}I_{i_k}\right)-1\right)\right]
	\end{align*}
	where the sum ranges over the elements $s'\in W$ whose reduced expression contain $s_1,\cdots,s_r$, that is $W_{1,\cdots, r}$. Indeed we can write $\ps{\hat s'\alpha-\alpha,\omega_i^\vee}=\ps{\alpha-Q,s'^{-1}\omega_i^\vee-\omega_i^\vee}$ where $ s'^{-1}\omega_i^\vee-\omega_i^\vee=0$ if this is not the case (indeed $s'_i\omega_j^\vee=0$ for $j\neq i$). Using the same estimates as before, we see that up to a term which vanishes in the limit, the integral can be reduced to an integral over $\mathcal{C}+(1-\eta)\bm c$ where $\eta>0$ is small enough. Now recall from Section~\ref{section:brown_inf} that 
	in the limit where $\bm x\rightarrow \infty$ inside $\mathcal{C}$ with $\ps{\hat s\alpha-\hat s'\alpha,\bm x}\to+\infty$ for all $s'\neq s\in W_{1,\cdots,r}$,
	\begin{equation}\label{eq:asym_deco_bis}
		\E_{\bm x}\left[\prod_{i=1}^r F_i(\mathrm J_i)\right]\text{ converges to }\prod_{i=1}^r\expect{F_i\Big(J_{\gamma_i}\left(\ps{s_{i+1}\cdots s_r(\alpha-Q),e_i^*}\right)\Big)}
	\end{equation}
	for $F_1,\cdots, F_r$ bounded continuous over $\R^+$.
	With Equation~\eqref{eq:asym_deco_bis} at hand we see that
	\begin{align*}
		&\lim\limits_{\bm c\to\infty}\int_{\mathcal{C}+\bm c}\prod_{i=1}^r\ps{\hat s'\alpha-\alpha,\omega_i^\vee}e^{\ps{\alpha-\hat s'\alpha,\M}}\E_{\varphi,\bm c-\M}\left[\prod_{k=1}^p\left(\exp\left(-e^{\gamma \ps{\bm c,e_{i_k}}}I_{i_k}\right)-1\right)\right]\\
		&=\prod_{i=1}^r\Gamma\left(1+\frac1\gamma\ps{\hat s'\alpha-\alpha,\omega_i^\vee}\right)\expect{J_{\gamma_i}\left(\ps{s_{i+1}\cdots s_r(\alpha-Q),e_i^*}\right)^{\frac1\gamma\ps{\hat s'\alpha-\alpha,\omega_i^\vee}}}.
	\end{align*}
	Now in the case where $s'=s$, we can use the fact that $s_j\omega_i^\vee=\omega_i^\vee$ for all $j\neq i$ while $s_i\omega_i^\vee=\omega_i^\vee - e_i^\vee$ to see that the exponent is actually equal to
	\begin{align*}
		\frac1\gamma\ps{s\alpha-\alpha,\omega_i^\vee}=\frac1\gamma\ps{\alpha-Q,s_r\cdots s_1\omega_i^\vee-\omega_i^\vee}=\frac1\gamma\ps{\alpha-Q,s_r\cdots s_{i+1}e_i^\vee}=\frac1{\gamma_i}\ps{s_{i+1}\cdots s_r(\alpha-Q),e_i^*}.
	\end{align*}
	This allows to evaluate the result in that case via Equation~\eqref{equ:eval_J}:
	\[
	\prod_{i=1}^r\left(\pi l\left(\frac{\gamma_i^2}{4}\right)\right)^{\frac{\ps{\hat s\alpha-\alpha,\omega_i^\vee}}\gamma}
	\frac{\Gamma\left(1+\frac1\gamma\ps{\hat s\alpha-\alpha,\omega_i^\vee}\right)\Gamma\left(1+\frac\gamma2\ps{\hat s\alpha-\alpha,\omega_i}\right)}{\Gamma\left(1+\frac1\gamma\ps{\alpha-\hat s\alpha,\omega_i^\vee}\right)\Gamma\left(1+\frac\gamma2\ps{\alpha-\hat s\alpha,\omega_i}\right)}\cdot
	\]
	Moreover under our assumptions we know that the higher order term in the expansion of
	\[\sum_{s'\in W_{1,\cdots,r}}e^{\ps{\hat s'\alpha-\alpha,\bm c}}\epsilon(s')\int_{\mathcal{C}+\bm c}\prod_{i=1}^r\ps{\hat s'\alpha-\alpha,\omega_i^\vee}e^{\ps{\alpha-\hat s'\alpha,\M}}\E_{\varphi,\bm c-\M}\left[\prod_{k=1}^p\left(\exp\left(-e^{\gamma \ps{\bm c,e_{i_k}}}I_{i_k}\right)-1\right)\right]
	\] will correspond to $s'=s$. The proof of Theorem~\ref{thm:expression_refl} is thus complete.

	\subsection{Proof of Theorems~\ref{thm:whittaker_expansion} and~\ref{thm:tail_expansion_refl}}
	The proof of the tail expansion of the GMC measures follows the very same lines as the reasoning presented above so we will be brief. One only needs to replace the terms of the form $\E_{\varphi}\left[\prod_{i=1}^p\left(\exp\left(-e^{\gamma \ps{\bm c,e_{i}}}I_{i}\right)-1\right)\right]$ by $\P\left(I_i>e^{-\gamma \ps{\bm c,e_i}}, 1\leq i\leq p\right)$.
	By doing so we will end up with expression of the type
	\begin{align*}
		\P\left(I_i>e^{-\gamma \ps{\bm c,e_i}},1\leq i\leq p\right))&=\int_{\mathcal{C}}d\P(\M)\P_{-\M}\left(\mathrm J_i>e^{-\gamma \ps{\bm c+\M,e_i}}, 1\leq i\leq p\right)\\
		&=\sum_{s\in W_{1,\cdots,p}}e^{\ps{\hat s\alpha-\alpha,\bm c}}\overline{R_s}(\alpha;\bm c,\varphi)+u
	\end{align*}
	with $u$ a lower order term and $\overline{R_s}(\alpha;\bm c,\varphi)$ defined as
	\[
	\epsilon(s)\int_{\mathcal{C}+\bm c}\prod_{i=1}^p\ps{\hat s\alpha-\alpha,\omega_i^\vee}e^{\ps{\alpha-\hat s\alpha,\M}}\P_{c-\M}\left(\mathrm J_i>e^{-\gamma \ps{\M,e_i}},1\leq i\leq p\right)d\M.
	\]
	Thanks to Proposition~\ref{prop:asymptot_J} above we can prove in the same way as before that $\overline{R_s}(\alpha;\bm c,\varphi)$ converges and its limit is given by $\overline{R_s}(\alpha)$ when $s$ is as in the statement of Theorem~\ref{thm:tail_expansion_refl}. This wraps up the proof of Theorem~\ref{thm:tail_expansion_refl}.

	The proof proceeds in the very same way for Whittaker functions. We work with
	\begin{equation}
		I_i\coloneqq \int_0^{\infty}e^{-\ps{B^\mu_t, e_i}}dt\quad\text{and}\quad J(\nu)\coloneqq \int_0^{\infty}e^{-\mathcal B^\nu_t}dt.
	\end{equation}
	In analogy with Toda Vertex Operators, we also introduce the renormalized quantity
	\begin{equation}
		\phi_\mu(\bm x)\coloneqq \expect{\exp\left(-\sum_{i=1}^re^{-\ps{\bm x,e_i}}\int_0^{\infty}e^{-\ps{B^\mu_t, e_i}}dt\right)}.
	\end{equation}
	Then the statement of Theorem~\ref{thm:whittaker_expansion} boils down to proving that $\phi_\mu(\bm x)$ has the asymptotic 
	\[
	\phi_\mu(\bm x)=1+\sum_{i=1}^r \frac{b'(s_i\mu)}{b'(\mu)}e^{\ps{s_i\mu-\mu,\bm x}} + \frac{b'(s_1s_2\mu)}{b'(\mu)}e^{\ps{s_1s_2\mu-\mu,\bm x}} + l.o.t.
	\]
	where $l.o.t.$ stands for lower order terms in the asymptotic studied, and
	\[
	b'(\mu)\coloneqq \prod_{e\in\Phi^+}\left(\frac2{\ps{e,e}}\right)^{-\frac{\ps{\mu,e^\vee}}2}\Gamma\left(\ps{\mu,e^\vee}\right).
	\]
	
	To prove this, we decompose $\phi_\mu$ under the form 
	\begin{align*}
		\phi_\mu(\bm x)=1&+\sum_{i=1}^r\expect{\exp\left(-e^{-\ps{\bm x,e_i}}I_i\right)-1}\\
		&+\sum_{i<j}\expect{\left(\exp\left(-e^{-\ps{\bm x,e_i}}I_i\right)-1\right)\left(\exp\left(-e^{-\ps{\bm x,e_j}}I_j\right)-1\right)}+l.o.t.
	\end{align*}
	where $l.o.t.$ is negligible from the proof of Theorem~\ref{thm:primary_reflection}.
	
	Along the same lines as above we see that 
	\[
	\expect{\exp\left(-e^{-\ps{\bm x,e_i}}I_i\right)-1}=e^{\ps{s_i\mu-\mu,\bm x}}r'_s(\mu)+u
	\]
	where $r_{s_i}(\mu)$ has the probabilistic expression 
	\begin{equation}
		r'_{s_i}(\mu)=-\ps{e_i,e_i}^{-\ps{\mu,e_i^\vee}}\Gamma(1-\ps{\mu,e_i^\vee})\expect{J\left(\frac12\ps{\mu,e_i^\vee}\right)^{\ps{\mu,e_i^\vee}}},
	\end{equation}
	and $u$ is a $\mathcal{O}\left(e^{(1-\eta)\left(1+\ps{\mu,e_i}\right)\ps{\bm x, e_i}}\right)$, via an analog of Proposition~\ref{prop:asymptot_1_term}.
	The expectation term in the right-hand side has been evaluated in Proposition~\ref{prop:eval_J} where it was found to be equal to 
	\begin{equation}
		\frac{2^{\ps{\mu,e_i^\vee}}}{\Gamma(1+\ps{\mu,e_i^\vee})}\cdot
	\end{equation}
	This implies that the coefficients $r'_{s_i}(\mu)$ are given by
	\[
	r'_{s_i}(\mu)=-\left(\frac2{\ps{e_i,e_i}}\right)^{\ps{\mu,e_i^\vee}}\frac{\Gamma(1-\ps{\mu,e_i^\vee})}{\Gamma(1+\ps{\mu,e_i^\vee})}=\left(\frac2{\ps{e_i,e_i}}\right)^{\ps{\mu,e_i^\vee}}\frac{\Gamma(-\ps{\mu,e_i^\vee})}{\Gamma(\ps{\mu,e_i^\vee})},
	\]
	which does coincides with the proposed expression
	\begin{equation}
		r'_s(\mu)=\prod_{e\in\Phi^+}\frac{\left(\frac2{\ps{e,e}}\right)^{-\frac{\ps{s\mu,e^\vee}}2}\Gamma(\ps{s\mu,e^\vee})}{\left(\frac2{\ps{e,e}}\right)^{-\frac{\ps{\mu,e^\vee}}2}\Gamma(\ps{\mu,e^\vee})}=\frac{b'(s\mu)}{b'(\mu)}
	\end{equation}
	when $s=s_i$ has length $1$.
	
	Likewise and relying on the path decomposition of Theorem~\ref{thm:williams}, the proof of Theorem~\ref{thm:primary_reflection} shows that the expectation terms that involve two distinct reflections admit the expansion 
	\[
	\expect{\left(\exp\left(-e^{-\ps{\bm x,e_i}}I_i\right)-1\right)\left(\exp\left(-e^{-\ps{\bm x,e_j}}I_j\right)-1\right)}=e^{\ps{s_is_j\mu-\mu,\bm x}}r'_{s_is_j}(\mu)+l.o.t.
	\]
	as soon as $\ps{s_is_j\mu-s_js_i\mu,\bm x}\to\infty$, and where $r'_{s_is_j}(\mu)$ has the probabilistic expression
	\begin{equation}
		r'_{s_is_j}(\mu)=\Gamma(1-2\ps{s_j\mu,e_i^\vee})\Gamma(1-2\ps{\mu,e_j^\vee})\expect{J\left(\frac12\ps{s_j\mu,e_i^\vee}\right)^{2\ps{s_j\mu,e_i^\vee}}}\expect{J\left(\frac12\ps{\mu,e_j^\vee}\right)^{2\ps{\mu,e_j^\vee}}}.
	\end{equation}
	The latter is nothing but $r'_{s_i}(s_j\mu)r'_{s_j}(\mu)$, so that the reflection coefficients $r'_{s}(\mu)$ are also found to be equal to
	\begin{equation}
		r'_s(\mu)=\prod_{e\in\Phi^+}\frac{\left(\frac2{\ps{e,e}}\right)^{-\frac{\ps{s\mu,e^\vee}}2}\Gamma(\ps{s\mu,e^\vee})}{\left(\frac2{\ps{e,e}}\right)^{-\frac{\ps{\mu,e^\vee}}2}\Gamma(\ps{\mu,e^\vee})}=\frac{b'(s\mu)}{b'(\mu)}
	\end{equation}
	for $s$ with length $2$.
	
	More generally, we can consider the reflection coefficients defined by setting, for $s=s_{i_1}\cdots s_{i_p}$ with $i_1,\cdots,i_p$ distinct, 
	\begin{equation}
		r'_s(\mu)\coloneqq \lim\limits_{\bm x\to\infty} e^{-\ps{s\mu,\bm x}}\expect{\prod_{k=1}^p\left(\exp\left(-e^{-\ps{\bm x,e_{i_k}}}I_{i_k}\right)-1\right)}
	\end{equation}
	where the limit is taken in the regime where
	\[
	\ps{s\mu-s'\mu,\bm x}\to+\infty\quad\text{for all }s'\neq s\in W_{i_1,\cdots,i_p}.
	\]
	The same reasoning as the one conducted in Subsection~\ref{subsec:all_refl} shows that these can be evaluated as
	\begin{equation}
		r'_s(\mu)=\frac{b'(s\mu)}{b'(\mu)}
	\end{equation}
	for such $s$, based on the property that $r'_{s\tau}(\mu)=r'_s(\tau\mu)r'_\tau(\mu)$, which is a consequence of the form of the process $\mathcal{B}^\nu$ seen from infinity (see Section~\ref{section:brown_inf}).
	
	This establishes the validity of Theorem~\ref{thm:whittaker_expansion} and concludes for the present document.
	


	
	\subsection*{Data Availability Statement}
	Data sharing is not applicable to this article as no datasets were generated or analysed.

	\bibliography{biblio}
	\bibliographystyle{plain}

\end{document}